\documentclass[11pt,a4paper]{article}
\usepackage{amssymb,amsmath}
\usepackage{times,theorem,latexsym,color}

\newcommand{\BOX}{\ensuremath\Box}

\newtheorem{theorem}{Theorem }[section]

{\theorembodyfont{\rmfamily}}
{\theorembodyfont{\rmfamily}}
{\theorembodyfont{\rmfamily}}
\newtheorem{lemma}[theorem]{Lemma}

{\theorembodyfont{\rmfamily}\newtheorem{remark}[theorem]{Remark}}
{\theorembodyfont{\rmfamily}}

\newcommand{\R}{\mathbb{R}}

\newcommand{\dd}{\,{\rm d}}

\newenvironment{proof}{{\vskip\baselineskip\noindent\textbf{Proof:}}}%
{\hspace*{.1pt}\hspace*{\fill}\BOX\vskip\baselineskip}

\newenvironment{proofx}[1]%
{\vskip\baselineskip\noindent\textbf{Proof of {#1}:}}%
{\hspace*{.1pt}\hspace*{\fill}\BOX\vskip\baselineskip}
{\vskip\baselineskip\noindent\textbf{Proof of Theorem \protect\ref{#1}:}}%
{\hspace*{.1pt}\hspace*{\fill}\BOX\vskip\baselineskip}
{\vskip\baselineskip\noindent\textbf{Proof of Theorems \protect\ref{#1} --
\protect\ref{#2}:}}%
{\hspace*{.1pt}\hspace*{\fill}\BOX\vskip\baselineskip}

\begin{document}

\title{On stationary Navier-Stokes flows around a rotating obstacle in two-dimensions}
\author{Mitsuo Higaki\thanks{Department of Mathematics, Kyoto University, Japan. E-mail: \texttt{mhigaki@math.kyoto-u.ac.jp}} \and Yasunori Maekawa\thanks{Department of Mathematics, Kyoto University, Japan. E-mail: \texttt{maekawa@math.kyoto-u.ac.jp}} \and Yuu Nakahara\thanks{Mathematical Institute, Tohoku University, Japan. E-mail: \texttt{yuu.nakahara.t3@dc.tohoku.ac.jp}}}

\date{}

\maketitle

\noindent {\bf Abstract} 
We study the two-dimensional stationary Navier-Stokes equations 
describing the flows around a rotating obstacle. 
The unique existence of solutions and their asymptotic behavior at spatial infinity are established
when the rotation speed of the obstacle and the given exterior force are sufficiently small. 

\vspace{0.3cm}

\noindent {\bf Keywords}\, Navier-Stokes equations $\cdot$ two-dimensional exterior flows $\cdot$ scale-criticality $\cdot$ flows around a rotating obstacle

\vspace{0.3cm}

\noindent {\bf Mathematics Subject Classification (2000)}\, 35B35 $\cdot$ 35Q30 $\cdot$ 76D05


\section{Introduction}\label{intro}

In this paper we consider the two-dimensional Navier-Stokes equations for viscous incompressible flows
around a rotating obstacle in two-dimensions:
\begin{equation}\label{NS}
  \left\{
\begin{aligned}
 \partial_t v -\Delta v + v\cdot \nabla v + \nabla q & \,=\, g \,,  
 \quad {\rm div}\, v \,=\, 0\,,   ~~~~~~ t>0\,,~y \in \Omega (t)\,, \\
  v & \,=\, \alpha y^\bot \,, ~~~~~~ t>0\,,~ y \in \partial \Omega (t)\,, \\
  v & \, \rightarrow \,     0   \,,  ~~~~~~ t>0\,,~ |y| \rightarrow \infty\,.
\end{aligned}\right.
\end{equation}
Here $v=v(y,t) = (v_1(y,t), v_2 (y,t))^\top$ and $q=q(y,t)$ are respectively unknown velocity field and pressure field, and $g(y,t) = (g_1(y,t), g_{2}(y,t))^\top$ is a given external force. 
The time dependent domain $\Omega(t)$ is defined as 
\begin{align}
\begin{split}
\Omega(t) & \, = \, \big \{ y\in \R^2~|~ y = O (\alpha t) x\,, x\in \Omega \big \}\,,\\
O (\alpha t ) & \, = \,
\begin{pmatrix}
\cos \alpha t & -\sin \alpha t\\
\sin \alpha t & \cos \alpha t
\end{pmatrix}\,,
\end{split}
\end{align}
where $\Omega$ is an exterior domain in $\R^2$ with a smooth compact boundary, while the real number $\alpha \in \R\setminus\{0\}$ represents the rotation speed of the obstacle $\Omega^c=\R^2\setminus\Omega$. 
We use the standard notation for derivatives: $\partial_t = \frac{\partial}{\partial t}$, $\partial_j = \frac{\partial}{\partial x_j}$, $\Delta = \sum_{j=1}^2 \partial^2_j$, ${\rm div}\, v = \sum_{j=1}^2 \partial_j v_j$, $v\cdot \nabla v = \sum_{j=1}^2 v_j \partial_j v$. The vector $x^{\perp}$ denotes the perpendicular: $x^\bot =(-x_2,x_1)^\top$. 
The system \eqref{NS} describes the flow around the obstacle $\Omega^c$ which rotates with a constant angular velocity $\alpha$, and the condition $v (t,y)=\alpha y^\bot$ on the boundary  $\partial \Omega (t)$  represents the no-slip boundary condition. To remove the difficulty due to the time dependence of the fluid domain 
it is more convenient to analyze the system \eqref{NS} in the reference frame: for $t\geq 0$ and $x\in \Omega$,
\begin{align*}
y = O (\alpha t) x\,, \quad u (x,t) = O (\alpha t)^\top v (y,t)\,, \quad p (x,t) = q (y,t)\,, \quad f (x,t) = O (\alpha t)^\top g (y,t)\,.
\end{align*}
Here $M^\top$ denotes the transpose of a matrix $M$. Then \eqref{NS} is equivalent with the equations in the time-independent domain $\Omega$:
\begin{equation*}
  \left\{
\begin{aligned}
 \partial_t u -\Delta u - \alpha ( x^\bot \cdot \nabla u - u^\bot ) + \nabla p & \,=\,  - u\cdot\nabla u  + f \,,  
 \quad {\rm div}\, u \,=\, 0\,, ~~~~~~ t>0\,,~x \in \Omega\,, \\
  u  & \,=\, \alpha x^\bot \,, ~~~~~~ t>0\,,~ x \in \partial \Omega\,, \\
  u  & \, \rightarrow \,     0   \,,  ~~~~~~ t>0\,,~ |x| \rightarrow \infty\,.
\end{aligned}\right.
\end{equation*}
In this paper we are interested in the stationary solutions to this system. Thus we assume that $f$ is independent of $t$ and consider the next system
\begin{equation}\tag{NS$_{\alpha}$}\label{NS_alpha.intro}
  \left\{\begin{aligned}
  -\Delta u - \alpha ( x^\bot \cdot \nabla u - u^\bot ) + \nabla p  &\,=\, - u\cdot\nabla u  + f \,, 
  \quad {\rm div}\, u \,=\, 0\,,  ~~~~~~   x \in \Omega\,, \\
 u  & \, =\,  \alpha x^\bot   \,,  ~~~~~~ x\in \partial\Omega\,, \\
 u  & \, \rightarrow \,     0   \,,  ~~~~~~ |x| \rightarrow \infty\,.
\end{aligned}\right.
\end{equation}
To state our result let us introduce the function spaces used in this paper. 
As usual, the class $C_{0,\sigma}^\infty (\Omega)$ is defined as the set of smooth divergence free vector fields with compact support in $\Omega$, and the homogeneous space $\dot{W}^{1,2}_{0,\sigma} (\Omega)$ is the closure of $C_{0,\sigma}^\infty (\Omega)$ with respect to the norm $\| \nabla f\|_{L^2 (\Omega)}$.
For a fixed number $s\geq 0$ we also introduce the weighted $L^\infty$ space $L^\infty_s (\Omega)$ 
and its subspace $L^\infty_{s,0}(\Omega)$ as follows.
\begin{align}\label{def.L^infty_s}
\begin{split}
L^\infty_s (\Omega) & \, = \, \big \{ f\in L^\infty (\Omega)~|~ (1+|x|)^s f \in L^\infty (\Omega) \big \}\,,\\
L^\infty_{s,0} (\Omega) & \, = \, \big \{ f\in L^\infty_s (\Omega) ~|~ \lim_{R\rightarrow \infty} {\rm ess.sup}_{|x|\geq R} |x|^s |f(x)| \, = \, 0\big \}\,.
\end{split}
\end{align}
These are Banach spaces equipped with the natural norm 
$$\|f\|_{L^\infty_s (\Omega)} = {\rm ess.sup}_{x\in\Omega} (1+|x|)^s |f(x)|\,,$$
and the set of functions with compact support is dense in $L^\infty_{s,0}(\Omega)$. Moreover, for any bounded sequence $\{f_n\}$ in $L^\infty_s (\Omega)$ (or $L^\infty_{s,0}(\Omega)$) with $\| f_n \|_{L^\infty_s (\Omega)} \le M$ for some $M>0$, there exists a subsequence $\{f_{n'}\}$ which converges in the weak-star topology in the sense that there is $f\in L^\infty_s (\Omega)$ (or $f\in L^\infty_{s,0}(\Omega)$, respectively) such that 
$$\lim_{n'\rightarrow \infty} \int_\Omega f_{n'} (x)  \phi (x)  (1+|x|)^s \dd x \, = \,  \int_\Omega f (x) \phi (x) (1+|x|)^s \dd x\,, \qquad {\rm for ~any~}~ \phi \in L^1 (\Omega)$$ and $\| f \|_{L^\infty_s (\Omega)} \le M$.
We denote by $L^2_{loc}(\overline{\Omega})$ the set of functions which belong to 
$L^2 (\Omega \cap K)$ for any compact set $K\subset \R^2$, and $W^{k,2}_{loc} (\overline{\Omega})$, $k=1,2,\cdots$, is defined in the similar manner.

The main result of this paper is stated as follows.
\begin{theorem}\label{thm.main} There exists $\epsilon=\epsilon (\Omega)>0$ such that the following statement holds. Assume that $f\in L^2 (\Omega)^2$ is of the form $f={\rm div}\, F=(\partial_1 F_{11} + \partial_2 F_{12}, \partial_1 F_{21} + \partial_2 F_{22})^\top$ with some $F=(F_{ij})_{1\leq i,j\leq 2}\in L^\infty_2 (\Omega)^{2\times 2}$ and $F_{12}-F_{21}\in L^1 (\Omega)$. If $\alpha \ne 0$ and 
\begin{align}\label{est.thm.main.1}
|\alpha|^\frac12\big  |\log |\alpha | \big | \, + \, |\alpha|^{-\frac12} \big  |\log |\alpha| \big | \, \big ( \| f \|_{L^2 (\Omega)} + \| F \|_{L^\infty_2 (\Omega)} + \| F_{12}-F_{21} \|_{L^1 (\Omega)} \big ) <\epsilon\,,
\end{align}
then there exists a solution $(u,\nabla p)\in \big (W^{2,2}_{loc}(\overline{\Omega}) \cap L^\infty_1 (\Omega) \big )^2 \times L^2_{loc} (\overline{\Omega})^2$ to \eqref{NS_alpha.intro}, which is unique in a suitable class of functions (see Theorem \ref{thm.nonlinear} for the precise description). If $F\in L^\infty_{2,0} (\Omega)^{2\times 2}$ in addition, then the solution $u$ behaves as 
\begin{align}\label{est.thm.main.2}
u(x) & \, = \, \beta \frac{x^\bot}{4\pi |x|^2} \, + \, o (|x|^{-1})\,, \qquad |x|\rightarrow \infty\,,
\end{align}
where
\begin{align}\label{est.thm.main.3}
\beta & \, = \, \int_{\partial\Omega} y^\bot \cdot \big ( T (u,p) \nu \big ) \dd \sigma_y \, + \, \lim_{\delta\rightarrow 0} \int_\Omega e^{-\delta |y|^2} y^\bot \cdot f \dd y\,. 
\end{align}
Here $T(u,p) = \nabla u + (\nabla u)^\top - p \, \mathbb{I}$\,, $\mathbb{I}=(\delta_{ij})_{1\leq i,j\leq 2}$\,, denotes the Cauchy stress tensor, and $\nu$ is the outward unit normal vector to $\partial\Omega$.
\end{theorem}

\begin{remark} (i) The smallness condition on $f$ and $F$ in \eqref{est.thm.main.1} can be slightly weakened with respect to the dependence on $\alpha$;  see Theorem \ref{thm.nonlinear} for details.

\noindent 
(ii) Both conditions $F\in L^\infty_2 (\Omega)^{2\times 2}$ and $F_{12}-F_{21}\in L^1 (\Omega)$
are critical in view of scaling. Note that the $L^1$ summability is needed only for the antisymmetric part of $F$.
These conditions are not enough to ensure that  $u$ behaves like the circular flow $\beta \frac{x^\bot}{4\pi |x|^2}$ as $|x|\rightarrow \infty$,
and the additional decay condition $F\in L^\infty_{2,0} (\Omega)^{2\times 2}$ as in Theorem \ref{thm.main} is required to achieve this asymptotic property.

\noindent 
(iii) The second term of the right-hand side of \eqref{est.thm.main.3} is well-defined if $F\in L^\infty_{2,0} (\Omega)$ and $F_{12}-F_{21}\in L^1 (\Omega)$. If $F$ possesses an additional decay such as $L^\infty_{2+\gamma} (\Omega)$ with $\gamma\in (0,1)$ then the order $o(|x|^{-1})$ in \eqref{est.thm.main.2} is replaced by 
$O(|x|^{-1-\gamma})$ at least when $|\alpha|$ and given data $f$ are further small depending on $\gamma$.
The precise statement on this result is stated in Theorem \ref{thm.nonlinear}. 

\noindent 
(iv) The pressure $p$ is determined uniquely up to a constant and belongs to $W^{1,2}_{loc} (\overline{\Omega})$. Then, since $u\in W^{2,2}_{loc}(\overline{\Omega})^2$, the coefficient $\beta$ in \eqref{est.thm.main.3} is well-defined. 

\noindent 
(v) In Theorem \ref{thm.main} we assume that the external force $f$ is of divergence form. 
In fact, this is not an essential assumption, and it is possible to deal with the external force $f$ satisfying 
\begin{align}\label{rem.add.1}
 x^\bot \cdot f\in L^1 (\Omega)\,, \qquad f\in L^\infty_{3}(\Omega)^2\,,
\end{align}
with the smallness in these norms, and the asymptotic expansion \eqref{est.thm.main.2} is verified if $f\in L^\infty_{3,0} (\Omega)^2$ in addition. This is obtained by using the recent result by the authors \cite{HMN} in the whole space which solves the linearized problem for $f$ satisfying \eqref{rem.add.1}.
Although this result is not so trivial since the condition \eqref{rem.add.1} is just in the scale-critical regime,
we focus only on $f$ of divergence form in this paper, for the argument becomes shorter due to the fact that the nonlinear term is also written in the divergence form as ${\rm div}\, (u\otimes u)$.    
\end{remark}

As far as the authors know, Theorem \ref{thm.main} is the first general existence result of the flows around a rotating obstacle {\it in the two-dimensional case}.
Before stating the idea of the proof of Theorem \ref{thm.main}, 
let us recall some known results on the mathematical  analysis of flows around a rotating obstacle.

So far the mathematical results on this topic have been obtained mainly for the three-dimensional problem,
as listed below.
For the nonstationary problem the existence of global weak solutions is proved by Borchers \cite{Bo}, and the unique existence of time-local regular solutions is shown by Hishida \cite{H1} and Geissert, Heck, and Hieber \cite{GHH}, while the global strong solutions for small data are obtained by Galdi and Silvestre \cite{GSi}. The spectrum of the linear operator  related to this problem is studied by Farwig and Neustupa \cite{FN}; see also the linear analysis by Hishida \cite{H2}. The existence of stationary solutions to the associated system is proved in \cite{Bo}, Silvestre \cite{Si},  Galdi \cite{G1}, and Farwig and Hishida \cite{FH0}. In particular, in \cite{G1} the stationary flows with the decay order  $O(|x|^{-1})$ are obtained, while the work of \cite{FH0} is based on the weak $L^{3}$ framework, which is another natural scale-critical space for the three-dimensional Navier-Stokes equations. Our Theorem \ref{thm.main} is considered as a two-dimensional counterpart of the three-dimensional result of \cite{G1}.
In $3$D case the asymptotic profiles of these stationary flows at spatial infinity are studied by Farwig and Hishida \cite{FH1,FH2} and Farwig, Galdi, and Kyed \cite{FGK}, where it is proved that the asymptotic profiles are described by the Landau solutions, stationary self-similar solutions to the Navier-Stokes equations in 
$\R^3\setminus\{0\}$. 
It is worthwhile to mention that, also in the two-dimensional case,
the asymptotic profile is given by the stationary self-similar solution $c \frac{x^\bot}{|x|^2}$, 
as is shown in Theorem \ref{thm.main}.
The stability of the above stationary solutions has been well studied in the three-dimensional case;
The global $L^2$ stability is proved in \cite{GSi}, 
and the local $L^3$ stability is obtained by Hishida and Shibata \cite{HShi}.

All results mentioned above are in the three-dimensional case, 
while only a few results are known so far for the flow around a rotating obstacle in the two-dimensional case. 
Recently an important progress has been made by Hishida \cite{H3}, 
where the asymptotic behavior of the two-dimensional stationary Stokes flow around a rotating obstacle is investigated in details.
The equations studied in \cite{H3} are written as 
\begin{equation}\tag{S$_{\alpha}$}\label{S_alpha.intro}
  \left\{
\begin{aligned}
  -\Delta u - \alpha ( x^\bot \cdot \nabla u - u^\bot ) + \nabla p & \,=\, f \,,  
  \quad {\rm div}\, u \,=\, 0\,, ~~~~~~ x \in \Omega, \\
  u & \,=\, b \,, ~~~~~~ x \in \partial \Omega\,. \\
  u & \, \rightarrow \,     0   \,,  ~~~~~~ |x| \rightarrow \infty\,,
\end{aligned}\right.
\end{equation}
Here $b$ is a given smooth function on $\partial\Omega$. It is proved in \cite{H3} that if $\alpha \ne 0$ and the smooth external force  $f$ satisfies the decay conditions
\begin{align}\label{condition.f}
\int_\Omega |x| |f| \dd x <\infty\,, \qquad f(x) \, = \, o \big (|x|^{-3} (\log |x|)^{-1} \big )\,, \quad {\rm as} ~|x|\rightarrow \infty\,,
\end{align}
then the solution $u$ to \eqref{S_alpha.intro} decaying at spatial infinity obeys the asymptotic expansion
\begin{align}\label{asymptotic.h.1}
u(x) \, = \, \frac{c_1 x^\bot - 2 c_2 x}{4\pi |x|^2} \, + \, (1+ |\alpha|^{-1}) \, o (|x|^{-1})\,, \qquad {\rm as} \quad  |x|\rightarrow \infty\,,
\end{align}
where 
\begin{align}\label{asymptotic.h.2}
\begin{split}
c_1 & \, = \, \int_{\partial\Omega} y^\bot \cdot \big ( T (u,p)  + \alpha \,  b \otimes y^\bot  \big ) \nu \dd \sigma_y  + \int_\Omega y^\bot \cdot f \dd y\,,\\
c_2 & \, = \, \int_{\partial\Omega} b \cdot \nu \dd \sigma_y\,. 
\end{split}
\end{align}
The result of \cite{H3} leads to an important conclusion that the rotation of the obstacle resolves the Stokes paradox (see Chang and Finn \cite{CF} for the rigorous description of the Stokes paradox) 
as in the Oseen resolution. 
We recall that when the obstacle is translating with a constant velocity $u_\infty\in \R^2\setminus \{0\}$ the Navier-Stokes flows have been constructed by Finn and Smith \cite{FS1,FS2} for small but nonzero $u_\infty$ through the analysis of the Oseen linearization; see also Galdi \cite{G3}. 
The resolution of the Stokes paradox for \eqref{S_alpha.intro} is due to the fact that the rotation removes the logarithmic singularity of the associated fundamental solution, which has been well known for the Oseen problem.

As a reference to the $2$D exterior problem related with ours, the reader is referred to a recent work by Hillairet and Wittwer \cite{HW}, where the stationary problem of \eqref{NS} is discussed when $\Omega (t) = \Omega=\{y\in \R^2~|~|y|>1\}$ and the boundary condition is given as $v=\alpha y^\bot + b$ with a smooth and time-independent $b$. 
We note that the stationary flow $\alpha \frac{y^\bot}{|y|^2}$ exactly solves this problem when $b=0$.
When $\alpha$ is large enough and $b$ is sufficiently small the stationary solutions are constructed  in \cite{HW} around the explicit solution $\tilde \alpha \frac{y^\bot}{|y|^2}$, where $\tilde \alpha$ a number  close to $\alpha$. Although the problem discussed in \cite{HW} is in fact different from ours due to the time-independent given data $b$ in the original frame \eqref{NS}, the solutions obtained in \cite{HW} share a common property with the ones in Theorem \ref{thm.main} in view of their asymptotic behaviors at spatial infinity.

It is well known that the existence of stationary Navier-Stokes flows in two-dimensional exterior domains  (hence, formally $\alpha=0$ in \eqref{NS_alpha.intro}) is an open problem in general. 
Partial results related to this problem have been obtained by 
Galdi \cite{G2}, Russo \cite{R}, Yamazaki \cite{Y}, and Pileckas and Russo \cite{PR},
where the solutions are constructed under some symmetry conditions on both domains and given data.
In particular, the Navier-Stokes flows decaying in the scale-critical order $O(|x|^{-1})$ are obtained in \cite{Y} in this category. The uniqueness is also available again under some symmetry conditions, see Nakatsuka \cite{N}.

The stability of the stationary solutions obtained in \cite{Y,HW} or in Theorem \ref{thm.main} is a highly challenging issue due to their spatial decay in the scale-critical order in two-dimensions, 
and it is still an open question in general. The difficulty is brought from the fact that the Hardy inequality 
$\| \frac{1}{|x|} f\|_{L^2 (\Omega)} \leq C \| \nabla f \|_{L^2(\Omega)}$, $f\in \dot{W}_0^{1,2}(\Omega)$, 
does not hold when $\Omega$ is an exterior domain in $\R^2$.
As far as the authors know, the only result available so far is \cite{M} by the second author of this paper,
where the local $L^2$ stability is established for the special solution $\alpha \frac{x^\bot}{|x|^2}$, $|\alpha|\ll 1$, when $\Omega$ is the exterior domain to the unit disk.

Finally, let us state the key idea for the proof of Theorem \ref{thm.main}. Our approach is motivated by the linear analysis developed in \cite{H3}, where \eqref{asymptotic.h.2} is obtained through the detailed analysis of the fundamental solution associated to the system \eqref{S_alpha.intro} in $\R^2$. The expansion \eqref{asymptotic.h.1} strongly indicates that 
the similar asymptotics is valid also for the Navier-Stokes flow, since the leading profile in \eqref{asymptotic.h.1} is a stationary self-similar solution to the Navier-Stokes equations in $\R^2\setminus\{0\}$. Thus our strategy for the proof of Theorem \ref{thm.main} can be summarized as follows: we derive at the same time the unique existence of solutions and their asymptotic behavior under the smallness condition on the given data $(\alpha,f)$ in \eqref{NS_alpha.intro}. The solution of the form of $u=\beta \frac{x^\bot}{|x|^2} + w$ is constructed through the Banach fixed point theorem, where  both the coefficient $\beta$ and the remainder term $w$ are sufficiently small corresponding to the size of $(\alpha,f)$. 
However, it is far from trivial to justify this idea directly from the results of \cite{H3}, 
especially to ensure the smallness of $(\beta, w)$ in the iteration scheme.
Indeed, there are at least two difficulties for this procedure: (I)
the condition \eqref{condition.f} is slightly restrictive to handle the nonlinear term $u\cdot \nabla u$ in the scale-critical framework, and more seriously, (II) the singularity on $|\alpha|$ in \eqref{asymptotic.h.1} for $0<|\alpha|\ll 1$ can prevent us closing the nonlinear estimates. In fact, the smooth flows subject to the system \eqref{NS_alpha.intro} are pointwise bounded above by $|\alpha|$ near the boundary due to the boundary condition $u= \alpha x^\bot$.

For resolving the difficulty (I), the structure of the nonlinear term $\nabla\cdot (u\otimes u)$ is essential.
Indeed, the symmetry of the tensor $u\otimes u$ leads to a crucial cancellation for the coefficient ``$\int_\Omega y^\bot \cdot (u\cdot \nabla u) \dd y$'', which removes a possible singularity caused by the scale-critical decay of the flow. 
To overcome the difficulty (II), we revisit 
the argument of \cite{H3} analyzing the fundamental solution to \eqref{S_alpha.intro} in $\R^2$
and 
modify the singularity of $\alpha$ in
the estimates of the remainder term for the linear problem; see Theorem \ref{thm.linear.whole}, Lemma \ref{lem.thm.linear.whole.1}, and Theorem \ref{thm.linear.exterior}.
Applying these improved estimates, the nonlinear problem \eqref{NS_alpha.intro} is solved 
by the Banach fixed point theorem.
However,
the argument becomes complicated since we have to control two kinds of norms; 
the one bounds the local quantity, while the other one controls the spatial decay. 
This machinery is needed since the flow in a far field region ($|x|\gg 1$) exhibits a different dependence on $|\alpha|$ from the flow in a finite fluid region, and in principle, the problem becomes more singular at $|x|\gg 1$ as $|\alpha|$ is decreasing. In order to close the nonlinear estimates it is important to distinguish these two dependences on $|\alpha|$
and to estimate their interaction through the nonlinearity carefully.

This paper is organized as follows. In Section \ref{sec.pre} the basic results on the oscillatory integrals are collected, which are used to establish the pointwise estimates of the fundamental solution to \eqref{S_alpha.intro} with a milder singularity on $|\alpha|$, $|\alpha|\ll 1$. In Section \ref{sec.linear} the linearized problem \eqref{S_alpha.intro} with $b=0$ is studied in details. Section \ref{subsec.linear.whole} is devoted to the analysis in $\R^2$, while the exterior problem is discussed in Section \ref{subsec.linear.exterior}. Finally the nonlinear problem \eqref{NS_alpha.intro} is solved in Section \ref{sec.nonlinear}.

\section{Preliminaries}\label{sec.pre}

In this section we collect the results of the oscillatory integrals used in Section \ref{subsec.linear.whole}.

\begin{lemma}\label{lem.thm.linear.whole.2}
Let $\alpha \in \R \setminus \{0\}$ 
and let $m, r>0$. Then we have
\begin{equation}\label{est.lem.thm.linear.whole.2.1}
\bigg| 
\int_{0}^{\infty} e^{i\alpha t} e^{ -\frac{r^2}{t} }\frac{\dd t}{t^m} \bigg| 
+
\bigg| 
\int_{0}^{\infty} e^{i\alpha t} \int_{t}^{\infty}
e^{ -\frac{r^2}{s} } \frac{\dd s}{s^{m+1}} \dd t \bigg| 
\leq C 
\min
\big\{
\frac{1}{|\alpha |r^{2m}},
\frac{1}{ |\alpha|^{\frac{1}{m+1}} r^{\frac{2m^2}{m+1}}}
\big\}\,,
\end{equation}
where  $C=C(m)$ is independent of $r$ and $\alpha$. 
Moreover, for $m> 1$ we have
\begin{equation}\label{est.lem.thm.linear.whole.2.2}
\int_{0}^{\infty} e^{ -\frac{r^2}{t} } \frac{ \dd t }{ t^m }  
= \frac{ \gamma(m-1) }{ r^{2(m-1)} }\,, 
\qquad
\displaystyle \int_{0}^{\infty} 
\displaystyle \int_{t}^{\infty}
e^{ -\frac{r^2}{s} } \frac{\dd s}{s^{m+1}} \dd t
= \frac{ \gamma(m-1) }{ r^{2(m-1)} }\,, 
\end{equation}

\noindent 
where $\gamma(\cdot)$ denotes the Euler gamma function.
\end{lemma}

\begin{proof}
The proof of \eqref{est.lem.thm.linear.whole.2.2} is a straightforward computation, 
and we omit the details.
To show \eqref{est.lem.thm.linear.whole.2.1} let us take a positive constant $l = l(r, \alpha) $ 
which will be determined later and split the integral as
\begin{equation*}
\begin{split} 
\int_{0}^{\infty} 
e^{i\alpha t} e^{ -\frac{r^2}{t} }\frac{\dd t}{t^m}
& \,=\,
\int_{0}^{l} e^{i\alpha t} e^{ -\frac{r^2}{t} }\frac{\dd t}{t^m}
+ \int_{l}^{\infty} e^{i\alpha t} e^{ -\frac{r^2}{t} }\frac{\dd t}{t^m}\,.
\end{split}
\end{equation*}
The first term is estimated without using the effect of oscillation:
\begin{equation*}
\begin{split} 
\bigg| \int_{0}^{l} e^{i\alpha t} e^{ -\frac{r^2}{t} } \frac{\dd t}{t^m} \bigg|
& \leq 
\frac{1}{r^{2m}}
\int_{0}^{l} e^{ -\frac{r^2}{t} }\bigg( \frac{r^2}{t} \bigg)^m \dd t 
\leq
\frac{Cl}{r^{2m}}\,.
\end{split}
\end{equation*}
For the second term we use the effect of oscillation to obtain
\begin{equation*}
\begin{split} 
\int_{l}^{\infty} e^{i\alpha t} e^{ -\frac{r^2}{t} } \frac{\dd t}{t^m}
& \,=\, 
\frac{1}{i \alpha}
\int_{l}^{\infty} \frac{\dd}{\dd t} \bigg[ e^{i \alpha t} \bigg]
\frac{ e^{ -\frac{r^2}{t}} }{t^m} \dd t \\
& \,=\, 
\frac{1}{i \alpha} 
\bigg[ 
e^{i \alpha t} 
\frac{ e^{ -\frac{r^2}{t}} }{t^m}
\bigg]^{t = \infty}_{t = l}
- \frac{1}{i \alpha} 
\int_{l}^{\infty} 
e^{i \alpha t} 
\bigg(
\frac{ r^2 e^{ -\frac{r^2}{t}} }{ t^{m+2} }
- \frac{m e^{ -\frac{r^2}{t}} }{t^{m+1}} 
\bigg) \dd t\,,
\end{split}
\end{equation*}
which yields
\begin{equation}\label{proof.lem.thm.linear.whole.2.1.1}
\begin{split} 
\bigg|
\int_{l}^{\infty} 
e^{i\alpha t} 
e^{ -\frac{r^2}{t} } 
\frac{\dd t}{t^m}
\bigg|
& \le
\frac{1}{|\alpha| } 
\bigg(
\frac{ e^{-\frac{r^2}{l}} }{l^m}
+ \frac{1}{r^{2(m+1)}}
\int_{l}^{\infty}  
\big ( \frac{r^2}{t} + m\big )
\bigg( \frac{r^2}{t} \bigg)^{m+1}  e^{-\frac{r^2}{t} } \dd t
\bigg)\,.
\end{split}
\end{equation}
By taking the limit of $l=0$ we observe that the left-hand side of
\eqref{proof.lem.thm.linear.whole.2.1.1} is then bounded from above by $\frac{C}{|\alpha| r^{2m}}$
in virtue of \eqref{est.lem.thm.linear.whole.2.2}.  
On the other hand, the right-hand side of \eqref{proof.lem.thm.linear.whole.2.1.1} is also bounded from above by 
$\frac{C}{|\alpha| l^{m}}$.
Taking  $l = r^{\frac{2m}{m+1}} |\alpha|^{-\frac{1}{m+1}}$, we have arrived at 
\begin{equation*}
\begin{split} 
\bigg|
\int_{0}^{\infty} 
e^{i\alpha t} e^{ -\frac{r^2}{t} }\frac{\dd t}{t^m}
\bigg|
& \le
\frac{C}{ |\alpha|^{\frac{1}{m+1}} r^{\frac{2m^2}{m+1}}}\,.
\end{split}
\end{equation*}
The estimate of the integral
\begin{equation*}
\int_{0}^{\infty} e^{i\alpha t} \int_{t}^{\infty}
e^{ -\frac{r^2}{s} } \frac{\dd s}{s^{m+1}} \dd t
\end{equation*}
is obtained exactly in the same manner, and hence the details are omitted here.
The proof is complete.
\end{proof}
%
\begin{lemma}\label{lem.thm.linear.whole.3}
Let $m>1$. Then we have
\begin{equation}\label{est.lem.thm.linear.whole.3.1}
\begin{split}
&~~~
\int_{0}^{\infty} 
\big | e^{-\frac{|O(\alpha t)x-y|^2}{4t}} - e^{-\frac{|x|^2}{4t}} \big |
\frac{\dd t}{t^m} 
 \, + \, \int_{0}^{\infty}  \int_{t}^{\infty}
\big | e^{-\frac{|O(\alpha t)x-y|^2}{4s}} - e^{-\frac{|x|^2}{4s}} \big |
\frac{\dd s}{s^{m+1}} \dd t  \\
& \leq C
\frac{|y|}{ |x|^{2m-1} }\,, \qquad \quad |x| > 2 |y|\,,
\end{split}
\end{equation}
and
\begin{equation}\label{est.lem.thm.linear.whole.3.2}
\bigg| 
\int_{0}^{\infty} 
e^{i \alpha t} e^{-\frac{|x|^2}{4t}} \frac{\dd t}{t^m} 
\bigg| 
\leq C
\min \big\{ \frac{1}{|\alpha| |x|^{2m}}, \frac{1}{|x|^{2(m-1)}} \big\}\,, 
\qquad  |x| >0\,.
\end{equation}
Moreover, for $m > 1$ we have
\begin{equation}\label{est.lem.thm.linear.whole.3.3}
\bigg| \displaystyle \int_{0}^{\infty} 
e^{i \alpha t}
\displaystyle \int_{t}^{\infty}
e^{-\frac{|x|^2}{4s}} 
\frac{\dd s}{s^{m+1}} \dd t 
\bigg| 
\leq C
\min \big\{ \frac{1}{|\alpha| |x|^{2m}}, \frac{1}{|x|^{2(m-1)}} \big\}\,,
\qquad  |x|>0 \,.
\end{equation}
Here  $C=C(m)$ is independent of $x$, $y$, and $\alpha$.
\end{lemma}

\begin{proof} 
By using the Taylor formula with respect to $y$ around $y=0$, we see 
\begin{equation}\label{formulation}
e^{-\frac{|O(\alpha t)x-y|^2}{4t}} 
\,=\, 
e^{-\frac{|x|^2}{4t}} + \frac{\langle O (\alpha t) x, y\rangle}{2 t} e^{-\frac{|x|^2}{4t}} + \frac{\langle y, Q y\rangle}{8 t^2} e^{-\frac{ |O(\alpha t)x-\theta y|^2 }{4t}}\,,
\end{equation}
where $Q= \big (O (\alpha t) x- \theta y \big ) \otimes \big ( O (\alpha t) x- \theta y \big ) - 2 t \mathbb{I}$
with $\theta =\theta(\alpha,t,x,y)\in (0,1)$ and $\langle x,y\rangle = x\cdot y$.
From 
\begin{align*}
|O(\alpha t)x-\theta y| \geq |x| - |y| > \frac{|x|}{2}\,, \qquad |x|> 2|y|\,, 
\end{align*}
Lemma \ref{lem.thm.linear.whole.2} leads to
\begin{equation*}
\begin{split}
& \int_{0}^{\infty} 
\big | e^{-\frac{|O(\alpha t)x-y|^2}{4t}} - e^{-\frac{|x|^2}{4t}} \big |
\frac{\dd t}{t^m}  \\
& \leq 
C \bigg( |x| |y|  \int_{0}^{\infty} e^{-\frac{|x|^2}{4t}} \frac{\dd t}{t^{m+1}} 
+  (|x|^2 |y|^2 + |x| |y|^3 + |y|^4) \int_{0}^{\infty} e^{-\frac{|x|^2}{16t}} 
\frac{\dd t}{ t^{m+2} } \bigg) \\
& \leq \frac{ C |y| }{ |x|^{2m-1} }\,, \qquad |x| > 2 |y|\,.
\end{split}
\end{equation*}
Similarly we have from Lemma \ref{lem.thm.linear.whole.2},
\begin{equation*}
\int_{0}^{\infty}  \int_{t}^{\infty}
\big | e^{-\frac{|O(\alpha t)x-y|^2}{4s}} - e^{-\frac{|x|^2}{4s}} \big |
\frac{\dd s}{s^{m+1}} \dd t  \leq 
\frac{ C |y| }{ |x|^{2m-1} }\,, \qquad |x| > 2 |y|\,.
\end{equation*}
The proof of \eqref{est.lem.thm.linear.whole.3.1} is complete. Since $m>1$, 
the estimates \eqref{est.lem.thm.linear.whole.3.2} and \eqref{est.lem.thm.linear.whole.3.3} are consequences of \eqref{est.lem.thm.linear.whole.2.1} and \eqref{est.lem.thm.linear.whole.2.2}. The proof is complete.
\end{proof}

\section{Stokes system with a rotation effect}\label{sec.linear}

This section is devoted to the analysis of the linearized problem \eqref{S_alpha.intro}, introduced in Section \ref{intro}, with $b=0$.

\subsection{Linear estimate in the whole plane}\label{subsec.linear.whole}

In this subsection let us consider the linear problem in whole plane for $\alpha\in \R \setminus \{0\}$:
\begin{equation}\tag{${\rm S}_{\alpha,\R^2}$}\label{a}
- \Delta u -\alpha (x^{\perp} \cdot \nabla u-u^{\perp})+ \nabla p \, =\, f\,, 
\qquad {\rm div}\ u \, =\, 0\,,  
~~~~~~ x \in  \mathbb{R}^2\,.
\end{equation}
Our main interest is the estimate of solutions that are represented in terms of the fundamental solution defined by \eqref{def.fundamental.solution} below. We will see that such solutions decay at spatial infinity for a suitable class of $f$ in virtue of the effect from the rotation; see also Remark \ref{rem.thm.linear.whole} about the uniqueness for  solutions to \eqref{a}. 
The couple $(u,p)$ is said to be a weak solution to \eqref{a} 
if $(u, p)\in L^{q_1}(\R^2)^2\times L^{q_2} (\R^2)$ for some $q_1\in [2,\infty)$ and $q_2\in [1,\infty)$, 
and (i) ${\rm div}\, u=0$ in the sense of distributions, and (ii) $(u,p)$ satisfies
\begin{equation*}\label{def.weak.whole}
\int_{\R^2}  u\cdot \mathcal{L}_{-\alpha} \phi \dd x 
- \int_{\R^2} p \, {\rm div}\, \phi \dd x  
\,=\,  \int_{\R^2} f \cdot \phi \dd x\,, \quad {\rm for~all} \ \ \phi\in \mathcal{S}(\R^2)^2\,,
\end{equation*}
where the operator $\mathcal{L}_{\alpha}$ is defined as
\begin{align*}
\mathcal{L}_{\alpha} u \,=\, -\Delta u -\alpha(x^\bot \cdot \nabla u -u^\bot)\,.
\end{align*}
The fundamental solution to \eqref{a} plays a central role throughout this paper, which is defined as 
\begin{equation}\label{def.fundamental.solution}
\Gamma_{\alpha}(x,y) \, = \,  \int_{0}^{\infty}  O(\alpha t)^{\top} K(O(\alpha t)x-y,t) \dd t\,,
\end{equation}
where
\begin{equation*}
K(x,t) \,=\, G(x,t) \mathbb{I} + H(x,t)\,,
\qquad  
H(x,t) \,=\, \int_{t}^{\infty} \nabla^2 G(x,s) \dd s\,,
\end{equation*}
and $G(x,t)$ is the two-dimensional Gauss kernel
\begin{equation*}
G(x,t) \,=\, \frac{1}{4\pi t} e^{-\frac{|x|^2}{4t}}\,.
\end{equation*}
The next theorem is the main result of this subsection, which extends the result of \cite{H3} to our functional setting. For $f\in L^2 (\R^2)^2$ and $F=(F_{ij})_{1\leq, i,j\leq 2} \in L^2 (\R^2)^{2\times 2}$ we formally set 
\begin{align}\label{def.c.tildec}
\begin{split}
c[f]  & \,=\, \lim_{\epsilon \rightarrow 0} \int_{\R^2} e^{-\epsilon |y|^2}  y^\bot \cdot f(y) \dd y\,,\\
\tilde c[F] & \,=\, \lim_{\epsilon \rightarrow 0} \int_{\R^2} e^{-\epsilon |y|^2}   \big( F_{12}(y) - F_{21}(y) \big) \dd y\,.
\end{split}
\end{align}
Note that if $f\in L^2 (\R^2)^2$ is of the form $f={\rm div}\, F=(\partial_1 F_{11}+\partial_2 F_{12}, \partial_1 F_{21}+\partial_2 F_{22})^\top$ with some $F\in L^1 (\R^2)^{2\times 2}$, then $c[f] = \tilde c[F]$. Indeed, from the integration by parts we have 
\begin{align*}
c[f] \,=\, 
\tilde c[F] 
+ \lim_{\epsilon \rightarrow 0} 2 \int_{\R^2} e^{-\epsilon |y|^2} \epsilon y^\bot  \cdot \big ( F (y)  y \big )  \dd y\,.
\end{align*} 
Then the Lebesgue dominated convergence theorem implies $c[f] = \tilde c[F]$. Moreover, if $F$ is symmetric then $\tilde c[F]=0$. Here and in what follows, $B_{R}$ denotes the open disk in $\R^2$ of radius $R>0$ and centered at the origin, and the complement of $B_{R}$ is denoted as $B^{{\rm c}}_{R} = \{ x \in \R^2 ~|~ |x| \ge R  \}$.
%
\begin{theorem}\label{thm.linear.whole}
Let $\alpha \in \R \setminus \{0\}$. We formally set
\begin{equation}
L [f] (x) \, = \, \lim_{\epsilon\rightarrow 0} \int_{\R^2} e^{-\epsilon |y|^2}\Gamma_{\alpha}(x,y) f(y) \dd y\,.
\end{equation}
Then the following statements hold.

\noindent 
{\rm (i)} Let $\gamma \in [0,1)$. 
Suppose that $f\in L^2 (\R^2)^2$ satisfies ${\rm supp}\, f \subset B_{R}$ for some $R\ge1$. 
Then $u=L [f]$ is a weak solution to \eqref{a} and is written as 
\begin{equation}\label{est.thm.linear.whole.1}
u(x) \, =\,  c[f] \frac{x^{\bot}}{4\pi |x|^2} + \mathcal{R}[f] (x) \,, \qquad x\ne 0\,,
\end{equation}
where $\mathcal{R}[f]$ satisfies 
\begin{equation}\label{est.thm.linear.whole.2}
\begin{split}
\| \mathcal{R}[f] \|_{L^{\infty}_{1+\gamma} (B^{{\rm c}}_{2R})}
\le 
C_1 \big( 
|\alpha|^{-\frac{1+\gamma}{2}} \| f \|_{L^1 (B_{R})} 
+ \| |y|^{1+\gamma} f\|_{L^{1}(B_{R})} \big)\,.
\end{split}
\end{equation}
Here $C_1$ is a numerical constant, and is independent of  $\gamma$, $\alpha$, $R$, and $f$.

\noindent {\rm (ii)} Let $\gamma\in [0,1)$. 
Suppose that $f\in L^2 (\R^2)^2$ is of the form $f={\rm div}\,F$ 
with some $F\in L^\infty_{2+\gamma} (\R^2)^{2\times 2}$, and in addition that $\tilde c[F]$ in \eqref{def.c.tildec} converges when $\gamma=0$. Then $u=L [f]$ is a weak solution to \eqref{a} and is written as
\begin{equation}\label{est.thm.linear.whole.3}
u(x) \, = \, \tilde c[F] \frac{x^{\bot}}{4\pi |x|^2}  + \mathcal{R}[f] (x) \,, \quad \quad x\ne 0\,,
\end{equation}
\noindent
where  $\mathcal{R}[f]$ satisfies for $R\ge1$,
\begin{equation}\label{est.thm.linear.whole.4}
\begin{split}
\| \mathcal{R}[f] \|_{L^{\infty}_{1+\gamma} (B^{{\rm c}}_{2R})}
& \leq 
C_2  \bigg(  \| F \|_{L^{\infty}_{2+\gamma} (B^{{\rm c}}_{R})}
 + \sup_{|x|\geq 2R} |x|^{-1+\gamma} \| y F \|_{L^1 (B_{\frac{|x|}{2}})} \\
& \quad  + \sup_{|x|\geq 2 R} \min \big \{ \frac{1}{|\alpha| |x|^{2-\gamma}}, |x|^\gamma \big \} \| F \|_{L^1 (B_{\frac{|x|}{2}})} \\
& \qquad + \sup_{|x|\geq 2R} |x|^\gamma \, 
\big| \lim_{\epsilon\rightarrow 0} \int_{2|y|\geq |x|}  e^{-\epsilon |y|^2} \big(F_{12}(y) - F_{21}(y) \big) \dd y \big|   
\bigg)\,.
\end{split}
\end{equation}
Here $C_2$ is a numerical constant, and is independent of $\gamma$, $\alpha$, $R$, and $f$.
\end{theorem}

\begin{remark}\label{rem.thm.linear.whole} 
Under the assumptions of (i) or (ii) in Theorem \ref{thm.linear.whole} it is not difficult to see that $L[f]$ belongs to $W^{2,2}_{loc} (\R^2)$, and thus, $L[f]$ is bounded in $\R^2$ by the Sobolev embedding in $B_1$ and the estimates stated in Theorem \ref{thm.linear.whole} for $|x|\geq 1$ (by taking $R=1$). Set 
\begin{align}
p \,=\, \int_{\R^2} \frac{x-y}{2\pi |x-y|^2} f (y) \dd y\,.\label{def.pressure}
\end{align}
Then, $\nabla p$ belongs to $L^2 (\R^2)^2$ under the assumptions of (i) or (ii) in Theorem \ref{thm.linear.whole} by the Calder{\'o}n-Zygmund inequality, and as is shown in \cite[Proposition 3.2]{H3}, the pair $(L[f], \nabla p)$ satisfies \eqref{a} in the sense of distributions. In virtue of the uniqueness result stated in \cite[Lemma 3.5]{H3}, 
if $f$ satisfies one of the assumptions in Theorem \ref{thm.linear.whole}, and if $(v,q) \in \mathcal{S}'(\R^2)^2 \times \mathcal{S}'(\R^2)$ is a solution to \eqref{a} in the sense of distributions, then $(v,q)$ has a representation as $v=L[f]+P_{1}$ and $q=p+P_{2}$ with some polynomials $P_{1}$ and $P_{2}$. Hence, by the definition stated above, any weak solution $(u, p)$ to \eqref{a} is represented as $u=L[f]$ and $p$ is given by \eqref{def.pressure}, as long as the condition (i) or (ii) on $f$ in Theorem \ref{thm.linear.whole} is assumed.
\end{remark}

We note that in (ii) of Theorem \ref{thm.linear.whole} the coefficient $\tilde c[F]$ is always well-defined when $\gamma>0$. The asymptotic expansion \eqref{est.thm.linear.whole.1} for the case (i)  is firstly established by \cite[Proposition 3.2]{H3}.
Indeed, for the case (i) it is shown in \cite[Proposition 3.2]{H3} that
 $\mathcal{R}[f]$ decays at infinity as $O(|x|^{-2})$,
while the singularity $|\alpha|^{-1}$ appears in the coefficient of the estimates there.
\noindent The novelty of Theorem \ref{thm.linear.whole} are \eqref{est.thm.linear.whole.2} and \eqref{est.thm.linear.whole.4},
where both the consistency in the weighted $L^\infty$ spaces and the milder singularity on $\alpha$ for small $|\alpha|$ are essential to solve the nonlinear problem in Section \ref{sec.nonlinear}. 
On the other hand, as in \cite{H3}, the key step to prove Theorem \ref{thm.linear.whole} is the expansion and 
the pointwise estimate of the fundamental solution $\Gamma_\alpha (x,y)$, which are stated in Lemma \ref{lem.thm.linear.whole.1} below. The fundamental solution $\Gamma_\alpha (x,y)$ is studied in details in \cite[Proposition 3.1]{H3} and we will revisit the argument developed by \cite{H3} in the proof of this lemma.

\begin{lemma}\label{lem.thm.linear.whole.1} Set
\begin{equation}\label{def.Lxy}
L(x,y) \, = \, \frac{x^{\perp}\otimes y^{\perp}}{4\pi |x|^2}\,. 
\end{equation}

\noindent Then for $m=0,1$ the kernel $\Gamma_{\alpha}(x,y)$ satisfies
\begin{equation}\label{est.lem.thm.linear.whole.1.1}
\begin{split}
& | \nabla_y ^m \big ( \Gamma_{\alpha}(x,y)- L(x,y) \big ) | \\
& \le
C \bigg ( \delta_{0m} \min \big\{ \frac{1}{|\alpha| |x|^2}, \frac{1}{|\alpha|^{\frac{1}{2}} |x|} \big\}
+  |x|^{1-m} 
\min \big\{ \frac{1}{|\alpha| |x|^3}, \frac{1}{|x|} \big\} 
+  \frac{|y|^{2-m}} {|x|^2}\bigg )\,,\\
&  \quad\quad\quad\quad\quad \quad\quad\quad\quad\quad\quad \quad\quad \quad\quad\quad \quad\quad\quad\quad\quad  {\rm for}   \quad |x|> 2 |y|\,.
\end{split}
\end{equation}

\noindent
Here $\delta_{0m}$ is the Kronecker delta and $C$ is independent of $x$, $y$, and $\alpha$.

\end{lemma}

\begin{remark} The case $m=0$ of \eqref{est.lem.thm.linear.whole.1.1} is obtained in \cite[Proposition 3.1]{H3} but with $|\alpha|^{-1}$ dependence of the coefficients in the estimate. The case $m=1$ is not stated explicitly in \cite{H3}, 
although it can be handled in the similar spirit as in the case $m=0$. In this sense Lemma \ref{lem.thm.linear.whole.1} is not completely new, and is an improvement of \cite[Proposition 3.1]{H3} with respect to the singularity on $|\alpha|$ 
for $|\alpha|\ll 1$.
\end{remark}

\begin{proofx}{Lemma \ref{lem.thm.linear.whole.1}}
In principle, our proof of Lemma \ref{lem.thm.linear.whole.1} will proceed
along the line of  \cite[Proposition 3.1]{H3}.
In fact, the only key difference of out proof for the case $m=0$ is the application of Lemmas \ref{lem.thm.linear.whole.2}, \ref{lem.thm.linear.whole.3} in  suitable parts. 
In the proof for the case $m=1$, the inequality \eqref{est.lem.thm.linear.whole.3.1} will be essentially used in addition. 

Following the argument of \cite[Section 3]{H3}, we decompose $\Gamma_\alpha (x,y)$ and define $\Gamma_{\alpha}^0(x,y)$, $\Gamma_{\alpha}^{11}(x,y)$, and $\Gamma_{\alpha}^{12}(x,y)$ as
\begin{equation}\label{proof.lem.thm.linear.whole.1.1}
\begin{split}
& ~~~ \Gamma_{\alpha}(x,y) \\
& \,=\, \Gamma_{\alpha}^0(x,y)+\Gamma_{\alpha}^{11}(x,y)+\Gamma_{\alpha}^{12}(x,y) \\
& \,=\, \int^{\infty}_0 O(\alpha t)^{\top} G(O(\alpha t)x-y,t) \dd t \\
&\quad + \int^{\infty}_0 O(\alpha t)^{\top}  (O(\alpha t)x-y)\otimes(O(\alpha t)x-y) \int^{\infty}_t G(O(\alpha t)x-y,s) \frac{\dd s}{4s^2} \dd t \\
&\quad - \int^{\infty}_0 O(\alpha t)^{\top} \int^{\infty}_t G(O(\alpha t)x-y,s)\frac{\dd s}{2s} \dd t\,.
\end{split}
\end{equation}
We also  decompose $L(x,y)$ and define $L^0(x,y)$, $L^{111}(x,y)$, $L^{112}(x,y)$, and $L^{122}(x,y)$ as 
\begin{equation}\label{proof.lem.thm.linear.whole.1.2}
\begin{split}
L(x,y)
& \,=\,  L^0(x,y) + L^{111}(x,y) + L^{112}(x,y) + L^{12}(x,y) \\ 
&\,=\, 
\frac{x \otimes y + x^\bot \otimes y^\bot}{4\pi |x|^{2}}
+ \frac{-3(x \otimes y) + x^\bot \otimes y^\bot}{8\pi |x|^{2}} 
+ \frac{x \otimes y}{4\pi |x|^{2}}
- \frac{x \otimes y + x^\bot \otimes y^\bot}{8\pi |x|^{2}} \,.
\end{split}
\end{equation}
Then, by Lemma \ref{lem.thm.linear.whole.2} the following representations hold:
\begin{equation}\label{proof.lem.thm.linear.whole.1.3}
\begin{split}
& L^0(x,y) \, = \, 
\int^{\infty}_0 G(x,t) \frac{\dd t}{4 t} 
\begin{pmatrix}
x\cdot y  & x^{\bot}\cdot y \\
-x^{\bot}\cdot y & x\cdot y \\
\end{pmatrix}\,, \\
& L^{111}(x,y) \,= \, 
\int_{0}^{\infty} \int^{\infty}_t G(x,s)  \frac{\dd s}{4s^2} \dd t
\, \bigg( \frac{-3(x\otimes y) +(x^{\perp}\otimes y^{\perp})}{2} \bigg)\,, \\
& L^{112}(x,y) \, = \, 
 \int^{\infty}_0 \int^{\infty}_t G(x,s) \frac{\dd s}{16s^3} \dd t
\, |x|^2 (x\otimes y)\,, \\
& L^{12}(x,y) \, =\, 
- \int^{\infty}_0 \int^{\infty}_t G(x,s) \frac{\dd s}{8s^2} \dd t
\begin{pmatrix}
x\cdot y & x^{\bot} \cdot y \\
-x^{\bot} \cdot y & x\cdot y \\
\end{pmatrix}\,, \\
\end{split}
\end{equation}
where we have used the equality
\begin{equation*}
x \otimes y + x^\bot \otimes y^\bot
\, = \,
\begin{pmatrix}
x\cdot y  & x^{\bot}\cdot y \\
-x^{\bot}\cdot y & x\cdot y \\
\end{pmatrix}\,.
\end{equation*}
To prove \eqref{est.lem.thm.linear.whole.1.1}  we observe that 
\begin{equation*}
\begin{split}
& | \nabla_y^m \big (\Gamma_{\alpha}(x,y)-L(x,y) \big ) | \\
& \leq | \nabla_y^m \big ( \Gamma_{\alpha}^0(x,y)-L^0(x,y)\big )  | 
+ | \nabla_y^m \big ( \Gamma_{\alpha}^{11}(x,y)-L^{111}(x,y)-L^{112}(x,y) \big )  | \\
& ~~~ + |\nabla_y^m \big (  \Gamma_{\alpha}^{12}(x,y)-L^{12}(x,y) \big )|\,. 
\end{split}
\end{equation*}
Let us estimate each term in the right-hand side of the above inequality. The key idea is to use the Taylor formula for $G(O(\alpha t) x -y, t')$ around $y=0$ as follows.
\begin{equation}\label{proof.lem.thm.linear.whole.1.4}
\begin{split}
G(O(\alpha t) x -y, t') & \, = \,  G(x,t')  +  \frac{\langle O(\alpha t)x, y\rangle}{2t'} G(x,t')   + \frac{\langle y, Q y\rangle}{8{t'}^2} G(O(\alpha t) x - \theta y, t') \, \,,
\end{split}
\end{equation}
where 
\begin{align*}
Q \,=\, Q(x,\theta y,\alpha t, t')  \, = \, (O(\alpha t)x-\theta y)\otimes (O(\alpha t)x-\theta y)-2t' \mathbb{I}\,,
\end{align*}
and $\theta = \theta(\alpha,t',x,y) \in (0,1)$. To estimate $\Gamma_\alpha^0 (x,y) - L^{0}(x,y)$ we use the identity
\begin{equation}\label{proof.lem.thm.linear.whole.1.5}
\begin{split}
O(\alpha t)^{\top} \langle O(\alpha t)x, y \rangle
&\, = \,  
\frac{1}{2}
\begin{pmatrix}
x\cdot y  & x^{\bot}\cdot y \\
-x^{\bot}\cdot y & x\cdot y \\
\end{pmatrix} \\
& \quad
+ \frac{\cos{2\alpha t}}{2}
\begin{pmatrix}
x\cdot y  & -x^{\bot}\cdot y \\
x^{\bot}\cdot y & x\cdot y \\
\end{pmatrix} 
+ \frac{\sin{2\alpha t}}{2}
\begin{pmatrix}
x^{\bot}\cdot y  & x\cdot y \\
-x\cdot y & x^{\bot}\cdot y \\
\end{pmatrix}\,.
\end{split}
\end{equation}
Let $|x|>2 |y|$. Then we have from \eqref{proof.lem.thm.linear.whole.1.4} and \eqref{proof.lem.thm.linear.whole.1.5},
\begin{align}
| \Gamma_{\alpha}^0(x,y)-L^0(x,y) | 
& \, = \, 
\bigg| 
\int_{0}^{\infty} O(\alpha t)^{\top} G(x,t) \dd t \nonumber \\
& \quad
+ \int_{0}^{\infty} 
\frac{1}{2 t} 
\bigg(
O(\alpha t)^{\top} \langle O(\alpha t)x,  y \rangle
-\frac{1}{2}
\begin{pmatrix}
x\cdot y  & x^{\bot}\cdot y \\
-x^{\bot}\cdot y & x\cdot y \\
\end{pmatrix}
\bigg)
G(x,t) \dd t \nonumber \\
& \qquad
+ \int_{0}^{\infty} 
O(\alpha t)^{\top} \frac{\langle y, Q y\rangle}{8t^2} G(O(\alpha t) x - \theta y, t)\dd t 
\bigg| \nonumber  \\
\begin{split}
& \le
\bigg| 
\int_{0}^{\infty} 
O(\alpha t)^{\top} G(x,t) \dd t 
\bigg| 
+ C |x| |y| \min \big\{ \frac{1}{|\alpha| |x|^4}, \frac{1}{|x|^2} \big\} \\
& \quad
+ C |y|^2
\int_{0}^{\infty} 
\big\{ (|x|^2 + |x| |y| + |y|^2)t^{-3} + t^{-2} \big\}
e^{-\frac{|x|^2}{16t}} \dd t \,.\label{proof.lem.thm.linear.whole.1.6}
\end{split}
\end{align}
Here we have used \eqref{est.lem.thm.linear.whole.3.2} for the second term and used the condition $|x|>2|y|$ for the third term to achieve the last line. Clearly the last term in the right-hand side of \eqref{proof.lem.thm.linear.whole.1.6} is bounded from above by $C \frac{|y|^2}{ |x|^{2}}$ for $|x|>2|y|$, while in virtue of \eqref{est.lem.thm.linear.whole.2.1} the first term is estimated as
\begin{equation}\label{proof.lem.thm.linear.whole.1.7}
\bigg| 
\displaystyle \int_{0}^{\infty} 
O(\alpha t)^{\top} G(x,t) \dd t 
\bigg| 
\le
C \min \big\{ \frac{1}{|\alpha| |x|^2}, \frac{1}{|\alpha|^{\frac{1}{2}} |x|} \big\}\,, \qquad |x| > 0 \,.
\end{equation}
Thus we have arrived at
\begin{align}
&| \Gamma_{\alpha}^0(x,y)-L^0(x,y) | \nonumber \\ 
\begin{split}
& \le C \bigg ( \min \big\{ \frac{1}{|\alpha| |x|^2}, \frac{1}{|\alpha|^{\frac{1}{2}} |x|} \big\}
+  |y| \min \big\{ \frac{1}{|\alpha| |x|^3}, \frac{1}{|x|} \big\} 
+  \frac{|y|^2} {|x|^2}\bigg )\,,  \qquad  |x| > 2 |y|\,.
\label{proof.lem.thm.linear.whole.1.8}
\end{split}
\end{align}
Next we consider the derivative estimate for $\Gamma_\alpha^0 (x,y)-L^0 (x,y)$. Let us go back to the definition of $\Gamma_\alpha^0 (x,y)$ in \eqref{proof.lem.thm.linear.whole.1.1}. Then  $\partial_{y_k} \big (\Gamma_\alpha^0 (x,y)-L^0 (x,y)\big )$ is computed as
\begin{align}
&~~~ \big|\partial_{y_k}(\Gamma_{\alpha}^0(x,y)-L^0(x,y)) \big| \nonumber \\ 
& = \bigg| \int_{0}^{\infty} 
\bigg(
\frac{ O(\alpha t)^{\top} (O(\alpha t)x-y)_k }{2 t} 
G(O(\alpha t)x-y, t)
-\frac{1}{4t}
\partial_{y_k}
\begin{pmatrix}
x\cdot y  & x^{\bot}\cdot y \\
-x^{\bot}\cdot y & x\cdot y \\
\end{pmatrix}
G(x,t) \bigg) \dd t
\bigg|  \nonumber \\
\begin{split}\label{proof.lem.thm.linear.whole.1.81}
& \leq 
\bigg| \int_{0}^{\infty} \frac{ O(\alpha t)^{\top} (O(\alpha t)x-y)_k }{2 t} 
\bigg ( G(O(\alpha t)x-y, t) - G(x,t) \bigg ) \dd t \bigg | \\
& \quad +  \bigg | \int_{0}^{\infty} 
\bigg(
O(\alpha t)^{\top} (O(\alpha t)x-y)_k
- \frac{1}{2}
\partial_{y_k}
\begin{pmatrix}
x\cdot y  & x^{\bot}\cdot y \\
-x^{\bot}\cdot y & x\cdot y \\
\end{pmatrix}
\bigg) G(x,t) \frac{\dd t}{2t} \bigg| \,.
\end{split}
\end{align}
By applying \eqref{est.lem.thm.linear.whole.3.1} the first term is bounded from above by $C \frac{(|x|+|y|)|y|}{|x|^3}$. To estimate the second term we observe that
\begin{equation}\label{proof.lem.thm.linear.whole.1.82}
\begin{split}
&~~~ O(\alpha t)^{\top} (O(\alpha t)x-y)_k
- \frac{1}{2}
\partial_{y_k}
\begin{pmatrix}
x\cdot y  & x^{\bot}\cdot y \\
-x^{\bot}\cdot y & x\cdot y \\
\end{pmatrix} \\
& \,=\,
\begin{cases}
\displaystyle 
\frac{\cos{2\alpha t} }{2}
\begin{pmatrix}
x_1 & x_2 \\
-x_2 &x_1 
\end{pmatrix} 
+ \frac{ \sin{2\alpha t} }{2}
\begin{pmatrix}
-x_2  &x_1 \\
-x_1 &-x_2 
\end{pmatrix} 
- y_1 O(\alpha t)^{\top}\,, 
\quad \mathrm{if} \ k\,=\,1\,, \\  
\displaystyle
\frac{\cos{2\alpha t}}{2}
\begin{pmatrix}
x_2  &-x_1 \\
x_1 &x_2
\end{pmatrix} 
+ \frac{\sin{2\alpha t}}{2}
\begin{pmatrix}
x_1 \ &x_2 \\
-x_2 &x_1 
\end{pmatrix} 
-y_2 O(\alpha t)^{\top}\,,
\quad \mathrm{if} \ k \,=\, 2\,,
\end{cases} 
\end{split}
\end{equation}
Then, by using \eqref{est.lem.thm.linear.whole.3.2} the second term in the right-hand side of \eqref{proof.lem.thm.linear.whole.1.81} is bounded from above by 
$C(|x|+|y|) \min\{ \frac{1}{|\alpha| |x|^4}, \frac{1}{|x|^2} \}$. 
Hence we have shown that
\begin{align}
\big |\partial_{y_k} (\Gamma_{\alpha}^0(x,y)-L^0(x,y)) \big |
& \leq C \bigg ( \frac{ |y| }{ |x|^2 }
+ \min \big\{ \frac{1}{ |\alpha| |x|^{3} }, \frac{1}{ |x| } \big \} \bigg) \,,  
\qquad |x| > 2|y|\,.\label{proof.lem.thm.linear.whole.1.83}
\end{align}
Exactly in the same way we obtain for $m=0,1$ and $|x|>2 |y|$,
\begin{equation}\label{proof.lem.thm.linear.whole.1.9}
\begin{split}
& | \nabla_y^m \big ( \Gamma_{\alpha}^{12}(x,y)-L^{12}(x,y) \big )  | \\
& \le
C \bigg (  \delta_{0m}\min \big\{ \frac{1}{|\alpha| |x|^2}, \frac{1}{|\alpha|^{\frac{1}{2}} |x|} \big\}
+  |y|^{1-m} \min \big\{ \frac{1}{|\alpha| |x|^3}, \frac{1}{|x|} \big\} 
+  \frac{|y|^{2-m}} {|x|^2}\bigg )\,.
\end{split}
\end{equation}
Next we estimate the term 
$|\Gamma_{\alpha}^{11}(x,y)-L^{111}(x,y)-L^{112}(x,y)|$.
By the Taylor expansion stated in \eqref{proof.lem.thm.linear.whole.1.4}, 
we decompose $\Gamma_{\alpha}^{11}(x,y)$ and define $\Gamma_{\alpha}^{111}(x,y)$, $\Gamma_{\alpha}^{112}(x,y)$, and $\Gamma_{\alpha}^{113}(x,y)$ as
\begin{equation*}
\begin{split}
&~~~ \Gamma_{\alpha}^{11}(x,y) \\
& \,=\, \Gamma_{\alpha}^{111}(x,y) + \Gamma_{\alpha}^{112}(x,y)+ \Gamma_{\alpha}^{113}(x,y) \\
& \,=\, \int_{0}^{\infty}  O(\alpha t)^{\top} (O(\alpha t)x-y)\otimes(O(\alpha t)x-y) 
\int_{t}^{\infty} G(x,s) \frac{\dd s}{4 s^2} \dd t \\
& \quad + \int_{0}^{\infty}  O(\alpha t)^{\top} (O(\alpha t)x-y)\otimes(O(\alpha t)x-y) 
\int_{t}^{\infty} \langle O(\alpha t)x, y\rangle  G(x,s) \frac{\dd s}{8 s^3} \dd t \\
& \quad +  \int_{0}^{\infty} O(\alpha t)^{\top} (O(\alpha t)x-y)\otimes(O(\alpha t)x-y)  
\int_{t}^{\infty} \langle y, Q y \rangle G(O(\alpha t)x-\theta y,s) \frac{\dd s}{32 s^4} \dd t\,.
\end{split}
\end{equation*}
For the last term $\Gamma_{\alpha}^{113}(x,y)$ it is straightforward to see from \eqref{est.lem.thm.linear.whole.2.2} that, for $|x|> 2 |y|$,
\begin{align}
|\Gamma_{\alpha}^{113}(x,y)|  \le
C |y|^2 ( |x| + |y| )^2 \int_{0}^{\infty}  \int_{t}^{\infty} (|x|^2+|y|^2 + s ) e^{-\frac{ |x|^2 }{16s}} 
\frac{\dd s}{s^5} \dd t & \le C \frac{|y|^2}{|x|^2}\,.\label{proof.lem.thm.linear.whole.1.10}
\end{align}
To estimate the first two terms we observe
\begin{equation}\label{proof.lem.thm.linear.whole.1.11}
\begin{split}
&~~~ O(\alpha t)^{\top} (O(\alpha t)x-y)\otimes(O(\alpha t)x-y) \\
& \, =\,  A_0  + (\cos{\alpha t}) A_1 + (\sin{\alpha t})A_2 + \frac{\cos{2\alpha t}}{2} A_3 
+ \frac{\sin{2\alpha t}}{2}A_4\,,
\end{split}
\end{equation}
where
\begin{equation*}
\begin{split}
&
A_0(x,y) 
\,=\, \frac{-3(x\otimes y) + (x^{\perp} \otimes y^{\perp})}{2}\,, 
\qquad
A_1 (x,y) 
\,=\, 
\begin{pmatrix}
x^2_1+y^2_1   &x_1x_2+y_1y_2 \\
x_1 x_2+y_1y_2 &x^2_2+y^2_2 
\end{pmatrix}\,, \\
&
A_2(x,y)
\,=\,
\begin{pmatrix}
-x_1x_2+y_1y_2   &x^2_1+y^2_2 \\
-(x^2_2+y^2_1)  &x_1x_2-y_1y_2
\end{pmatrix}\,, 
\qquad
A_3 (x,y)
\,=\,
\begin{pmatrix}
-x\cdot y   &x^{\perp}\cdot y \\
-x^{\perp}\cdot y  &-x\cdot y 
\end{pmatrix}\,, \\
&
A_4 (x,y)
\,=\,
\begin{pmatrix}
-x^{\perp}\cdot y   &-x\cdot y\ \\
x\cdot y  &-x^{\perp}\cdot y
\end{pmatrix}\,.
\end{split}
\end{equation*}
Then, by using  \eqref{proof.lem.thm.linear.whole.1.11} and by applying \eqref{est.lem.thm.linear.whole.2.1} the term $\Gamma_\alpha^{111}(x,y)$ is estimated as 
\begin{align}
& \big| \Gamma_\alpha^{111} (x,y) - L^{111}(x,y) \big|  \nonumber \\
& = \bigg|  \int_{0}^{\infty}  \int_{t}^{\infty}
\bigg(
(\cos{\alpha t}) A_1 
+ (\sin{\alpha t})A_2 + \frac{\cos{2\alpha t}}{2}A_3 
+ \frac{\sin{2\alpha t}}{2} A_4
\bigg) G(x,s)  \frac{\dd s}{4 s^2} \dd t 
\bigg|  \nonumber \\
& \leq 
|x| \min \big\{ \frac{1}{|\alpha| |x|^3}, \frac{1}{|x|} \big\} \,, 
\qquad |x| > 2 |y|\,.
\label{proof.lem.thm.linear.whole.1.13}
\end{align}
Next we see 
\begin{align}\label{proof.lem.thm.linear.whole.1.12}
\begin{split}
&~~~ \langle O(\alpha t) x, y \rangle \, O(\alpha t)^{\top} (O(\alpha t)x-y)\otimes(O(\alpha t)x-y) \\
& \,=\,  \frac{|x|^2}{2} x\otimes y + (\cos 2\alpha t)B_1(x,y) +  (\sin 2\alpha t)B_2(x,y) + B_3 (x,y,\alpha t)\,,
\end{split}
\end{align}
where each component of the matrices $B_1$ and $B_2$ is a fourth order polynomial of $x,y$ written as a suitable sum of  the terms $x_1^{l_1} x_2^{l_2} y_1^{k_1} y_2^{k_2}$ with $l_1+l_2=3$ and $k_1+k_2=1$, while $B_3$ is estimated as $|B_3|\leq C|x|^2|y|^2$ for $|x|>2|y|$. Thus we have from \eqref{proof.lem.thm.linear.whole.1.12} and \eqref{est.lem.thm.linear.whole.2.1},
\begin{align}
&~~~ \big| \Gamma_\alpha^{112} (x,y) - L^{112}(x,y) \big| \nonumber \\
& \leq  \bigg|  \int_{0}^{\infty}  \int_{t}^{\infty}
\bigg( (\cos 2\alpha t)B_1(x,y) +  (\sin 2\alpha t)B_2(x,y) \bigg) G(x,s)  \frac{\dd s}{8 s^3} \dd t 
\bigg| \nonumber \\
& ~~~ + C |x|^2 |y|^2 \int_{0}^{\infty}  \int_{t}^{\infty}  G(x,s) \frac{\dd s}{s^3} \dd t \nonumber \\
& \leq 
C \bigg ( |x| \min \big\{ \frac{1}{|\alpha| |x|^3}, \frac{1}{|x|} \big\} 
+ \frac{|y|^2}{|x|^2} \bigg )\,, \qquad |x| > 2 |y|\,.
\label{proof.lem.thm.linear.whole.1.14}
\end{align}
Summing up \eqref{proof.lem.thm.linear.whole.1.10}, \eqref{proof.lem.thm.linear.whole.1.13}, and \eqref{proof.lem.thm.linear.whole.1.14}, we obtain
\begin{equation}\label{proof.lem.thm.linear.whole.1.15}
\begin{split}
&~~~ \big| \Gamma_{\alpha}^{11}(x,y)-L^{111}(x,y)-L^{112}(x,y) \big| \\
& \le
C \bigg( \min \big\{ \frac{1}{|\alpha| |x|^2}, \frac{1}{|\alpha|^{\frac{1}{2}} |x|} \big\}
+ 
|x| \min \big\{ \frac{1}{|\alpha| |x|^3}, \frac{1}{|x|} \big\} 
+  \frac{|y|^2} {|x|^2} \bigg )\,,  \quad |x| > 2 |y|\,. 
\end{split}
\end{equation}
To estimate the derivatives in $y$ of $\Gamma_\alpha^{11}(x,y)$ we recall
the definition of $\Gamma_\alpha^{11} (x,y)$ in \eqref{proof.lem.thm.linear.whole.1.1}
and use \eqref{proof.lem.thm.linear.whole.1.11}, which leads to the representation
\begin{equation}\label{proof.lem.thm.linear.whole.1.16}
\begin{split}
&~~~ \Gamma_{\alpha}^{11}(x,y) \\
&\,=\, 
\int_0^\infty \int_t^\infty A_0 \, G( O(\alpha t) x-y, s) \frac{\dd s}{4 s^2} \dd t \\
& + \int_{0}^{\infty}  \int_{t}^{\infty}
\bigg ( (\cos{\alpha t}) A_1 + (\sin{\alpha t})A_2 + \frac{\cos{2\alpha t}}{2} A_3  + \frac{\sin{2\alpha t}}{2} A_4 
\bigg)  G ( O(\alpha t) x-y, s) \frac{\dd s}{4 s^2} \dd t\\
& \,=\, \tilde \Gamma_\alpha^{111}  (x,y) + \tilde \Gamma_\alpha^{112} (x,y)\,.
\end{split}
\end{equation}
From the expression of $L^{111}(x,y)$ in \eqref{proof.lem.thm.linear.whole.1.3}, we have for $|x|> 2 |y|$,
\begin{align}
&~~~\big| \partial_{y_k} \big ( \tilde \Gamma_\alpha^{111}  (x,y)  - L^{111} (x,y) \big ) \big| \nonumber \\
& \, = \, \bigg |  \int^{\infty}_0 \ \int^{\infty}_t  \big ( \partial_{y_k} A_0 \big ) \bigg ( G ( O(\alpha t) x -y, s) - G(x,s) \bigg ) \frac{\dd s}{4 s^2} \dd t \nonumber \\
& \quad 
+ \int^{\infty}_0 \ \int^{\infty}_t  (O(\alpha t)x-y)_{k}  \, A_0 \, 
\bigg( G ( O(\alpha t) x -y, s) - G(x,s) \bigg ) \frac{\dd s}{8 s^3} \dd t  \nonumber \\
& \qquad 
+ \int^{\infty}_0 \ \int^{\infty}_t  (O(\alpha t)x-y)_{k} \, A_0 \,  G(x,s)  \frac{\dd s}{8 s^3} \dd t  \bigg | \nonumber \\
& \le C \bigg ( \frac{ |x| |y| }{ |x|^3 } + \frac{ (|x|^2 |y| + |x| |y|^2) |y| }{|x|^5}  + \frac{ (|x|^2 |y| + |x| |y|^2) }{|x|^4}
\bigg )  
\le C \frac{ |y| }{ |x|^2 }\,.\label{proof.lem.thm.linear.whole.1.17}
\end{align}
Here we have used \eqref{est.lem.thm.linear.whole.3.1}. Next we estimate the derivatives of $\tilde \Gamma_\alpha^{112} (x,y)$, which are computed as 
\begin{align}
&~~~ \partial_{y_k} \tilde \Gamma_\alpha^{112} (x,y) \nonumber \\
& \,=\, \int_{0}^{\infty}  \int_{t}^{\infty}
\bigg ( (\cos{\alpha t}) \partial_{y_k} A_1 + (\sin{\alpha t}) \partial_{y_k} A_2 + \frac{\cos{2\alpha t}}{2} \partial_{y_k} A_3  + \frac{\sin{2\alpha t}}{2} \partial_{y_k} A_4  \bigg)  \nonumber \\
& \quad \quad \quad \times G ( O(\alpha t) x-y, s) \frac{\dd s}{4 s^2} \dd t \nonumber \\
& \quad + \int_{0}^{\infty}  \int_{t}^{\infty} (O(\alpha t)x-y)_{k} 
\bigg ( (\cos{\alpha t}) A_1 + (\sin{\alpha t}) A_2 + \frac{\cos{2\alpha t}}{2} A_3  + \frac{\sin{2\alpha t}}{2}  A_4  \bigg)  \nonumber \\
& \quad \quad \quad \times G ( O(\alpha t) x-y, s) \frac{\dd s}{8 s^3} \dd t \nonumber \\
& \, = \,  I_k (x,y) + II_k (x,y)\,.\label{proof.lem.thm.linear.whole.1.18}
\end{align}
To estimate $I_k (x,y)$ we observe that 
\begin{align}
& \bigg | \int_{0}^{\infty}  \int_{t}^{\infty}
\bigg ( (\cos{\alpha t}) \partial_{y_k} A_1 + (\sin{\alpha t}) \partial_{y_k} A_2 \bigg)   G ( O(\alpha t) x-y, s) \frac{\dd s}{4 s^2} \dd t \bigg | \nonumber \\
& \leq C |y| \int_0^\infty \int_t^\infty e^{-\frac{|x|^2}{16 s}} \frac{\dd s}{s^3} \dd t \,  \leq \,  C \frac{|y|}{|x|^2}\,, 
\qquad |x| > 2 |y|\,,\label{proof.lem.thm.linear.whole.1.19}
\end{align}
and that
\begin{align}
& \bigg | \int_{0}^{\infty}  \int_{t}^{\infty}
\bigg ( \frac{\cos{2\alpha t}}{2} \partial_{y_k} A_3  + \frac{\sin{2\alpha t}}{2} \partial_{y_k} A_4  \bigg)   G ( O(\alpha t) x-y, s) \frac{\dd s}{4 s^2} \dd t \bigg | \nonumber \\
& \leq \bigg | \int_{0}^{\infty}  \int_{t}^{\infty}
\bigg ( \frac{\cos{2\alpha t}}{2} \partial_{y_k} A_3  + \frac{\sin{2\alpha t}}{2} \partial_{y_k} A_4  \bigg)   \bigg ( G ( O(\alpha t) x-y, s) - G(x,s) \bigg )  \frac{\dd s}{4 s^2} \dd t \bigg | \nonumber \\
& \quad +  \bigg | \int_{0}^{\infty}  \int_{t}^{\infty}
\bigg ( \frac{\cos{2\alpha t}}{2} \partial_{y_k} A_3  + \frac{\sin{2\alpha t}}{2} \partial_{y_k} A_4  \bigg)   G (x, s) \frac{\dd s}{4 s^2} \dd t \bigg | \nonumber \\
& \leq C \frac{|y|}{|x|^2} +  C \min \big \{ \frac{1}{|\alpha| |x|^3}, \frac{1}{|x|} \big \}\,,
\qquad |x| > 2 |y|\,.
\label{proof.lem.thm.linear.whole.1.20}
\end{align}
Here we have used \eqref{est.lem.thm.linear.whole.3.1} for the first term and \eqref{est.lem.thm.linear.whole.3.3} for the second term to derive the last line. It remains to estimate $II_k (x,y)$ in \eqref{proof.lem.thm.linear.whole.1.18}. Below we consider the case $k=1$ only, for the case $k=2$ is obtained in the same manner. 
The direct computation yields the following key identity:
\begin{align}\label{proof.lem.thm.linear.whole.1.21}
\begin{split}
&~~~ (O(\alpha t)x-y)_{1}  \big ( (\cos\alpha t) A_1 + (\sin\alpha t) A_2  \big ) \\
& \,=\, 
\frac{|x|^2}{2} 
\begin{pmatrix}
x_1 & 0 \\
x_2 & 0 
\end{pmatrix} 
+ (\cos 2\alpha t) 
D_1 (x,y) 
+ (\sin 2\alpha t) 
D_2 (x,y) 
+ D_3 (x,y, \alpha t)\,.
\end{split}
\end{align}
Here $D_1$ and $D_2$ are the matrices whose components are suitable sums of the third order polynomials of the form $x_1^{l_1} x_2^{l_2} y_1^{k_1} y_2^{k_2}$ with $l_1+l_2 \geq 1$,
while $D_3(x,y,\alpha t)$ is estimated as $|D_3|\leq C |x|^2 |y|$ for $|x|>2|y|$. 
Hence, recalling the expression of $L^{112}(x,y)$ in \eqref{proof.lem.thm.linear.whole.1.3}, we have
\begin{align}
& \bigg |\int_{0}^{\infty}  \int_{t}^{\infty} (O(\alpha t)x-y)_{1} 
\big ( (\cos{\alpha t}) A_1 + (\sin{\alpha t}) A_2  \big)   G ( O(\alpha t) x-y, s) \frac{\dd s}{8 s^3} \dd t  - \partial_{y_1} L^{112} (x,y) \bigg | \nonumber \\
&  \, = \, \bigg |\int_{0}^{\infty}  \int_{t}^{\infty} 
\bigg ( (\cos 2\alpha t) D_1 + (\sin 2\alpha t) D_2  + D_3  \bigg )  \,  G ( O(\alpha t) x-y, s) \frac{\dd s}{8 s^3} \dd t  \bigg |  \nonumber \\
& \leq \bigg |\int_{0}^{\infty}  \int_{t}^{\infty} 
\bigg ( (\cos 2\alpha t) D_1 + (\sin 2\alpha t) D_2  + D_3  \bigg )  \,  \bigg ( G ( O(\alpha t) x-y, s)  - G(x,s) \bigg ) \frac{\dd s}{8 s^3} \dd t  \bigg | \nonumber \\
& \quad + \bigg | \int_{0}^{\infty}  \int_{t}^{\infty} 
\bigg ( (\cos 2\alpha t) D_1 + (\sin 2\alpha t) D_2  + D_3  \bigg )  \,  G (x, s) \frac{\dd s}{8 s^3} \dd t  \bigg | \nonumber \\
& \leq C \frac{|y|}{|x|^2} + C \min \big \{ \frac{1}{|\alpha| |x|^3}, \frac{1}{|x|} \big \}\,,
\qquad |x| > 2 |y|\,.
\label{proof.lem.thm.linear.whole.1.22}
\end{align}
\noindent
Here, we have again applied \eqref{est.lem.thm.linear.whole.3.1} for the first term and \eqref{est.lem.thm.linear.whole.3.3} for the second term to derive the last line. Finally we have 
\begin{align}
& \bigg |  \int_{0}^{\infty}  \int_{t}^{\infty} (O(\alpha t)x-y)_{1} 
\bigg (  \frac{\cos{2\alpha t}}{2} A_3  + \frac{\sin{2\alpha t}}{2}  A_4  \bigg) \, G( O (\alpha t) x - y, s) \frac{\dd s}{8 s^3} \dd t \bigg |  \nonumber \\
& \leq C (|x|+|y|) |x| |y| \int_0^\infty \int_t^\infty e^{-\frac{|x|^2}{16 s}} \frac{\dd s}{s^4} \dd t \, \leq \, C \frac{|y|}{|x|^2}\,, \qquad  |x| > 2 |y|\,.\label{proof.lem.thm.linear.whole.1.23}
\end{align}
Collecting \eqref{proof.lem.thm.linear.whole.1.19}, \eqref{proof.lem.thm.linear.whole.1.20}, \eqref{proof.lem.thm.linear.whole.1.22}, and \eqref{proof.lem.thm.linear.whole.1.23}, we have shown that 
\begin{align}
\big| \partial_{y_1} \big ( \tilde \Gamma_\alpha^{112}  (x,y)  - L^{112} (x,y) \big ) \big|  \leq C \bigg ( \frac{|y|}{|x|^2} +  \min \big \{ \frac{1}{|\alpha| |x|^3}, \frac{1}{|x|} \big \}  \bigg )\,,
\quad  |x| > 2 |y|\,. \label{proof.lem.thm.linear.whole.1.24}
\end{align}
The estimate of $\partial_{y_2} \big ( \tilde \Gamma_\alpha^{112}  (x,y)  - L^{112} (x,y) \big )$ is obtained in the similar manner. Thus, from \eqref{proof.lem.thm.linear.whole.1.17} and \eqref{proof.lem.thm.linear.whole.1.24} we have obtained the estimates of the derivatives in $y$ for $\Gamma_\alpha^{11} (x,y)$. 
The proof of Lemma \ref{lem.thm.linear.whole.1} is complete.
\end{proofx}

\begin{proofx}{Theorem \ref{thm.linear.whole}} 
The assertion that $u=L[f]$ is a weak solution to \eqref{a} (whose definitions are stated in the beginning of this subsection) follows from a similar argument as in \cite[Proposition 3.2]{H3}. So we omit the details on this part  and we focus on the proof for the estimates of $u$ here. (i) Let $\gamma \in [0,1)$. Suppose that ${\rm supp}\, f\subset B_{R}$ for some $R\ge1 $. Note that $\frac{(y^\bot \cdot f(y)) x^\bot}{4\pi |x|^2} = L(x,y) f(y)$ holds. Let $|x|\geq 2 R$. Then we have from Lemma \ref{lem.thm.linear.whole.1} with $m=0$,
\begin{align}
& \bigg| \int_{\R^2} \Gamma_\alpha (x,y) f \dd y - c[f] \frac{x^\bot}{4\pi |x|^2} \bigg|  \nonumber \\
& \,=\,  \bigg| \int_{|y|\leq R} \big ( \Gamma_\alpha (x,y) - L (x,y) \big ) f(y) \dd y \bigg| \nonumber \\
& \leq C \int_{|y|\leq R}  \bigg (  \min \big\{ \frac{1}{|\alpha| |x|^2}, \frac{1}{|\alpha|^{\frac{1}{2}} |x|} \big\}
+ |x|  \min \big\{ \frac{1}{|\alpha| |x|^3}, \frac{1}{|x|} \big\}  +  \frac{|y|^{2}} {|x|^2} \bigg ) 
|f(y)| \dd y\nonumber \,,
\end{align}
which implies $L[f] (x) = c[f] \frac{x^\bot}{4\pi |x|^2} + \mathcal{R}[f](x)$ with 
\begin{align}
\begin{split}
|x|^{1+\gamma} \, |\mathcal{R}[f](x)| 
& \leq 
C \bigg( 
\min \big\{ \frac{1}{|\alpha| |x|^{1-\gamma}}, \frac{|x|^\gamma}{|\alpha|^{\frac{1}{2}}} \big\}  
\| f\|_{L^1 (B_{R})}  \\
& \quad  
+ \min \big \{ \frac{1}{|\alpha| |x|^{1-\gamma}}, 
|x|^{1+\gamma} \big \} 
\| f \|_{L^1 (B_{R})}
+ \| |y|^{1+\gamma} f \|_{L^1 (B_{R})} 
\bigg ) \,.\label{proof.thm.linear.whole.1}
\end{split}
\end{align}
Here $C$ is independent of $x$, $R$, $\alpha$, $\gamma$, and $f$. Then we use the inequality for $\gamma\in [0,1)$,
\begin{align}\label{proof.thm.linear.whole.2}
\min \big\{ \frac{1}{|\alpha| |x|^{1-\gamma}}, \frac{|x|^\gamma}{|\alpha|^{\frac{1}{2}}} \big\} \leq |\alpha|^{-\frac{1+\gamma}{2}}\,,\quad\quad  
\min \big \{ \frac{1}{|\alpha| |x|^{1-\gamma}}, |x|^{1+\gamma} \big \} \leq |\alpha|^{-\frac{1+\gamma}{2}}\,,
\end{align}
\noindent
which leads to \eqref{est.thm.linear.whole.2}.

\noindent (ii) Let $\gamma\in [0,1)$ 
and write $\Gamma_\alpha (x,y) = \big ( \Gamma_\alpha (x,y)_{ij}\big )_{1\leq i,j\leq 2}$ and $L(x,y) = (L(x,y)_{ij})_{1\leq i,j\leq 2}$. From the integration by parts we see for $k=1,2$ and $f=(\sum_{l=1,2}\partial_l F_{1l}, \sum_{l=1,2}\partial_l F_{2l})^\top$,
\begin{align*}
&~~~ \int_{\R^2} e^{-\epsilon |y|^2} (\Gamma_\alpha (x,y) f )_k \dd y  \, = \, \sum_{j=1,2} \int_{\R^2} e^{-\epsilon |y|^2} \Gamma_\alpha (x,y)_{kj} f_j \dd y\\
& \, = \, - \sum_{j=1,2} \sum_{l=1,2} \int_{\R^2} e^{-\epsilon |y|^2}  \partial_{y_l} \Gamma_\alpha (x,y )_{kj}  F_{jl} \dd y 
+ 2\epsilon \sum_{j=1,2} \sum_{l=1,2} \int_{\R^2} e^{-\epsilon |y|^2} y_l \Gamma_\alpha (x,y)_{kj} F_{jl} \dd y\\
& \, = \, - \sum_{j=1,2} \sum_{l=1,2} \int_{\R^2} e^{-\epsilon |y|^2}  \partial_{y_l} \big ( \Gamma_\alpha (x,y )_{kj} -  L (x,y)_{kj} \big )  F_{jl} \dd y \\
& \quad  -\sum_{j=1,2} \sum_{l=1,2}  \int_{\R^2} e^{-\epsilon |y|^2}  \partial_{y_l}  L (x,y)_{kj}   \, F_{jl} \dd y  + 2\epsilon \int_{\R^2} e^{-\epsilon |y|^2} (\Gamma_\alpha (x,y) \, F \, y )_k \dd y \,.
\end{align*} 
Note that  
\begin{align*}
\big (- \sum_{j=1,2} \sum_{l=1,2} \partial_{y_l} L (x,y) _{1j}  \, F_{jl}\,, \, \, - \sum_{j=1,2} \sum_{l=1,2} \partial_{y_l} L (x,y) _{2j} \, F_{jl} \big )^\top \, = \, (F_{12}-F_{21})\frac{x^\bot}{4\pi |x|^2}
\end{align*}
by the definition of $L(x,y)$. Moreover, we have $|\Gamma_\alpha (x,y)| \leq \frac{C(\alpha, |x|)}{|y|}$ for $|y|>2|x|$ by \cite[Proposition 3.1]{H3}, and $\int_{|y|\leq 2|x|} |\Gamma_\alpha (x,y)| \dd y\leq C' (\alpha, |x|)<\infty$ by \cite[Lemma 3.3]{H3}, which implies 
\begin{align*}
\lim_{\epsilon\rightarrow 0} \epsilon \int_{\R^2} e^{-\epsilon |y|^2} \Gamma_\alpha (x,y) \, F \, y \dd y \, = \, 0
\end{align*}
for $F\in L^\infty_{2+\gamma} (\R^2)^{2\times 2}$. For simplicity we use the next notations:
\begin{align*}
&\nabla_y \Gamma_\alpha (x,y) \, F
\,=\,
\big ( \sum_{j=1,2} \sum_{l=1,2} \partial_{y_l} \Gamma_\alpha (x,y) _{1j}  \, F_{jl}\,, \, \, 
\sum_{j=1,2} \sum_{l=1,2} \partial_{y_l} \Gamma_\alpha (x,y) _{2j} \, F_{jl} \big )^\top \,, \\
&\nabla_y L (x,y) \, F 
\,=\,
\big ( \sum_{j=1,2} \sum_{l=1,2} \partial_{y_l} L(x,y) _{1j}  \, F_{jl}\,, \, \, 
\sum_{j=1,2} \sum_{l=1,2} \partial_{y_l} L(x,y) _{2j} \, F_{jl} \big )^\top \,.
\end{align*}
Then we have 
\begin{align}\label{proof.thm.linear.whole.3}
\begin{split}
L[f] (x) 
& \, = \, -\int_{\R^2} \nabla_y \Gamma_\alpha (x,y) \, F (y) \dd y\\
& \, = \,  - \int_{|y|<\frac{|x|}{2}}   \nabla_y \big ( \Gamma_\alpha (x,y ) -  L (x,y) \big )  F(y) \dd y -  \int_{|y| \ge \frac{|x|}{2} } \nabla_{y} \Gamma_{\alpha}(x,y) \, F(y) \dd y \\
& ~~~~~ - \lim_{\epsilon\rightarrow 0} \int_{|y|\geq \frac{|x|}{2}} e^{-\epsilon |y|^2} \big(F_{12}(y) - F_{21}(y) \big) \dd y \frac{x^\bot}{4\pi |x|^2} +\tilde c[F] \frac{x^\bot}{4\pi |x|^2}\,.
\end{split}
\end{align}
The sum of the first three terms of the right-hand side of this equality is denoted by $\mathcal{R}[f]$.
To estimate $\mathcal{R}[f]$ we firstly observe from Lemma \ref{lem.thm.linear.whole.1},
\begin{align}
& \bigg| \int_{|y| < \frac{|x|}{2} }  \nabla_{y} \big(\Gamma_{\alpha}(x,y) - L(x,y) \big) F(y) \dd y \bigg |  \nonumber \\
& \le C \bigg ( \frac{1}{|x|^2} \int_{|y| < \frac{|x|}{2} }  |y \, F(y)| \dd y 
 + \min \big\{ \frac{1}{|\alpha||x|^3}, \frac{1}{|x|} \big\} \int_{|y|<\frac{|x|}{2}} |F(y)| \dd y \bigg )\,, \qquad x\ne 0 \,.\label{proof.thm.linear.whole.4}
\end{align}
Next we have from the direct calculation 
\begin{equation*}
|(\nabla_x K)(x,t)| \leq  C \big( t^{-\frac{3}{2}} e^{-\frac{|x|^2}{16t}} 
+  \int_{t}^{\infty} s^{-\frac{5}{2}} e^{- \frac{|x|^2}{16s}} \dd s \big)\,,
\end{equation*} 
which implies 
\begin{align*}
\int_0^\infty |(\nabla K) (O(\alpha t) x,t )| \dd t\leq \frac{C}{|x|}\,,~~~~\quad ~x\ne 0\,.
\end{align*}
Then by the transformation of the variables $y = O(\alpha t)z$ we have 
\begin{align}
&~~~ 
\bigg|  \int_{|y| \ge \frac{|x|}{2} } \nabla_{y} \Gamma_{\alpha}(x,y) F(y) \dd y\bigg|  \nonumber \\
& \le  
\int_{|y| \ge \frac{|x|}{2} }  \bigg(  \int_{0}^{\infty}  |(\nabla K)(O(\alpha t)x - y, t)| \dd t \bigg) |F(y)| \dd y \nonumber \\
& \le 
\|F\|_{L^\infty_{2+\gamma} (B^{{\rm c}}_{\frac{|x|}{2}})} 
\int_{|z| \ge \frac{|x|}{2} } \bigg( \int_{0}^{\infty}  |(\nabla K)(O(\alpha t)(x - z), t)| \dd t  \bigg) |z|^{-2-\gamma} \dd z \nonumber \\
& \le  
C \|F\|_{L^\infty_{2+\gamma} (B^{{\rm c}}_{\frac{|x|}{2}})} 
\int_{|z| \geq  \frac{|x|}{2}} |x-z|^{-1} |z|^{-2-\gamma} \dd z\nonumber \\
& \leq 
\frac{C}{|x|^{1+\gamma}} \|F\|_{L^\infty_{2+\gamma} (B^{{\rm c}}_{\frac{|x|}{2}})} \,.
\label{proof.thm.linear.whole.5}
\end{align}
Here $C$ is independent of $x$ and $\gamma \in [0,1)$.
Collecting \eqref{proof.thm.linear.whole.3}, \eqref{proof.thm.linear.whole.4}, and \eqref{proof.thm.linear.whole.5}, we obtain  \eqref{est.thm.linear.whole.3} and \eqref{est.thm.linear.whole.4}. The proof of Theorem \ref{thm.linear.whole} is complete.
\end{proofx}

Based on the results of Theorem \ref{thm.linear.whole} we study the exterior problem \eqref{Sa} in the next subsection, where its asymptotic profile is represented as the solution to \eqref{a} by a cut-off technique. However, the existence of solutions to \eqref{Sa} decaying at spatial infinity has to be proved carefully. As in \cite{H3}, for the exterior problem, a natural way to construct solutions decaying at spatial infinity is to consider first a regularized system and to take the limit; see the proof of Theorem \ref{thm.linear.exterior} for details. In this procedure we need to consider the following system in the whole space:
\begin{equation}\tag{S$^{\lambda}_{\alpha,\R^2}$}\label{RSea}
  \left\{\begin{aligned}
 \lambda u_{\lambda} - \Delta u_{\lambda}  - \alpha ( x^\bot \cdot \nabla u_{\lambda} - u_{\lambda}^\bot ) + \nabla p_{\lambda} &\,=\, f \,,  ~~~  {\rm div}\, u_{\lambda} \,=\, 0\,, ~~~~~~ x \in \R^2\,, \\
 u_{\lambda}  & \, \rightarrow \,     0   \,,  ~~~~~~ |x| \rightarrow \infty\,, 
\end{aligned}\right.
\end{equation}
where $\lambda$ is a small positive number. Let us introduce the integral kernel $\Gamma_{\alpha}^{\lambda}(x,y)$ as 
\begin{equation}\label{def.Gamma_epsilon}
\Gamma_{\alpha}^{\lambda}(x,y)
\, =\, \int_{0}^{\infty} e^{-\lambda t} O(\alpha t)^{\top}K(O(\alpha t)x-y,t) \dd t\,,~~~~~~ x\ne y\,.
\end{equation}
In virtue of the positive $\lambda$, the integral in \eqref{def.Gamma_epsilon} converges absolutely for $x\ne y$. Furthermore,  the velocity $u_\lambda$ defined by
\begin{align}
u_\lambda (x) \, = \, \int_{\R^2} \Gamma_{\alpha}^{\lambda}(x,y) f(y) \dd y\,, \qquad 
f \in L^2 (\R^2)^2\,, 
\label{regular.u}
\end{align}
satisfies \eqref{RSea} in the sense of distributions with a suitable pressure $\nabla p_\lambda$. 
The next lemma will be used in the proof of Theorem \ref{thm.linear.exterior}.

\begin{lemma}\label{lem.thm.linear.exterior.1}
Let $ \alpha \in \R \setminus \{0\}$ and $\gamma \in [0,1)$.
Suppose that $f \in L^2 (\R^2)^2$ is of the form $f = {\rm div}\,F$
with some $F\in L^\infty_{2+\gamma} (\R^2)^{2\times 2}$. 
Then for any $\theta \in (0,1)$ and $R\ge1$, the velocity $u_\lambda$ defined by \eqref{regular.u} satisfies 
\begin{equation}\label{est.lem.thm.linear.exterior.1.1}
\| u_{\lambda} \|_{L^{\infty}_{\theta}(B^{{\rm c}}_{2R})}
\leq C 
\big( \|F\|_{L^{\infty}_{2+\gamma}(B^{{\rm c}}_{R})} + \| F \|_{L^1 (B_{R})} \big) \,.
\end{equation}
Here the constant $C$ is independent of $\lambda$ and $\gamma$, and depends only on $\theta$ and $R$.
\end{lemma}

\begin{proof}
In the same way as in the proof of Lemma \ref{lem.thm.linear.whole.1},
we define $L^{\lambda} = L^{\lambda}(x,y)$ by
\begin{equation*}
L^{\lambda}(x,y) 
\,=\,  L^{\lambda, 0}(x,y) + L^{\lambda, 111}(x,y) 
+ L^{\lambda, 112}(x,y) + L^{\lambda, 12}(x,y)\,,
\end{equation*}
where
\begin{equation*}
\begin{split}
& L^{\lambda,0}(x,y) \, = \, 
\int^{\infty}_0 
e^{-\lambda t}
G(x,t) \frac{\dd t}{4 t} 
\begin{pmatrix}
x\cdot y  & x^{\bot}\cdot y \\
-x^{\bot}\cdot y & x\cdot y \\
\end{pmatrix}\,, \\
& L^{\lambda,111}(x,y) \,= \, 
\int_{0}^{\infty} \int^{\infty}_t 
e^{-\lambda t}
G(x,s)  \frac{\dd s}{4s^2} \dd t
\, \bigg( \frac{-3(x\otimes y) +(x^{\perp}\otimes y^{\perp})}{2} \bigg)\,, \\
& L^{\lambda,112}(x,y) \, = \, 
 \int^{\infty}_0 \int^{\infty}_t 
e^{-\lambda t}
G(x,s) \frac{\dd s}{16s^3} \dd t
\, |x|^2 (x\otimes y)\,, \\
& L^{\lambda,12}(x,y) \, =\, 
- \int^{\infty}_0 \int^{\infty}_t 
e^{-\lambda t}
G(x,s) \frac{\dd s}{8s^2} \dd t
\begin{pmatrix}
x\cdot y & x^{\bot} \cdot y \\
-x^{\bot} \cdot y & x\cdot y \\
\end{pmatrix}\,. \\
\end{split}
\end{equation*}
Then we have
\begin{align}
| \nabla_{y} L^{\lambda} (x,y) | 
& \le 
C |x| 
\bigg(
\int_{0}^{\infty} 
e^{-\frac{|x|^2}{4t}} \frac{\dd t}{t^2} 
+ \int_{0}^{\infty} \int_{t}^{\infty} 
e^{-\frac{|x|^2}{4s}} \frac{\dd s}{s^3} \dd t
+ |x|^2 \int_{0}^{\infty}  \int_{t}^{\infty} 
e^{-\frac{|x|^2}{4s}} \frac{\dd s}{s^4} \dd t
\bigg) \nonumber \\
& \le
\frac{C}{|x|}\,,  \qquad|x| > 0\,, 
\label{proof.lem.thm.linear.exterior.1.1}
\end{align}
where the constant $C$ is independent of $\alpha$ and $\lambda$. By the integration by parts we rewrite $u_\lambda$ as
\begin{equation}\label{proof.lem.thm.linear.exterior.1.2}
\begin{split}
u_\lambda (x) & \, =\, - \int_{\R^2} \nabla_y \Gamma_\alpha^\lambda (x,y) \, F (y) \dd y \\ 
& \, =\, 
- \int_{|y| < \frac{|x|}{2}} \nabla_{y} \big( \Gamma^{\lambda}_{\alpha}(x,y) - L^{\lambda}(x,y) \big) F(y) \dd y 
-  \int_{|y| \geq \frac{|x|}{2}}  \nabla_{y} \Gamma^{\lambda}_{\alpha}(x,y) F(y) \dd y \\
& ~~~~~ - \int_{|y| < \frac{|x|}{2}} \nabla_{y} L^{\lambda}(x,y) F(y) \dd y\,.
\end{split}
\end{equation}
Then, proceeding as in the proof of Lemma \ref{lem.thm.linear.whole.1}, we obtain
\begin{equation}\label{proof.lem.thm.linear.exterior.1.3}
|\nabla_{y} \big(\Gamma_{\alpha}^{\lambda}(x,y) - L^{\lambda}(x,y) \big)| 
\le
C \bigg( \frac{|y|}{|x|^2} + \min \big \{ \frac{1}{|\alpha||x|^3}, \frac{1}{|x|} \big \} \bigg)\,,
\qquad |x| > 2 |y|\,, 
\end{equation}
where $C$ is independent of $x$, $y$, $\alpha$, and $\lambda$. Then we have
\begin{equation}\label{proof.lem.thm.linear.exterior.1.4}
\begin{split}
\bigg| \int_{|y| < \frac{|x|}{2}} \nabla_{y} \big( \Gamma^{\lambda}_{\alpha}(x,y) - L^{\lambda}(x,y) \big) F(y) \dd y \bigg| 
& \leq \frac{C}{|x|} \| F \|_{L^1 (B_{ \frac{|x|}{2} }) } \\
& \leq \frac{C \log (2 + |x|) }{|x|} \| F \|_{L^\infty_{2+\gamma} (\R^2)}\,, \quad |x| > 1\,,
\end{split}
\end{equation}
where the constant $C$ is independent of $\lambda$ and $\gamma$. The second term in the right-hand side of \eqref{proof.lem.thm.linear.exterior.1.2} is also estimated as in the proof of Lemma \ref{lem.thm.linear.whole.1}, resulting the estimate
\begin{equation}\label{proof.lem.thm.linear.exterior.1.5}
\bigg | \int_{ |y| \geq \frac{|x|}{2} } \nabla_{y} \Gamma^{\lambda}_{\alpha}(x,y) F(y) \dd y \bigg | 
\leq \frac{C}{|x|^{1+\gamma}}  
\|F\|_{L^\infty_{2+\gamma} (B^{{\rm c}}_{\frac{|x|}{2}})}\,.
\end{equation}
For the last term in the right-hand side of \eqref{proof.lem.thm.linear.exterior.1.2} it is straightforward from \eqref{proof.lem.thm.linear.exterior.1.1} to see
\begin{equation}\label{proof.lem.thm.linear.exterior.1.6}
\bigg | \int_{ |y| <  \frac{|x|}{2}} \nabla_{y} L^{\lambda} (x,y) F \dd y \bigg |
\leq 
\frac{C \log (2 + |x|) }{|x|} \| F \|_{L^\infty_{2+\gamma} (\R^2)}\,, \qquad |x| >1\,.
\end{equation}
Collecting \eqref{proof.lem.thm.linear.exterior.1.4}, \eqref{proof.lem.thm.linear.exterior.1.5}, and \eqref{proof.lem.thm.linear.exterior.1.6}, we obtain \eqref{est.lem.thm.linear.exterior.1.1}.
This completes the proof.
\end{proof}

\subsection{Linear estimate in the exterior domain}\label{subsec.linear.exterior}

In this subsection we study the asymptotic estimates for solutions to the Stokes system in the exterior domain
\begin{equation}\tag{S$_{\alpha}$}\label{Sa}
  \left\{\begin{aligned}
  -\Delta u - \alpha ( x^\bot \cdot \nabla u - u^\bot ) + \nabla p  &\,=\, f \,, 
  \quad {\rm div}\, u \,=\, 0\,,  ~~~~~~   x \in \Omega\,, \\
 u  & \, =\,  0   \,,  ~~~~~~ x\in \partial\Omega\,, \\
 u  & \, \rightarrow \,     0   \,,  ~~~~~~ |x| \rightarrow \infty\,,
\end{aligned}\right.
\end{equation}
where $\alpha \in \R\setminus \{0\}$ is a given constant. In the following,
we fix a positive number $R_0\geq 1$ large enough so that $\R^2\setminus \Omega\subset B_{R_0}$ holds. 
We also fix a radial cut-off function $\varphi\in C_0^\infty (\R^2)$ such that $\varphi (x)=1$ for $|x|\leq R_0$ and $\varphi (x) =0$ for $|x|\geq 2R_0$. As in the previous subsection, for $f\in L^2(\Omega)^2$ and $F\in L^2 (\Omega)^{2\times 2}$ we formally set 
\begin{align}
\begin{split}
c_\Omega [f] & \, = \, 
\lim_{\epsilon \rightarrow 0} \int_{\Omega} e^{-\epsilon |y|^2} y^\bot \cdot f(y) \dd y\,,\label{def.c.tildec.exterior}\\
\tilde c_\Omega [F] & \, = \, 
\lim_{\epsilon \rightarrow 0} \int_{\Omega} e^{-\epsilon |y|^2} \big(F_{12}(y) - F_{21}(y) \big)\dd y\,.
\end{split}
\end{align}
These are well-defined at least when $f={\rm div}\, F$ with $F\in L^\infty_{2+\gamma} (\Omega)^{2\times 2}$ for some $\gamma>0$, and $c_\Omega [f]=\tilde c_\Omega [F]$ holds in this case if the generalized traces $\nu \cdot (x_2 \vec{F_1})$, $\nu \cdot (x_1 \vec{F_2})$ on $\partial\Omega$ are zero in addition. Here we have set $F=(\vec{F_1}, \vec{F_2})^\top$. Note that the coefficient $\tilde c_\Omega[F]$ is well-defined only under the condition $F_{12}-F_{21}\in L^1 (\Omega)$. In general, we have the following.

\begin{lemma}\label{lem.coefficient} Let $f\in L^2 (\Omega)^2$ be of the form $f={\rm div}\, F=(\sum_{j=1,2}\partial_j F_{1j}, \sum_{j=1,2}\partial_j F_{2j})^\top$ for some $F\in L^\infty_{2,0}(\Omega)^{2\times 2}$ and $F_{12}-F_{21}\in L^1 (\Omega)$. Then both $c_\Omega [f]$ and $\tilde c_\Omega [F]$ converge.
\end{lemma}

\begin{proof} It is trivial that $\tilde c_\Omega [F]$ converges. 
Let $\varphi\in C_0^\infty (\R^2)$ be a cut-off function introduced at the beginning of this subsection. The convergence of $c_\Omega[f]$ easily follows from the integration by parts: 
\begin{align}
\begin{split}
c_\Omega [f] 
& \, = \, 
\int_\Omega y^\bot \cdot f \varphi \dd y 
+ \lim_{\epsilon \rightarrow 0} 
\int_\Omega e^{-\epsilon |y|^2} (1 - \varphi) y^\bot \cdot f \dd y \\
& \, = \, 
\int_\Omega y^\bot \cdot f \varphi \dd y + \, \tilde c_\Omega [F] \, - \int_\Omega (F_{12} - F_{21}) \varphi \dd y \\
& \quad
+ \int_\Omega y^\bot  \cdot  F \nabla \varphi  \dd y 
+ \lim_{\epsilon \rightarrow 0} 
2 \int_\Omega e^{-\epsilon |y|^2} \epsilon y^\bot \cdot (F y) (1 - \varphi) \dd y \,.
\label{proof.lem.coefficient.1}
\end{split}
\end{align}
The last term in the right-hand side of  \eqref{proof.lem.coefficient.1} vanishes in virtue of the decay $|F(x)| = o (|x|^{-2})$ as $|x|\rightarrow \infty$. In fact, by extending $F$ to the whole space by zero we have
\begin{align*}
\begin{split}
\bigg| \int_\Omega e^{-\epsilon |y|^2} \epsilon y \cdot \big(F(y) y^\bot \big) (1 - \varphi) \dd y \bigg| 
& \leq
\int_{\R^2} e^{-\epsilon |y|^2} \epsilon |y|^2 |F(y)| \dd y \\
& \,=\,
\int_{\R^2} e^{-|z|^2} \big( \frac{|z|}{\epsilon^{\frac12}} \big)^2 \, \big|F \big(\frac{z}{\epsilon^{\frac12}} \big)\big| \dd z\,,
\end{split}
\end{align*}
where we have used the transformation of the variables $y = \epsilon^{-\frac12} z$. Then the Lebesgue dominated convergence theorem implies the right-hand side of the above inequality goes to zero as $\epsilon \rightarrow 0$. In particular, we have 
\begin{align}
c_\Omega [f] 
\, = \, 
\tilde c_\Omega [F] 
+ \int_\Omega \big\{ \big( y^\bot \cdot f  - F_{12} + F_{21} \big) \varphi  + y^\bot \cdot F \nabla \varphi  \big\} \dd y\,.\label{proof.lem.coefficient.2}
\end{align} 
The proof is complete.
\end{proof}

Let us denote by $T(u,p)$ the stress tensor, which is defined as 
\begin{align}
T(u,p) \, = \, D u - p \mathbb{I}\,,~~~~~ D u = \nabla u + (\nabla u )^\top\,,~~~~~\mathbb{I} = (\delta_{j k})_{1\leq j,k\leq 2}\,.
\end{align}
The next lemma is a counterpart of \cite[Theorem 2.1]{H3} in our functional setting.
We denote by $\Omega_r$ the truncated domain defined as $\Omega_r = \{x\in \Omega~|~|x|<r\}$ for $r>0$.

\begin{lemma}\label{lem.linear.exterior}
Let $\alpha \in \R \setminus \{0\}$ and $\gamma \in [0,1)$. Assume that  $f\in L^2 (\Omega)^2$ is of the form $f = {\rm div}\,F$ with some $F\in L^\infty_{2+\gamma} (\Omega)^{2\times 2}$, and that $\tilde c_\Omega [F]$ converges when $\gamma=0$.
Suppose that $(u, \nabla p)\in W^{2,2}_{loc} (\overline{\Omega})^2 \times L^2_{loc} (\overline{\Omega})^2$ is a solution to the system \eqref{Sa} satisfying  $\| \nabla u \|_{L^2 (\Omega)}<\infty$ and $\displaystyle \lim_{|x|\rightarrow \infty} |u (x)|=0$.
Then $u$ is represented as 
\begin{equation}\label{est.lem.linear.exterior.1}
u(x) \, = \beta \frac{x^{\bot}}{4\pi|x|^2} + \mathcal{R}(x)\,,  \qquad x\in \Omega\setminus \{0\}\,,
\end{equation}
where  
\begin{align}\label{est.lem.linear.exterior.3}
\begin{split}
\beta & \, = \,  \int_{\partial \Omega} y^\bot \cdot \big( T(u,p)\nu \big) \dd \sigma_{y} \,+\, b_\Omega [f]\,,\\
b_\Omega [f] & \, = \, \tilde c_\Omega [F] + \int_\Omega \big\{ \big ( y^\bot \cdot f  - F_{12} + F_{21} \big ) \varphi + y^\bot \cdot F \nabla \varphi  \big \} \dd y\,,
\end{split}
\end{align}
while $\mathcal{R}$ satisfies
\begin{align}
\begin{split}
\| \mathcal{R} \|_{L^{\infty}_{1+\gamma} (B^{{\rm c}}_{4R_{0}})}
& \le 
C \bigg ( 
\| F \|_{L^{\infty}_{2+\gamma} (B^{{\rm c}}_{2R_{0}})}
+ \sup_{|x|\geq 4 R_0} |x|^{-1 +\gamma} \| y F \|_{L^1 (\Omega_{ \frac{|x|}{2} })}  \\
&  \qquad \qquad 
+ \sup_{|x|\geq 4 R_0} \min \big \{ \frac{1}{|\alpha| |x|^{2-\gamma}}, |x|^\gamma \big \} 
\| F \|_{L^1 (\Omega_{ \frac{|x|}{2} })}  \\
& \qquad \qquad  \qquad 
+ \sup_{|x|\geq 4 R_0} |x|^\gamma \, \big | \lim_{\epsilon \rightarrow 0}  \int_{2 |y|\geq |x|} e^{-\epsilon |y|^2} (F_{12}-F_{21} ) \dd y \big | \bigg ) \\
& \quad  + C \big ( |\alpha|^{-\frac{1+\gamma}{2}} 
+ |\alpha|^{-\frac12} + 1 \big ) \big ( \| F\|_{L^2 (\Omega_{2 R_0})} + (1+|\alpha|) \| \nabla u \|_{L^2 (\Omega_{2 R_0})} \big )\,.  \label{est.lem.linear.exterior.2}
\end{split}
\end{align}
Here the constant $C$ is independent of $\gamma$, $\alpha$, and $F$. 
The coefficient $b_\Omega [f]$ coincides with $c_\Omega [f]$ when $F$ belongs in addition to $L^\infty_{2,0} (\Omega)^{2\times 2}$.
\end{lemma}
%
\begin{proof}
We may assume that $\int_{\Omega_{2 R_0}} p \dd x =0$. Let $\varphi\in C_0^\infty (\R^2)$ be a cut-off function introduced at the beginning of this subsection. We introduce the Bogovskii operator $\mathbb{B}$ in the closed annulus $A=\{x\in \R^2~|~ R_0\leq |x|\leq 2 R_0\}$, and set 
\begin{align*}
v \,=\, (1-\varphi )u + \mathbb{B} [\nabla \varphi \cdot u]\,, \qquad q \,=\, (1-\varphi ) p \,.
\end{align*}
Note that $\mathbb{B}[\nabla \varphi \cdot u]$ satisfies 
\begin{align}
{\rm supp}\, \mathbb{B}[\nabla \varphi \cdot u] \subset A\,, \qquad 
{\rm div}\, \mathbb{B} [\nabla \varphi \cdot u] \,=\, \nabla \varphi \cdot u\,,
\label{proof.lem.linear.exterior.1}
\end{align}
and the estimates
\begin{equation}
\|\mathbb{B} [\nabla \varphi \cdot u]\|_{W^{m+1,2} (\Omega)} 
\leq C \| \nabla \varphi \cdot u \|_{W^{m,2} (\Omega)}\,, \qquad m=0,1\,.
\label{proof.lem.linear.exterior.2}
\end{equation}
See, e.g. Borchers and Sohr \cite{BS}. Then $(v, \nabla q)$ satisfies
\begin{equation}\label{proof.lem.linear.exterior.3}
 -\Delta v - \alpha ( x^\bot \cdot \nabla v - v^\bot ) + \nabla q  \,=\,  {\rm div}\, \mathcal{F} + g\,, 
\quad  {\rm div}\, v \,=\, 0\,,  \qquad   x \in \R^2\,, 
\end{equation}
where $\mathcal{F}$ and $g$ are the functions on $\R^2$ given by
\begin{align*}
\mathcal{F} &\,=\, (1-\varphi ) F - \nabla \mathbb{B}[\nabla \varphi \cdot u]\,,\\
g & \,=\, F \cdot \nabla \varphi + 2 \nabla \varphi \cdot \nabla u + (\Delta \varphi + \alpha x^\bot \cdot \nabla \varphi ) u \\
& \quad 
- \alpha \big ( x^\bot \nabla  \mathbb{B}[\nabla \varphi \cdot u] - \mathbb{B}[\nabla \varphi \cdot u]^\bot  \big ) - ( \nabla \varphi ) p \,.
\end{align*}
Note that ${\rm supp}\, g \subset A$ due to \eqref{proof.lem.linear.exterior.1}.  Recalling the uniqueness result stated in Remark \ref{rem.thm.linear.whole},  we find
\begin{align}
u(x)  
\,=\, v(x) 
&\, = \, L[{\rm div}\, \mathcal{F}] + L[g]  \nonumber \\
& \, = \,  \big( \tilde c [\mathcal{F}] + c[g] \big) \frac{x^{\bot}}{4\pi |x|^2}  + \mathcal{R}(x) \,,  
\qquad |x|\geq 4 R_0\,,
\label{proof.lem.linear.exterior.4}
\end{align}
where  $\tilde c [\mathcal{F}]$ and $c[g]$ are defined in \eqref{def.c.tildec}. Recalling that $R_0\geq 1$, we see from Theorem \ref{thm.linear.whole} that $\mathcal{R}(x)$ satisfies
\begin{equation*}
\begin{split}
\| \mathcal{R} \|_{L^{\infty}_{1+\gamma} (B^{{\rm c}}_{4R_{0}})}
& \le
C \big( 
\| \mathcal{R} [{\rm div}\, \mathcal{F} ] \|_{L^{\infty}_{1+\gamma} (B^{{\rm c}}_{4R_{0}})}
+ \| \mathcal{R} [g] \|_{L^{\infty}_{1+\gamma} (B^{{\rm c}}_{4R_{0}})}
\big) \\
& \le 
C \bigg ( 
\| F \|_{L^{\infty}_{2+\gamma}(B^{{\rm c}}_{2R_{0}})}  + \sup_{|x|\geq 4 R_0} |x|^{-1 +\gamma} \| y F \|_{L^1 (\{2 R_0\leq |y|\leq \frac{|x|}{2}\})}  \\
&  \quad 
+ \sup_{|x|\geq 4 R_0} \min \big \{  \frac{1}{|\alpha| |x|^{2-\gamma}}, |x|^\gamma \big \} \| F \|_{L^1 (\{2R_0 \leq |y| \leq  \frac{|x|}{2} \})}  \\
& \qquad 
+ \sup_{|x|\geq 4 R_0} |x|^\gamma \, \big | \lim_{\epsilon \rightarrow 0}  \int_{2 |y|\geq |x|} e^{-\epsilon |y|^2} (F_{12}-F_{21} ) \dd y \big | \\
& \quad \qquad
+ \big( \sup_{|x|\geq 4 R_0} \min \big \{ \frac{1}{|\alpha| |x|^{2-\gamma}}, |x|^\gamma \big \}  + 1 \big) 
\| \mathcal{F} \|_{L^1 (B_{2R_0})} \bigg) \\
& \quad 
+ C (|\alpha|^{-\frac{1+\gamma}{2}} + 1) \|g\|_{L^1(B_{2R_0})}\,.
\end{split}
\end{equation*}
Here $C$ depends only on $R_0$. It is easy to see
\begin{align*}
\| \mathcal{F} \|_{L^1(B_{2R_0})} 
& \leq C \big ( \| F \|_{L^2 (\Omega_{2R_0})} + \| \nabla u \|_{L^2 (\Omega_{2 R_0})} \big )
\end{align*}
by applying  \eqref{proof.lem.linear.exterior.2} and the Poincar${\rm \acute{e}}$ inequality. Similarly, the function $g$ is estimated as 
\begin{align*}
\| g\|_{L^1(B_{2R_0})} 
& \leq C \big( \| F \|_{L^{2}(\Omega_{2 R_0})} + (1 + |\alpha|) \| \nabla u \|_{L^{2}(\Omega_{2 R_0})}
+ \| p \|_{L^{2}(\Omega_{2 R_0})}  \big) \,.
\end{align*}
In order to estimate the pressure term let us recall the condition
$\int_{\Omega_{2 R_0}} p \dd x = 0$, which yields from \eqref{Sa},
\begin{equation*}
\begin{split}
\| p \|_{L^2 (\Omega_{2 R_0})} 
\leq C \| \nabla p \|_{H^{-1} (\Omega_{2 R_0})}
& \,=\, C \| {\rm div}\, [ F + \nabla u 
+ \alpha ( u \otimes x^\bot - x^\bot \otimes u ) ] \|_{H^{-1} (\Omega_{2 R_0})} \\
& \le C \big ( \| F \|_{L^2 (\Omega_{2 R_0})} + (1 + |\alpha|) \| \nabla u \|_{L^{2}(\Omega_{2 R_0})} \big ) \,,
\end{split}
\end{equation*}
where $H^{-1} (\Omega_{2 R_0})$ is the topological dual of $W^{1,2}_0 (\Omega_{2 R_0})$.
Collecting these estimates, we obtain \eqref{est.lem.linear.exterior.2}.

Finally let us determine the coefficient $\beta$ in \eqref{est.lem.linear.exterior.1}. In view of \eqref{proof.lem.linear.exterior.4} it suffices to compute $\tilde c[\mathcal{F}] + c[g]$. We follow the argument in the proof of \cite[Theorem 2.1]{H3}. Fix $N\geq 2R_0$ and let $\phi_N\in C_0^\infty (\R^2)$ be a radial cut-off function such that $\phi_N(x)=1$ for $|x|\leq N$ and $\phi_N (x) =0$ for $|x|\geq 2N$. Then we have 
\begin{align}
\tilde c[\mathcal{F}] + c[g]  
& \, = \, \lim_{\epsilon \rightarrow 0} \int_{\R^2} e^{-\epsilon |y|^2} ( F_{12}-F_{21}) (1-\phi_N) \dd y  \nonumber \\
& ~~~ +  \int_{\R^2}  ( \mathcal{F}_{12} - \mathcal{F}_{21}) \phi_N \dd y + \int_{\R^2} y^\bot \cdot g \phi_N \dd y \nonumber \\
& \, = \, \tilde c_\Omega [F] - \int_\Omega (F_{12}-F_{21}) \phi_N \dd y +   \int_{\R^2}  ( \mathcal{F}_{12} - \mathcal{F}_{21}) \phi_N \dd y + \int_{\R^2} y^\bot \cdot g \phi_N \dd y\,.\label{proof.lem.linear.exterior.5}
\end{align}
We set $S(v, q)(x) = T(v, q)(x) + \alpha(v \otimes x^\bot - x^\bot \otimes v)$.
Since ${\rm div}\, \mathcal{F}  + g = - {\rm div}\, S(v, q)=(-\sum_{j=1,2} \partial_j S_{1j} (v,q), -\sum_{j=1,2} \partial_j S_{2j} (v,q))^\top$ in $\R^2$, the integration by parts and the symmetry of $T(v,q)$ yield
\begin{align}
\int_{\R^2} y^\bot \cdot g \phi_N \dd y & \, = \,  -\int_{\R^2} \phi_N    y^\bot \cdot  {\rm div}\, S(v, q) \dd y  - \int_{\R^2}  \phi_N  y^\bot \cdot  {\rm div}\, \mathcal{F}   \dd y \nonumber \\
& \, = \, 
2 \int_{\R^2} \phi_N y \cdot v \dd y + \int_{\R^2} y^\bot \cdot S(v,q) \nabla \phi_N \dd y \nonumber \\
& \quad  
- \int_{\R^2} (\mathcal{F}_{12}-\mathcal{F}_{21}) \phi_N \dd y + \int_{\R^2} y^\bot  \cdot  \mathcal{F} \nabla \phi_N  \dd y \nonumber \\
& \, = \, \int_{\R^2} y^\bot \cdot S(v,q) \nabla \phi_N \dd y  \nonumber \\
& \quad  
- \int_{\R^2} (\mathcal{F}_{12}-\mathcal{F}_{21}) \phi_N \dd y + \int_{\R^2} y^\bot  \cdot  \mathcal{F} \nabla \phi_N  \dd y \,.\label{proof.lem.linear.exterior.6}
\end{align}
Here we have used the fact that $\phi_N$ is radial, and thus, $y \phi_N (y) = \nabla_y \big(\int_{|y|}^{\infty} r \tilde \phi_N (r) \dd r \big)$, where $\tilde \phi_N (r)$ is such that $\tilde \phi_N (|y|)=\phi_N(y)$.
Since $S(v, q) = S(u, p)$ for $|x| \geq 2 R_0$
and $-{\rm div}\,S(u, p) =  f$ in $\Omega$, again from the integration parts we have
\begin{align}
& \int_{\R^2} y^\bot \cdot S(v, q)  \nabla \phi_N \dd y \nonumber \\
& \, = \, \int_{\Omega} y^\bot \cdot S(u, p) \nabla \phi_N \dd y  \nonumber \\
& \, = \, \int_{\partial \Omega} y^\bot \cdot S (u, p) \nu \dd \sigma_{y} 
- 2\int_\Omega \phi_N y \cdot u \dd y + \int_{\Omega} \phi_N y^\bot \cdot f \dd y \nonumber \\
& \, = \, \, \int_{\partial \Omega} y^\bot \cdot T(u, p) \nu \dd \sigma_{y} + \int_{\Omega} \phi_N y^\bot \cdot f \dd y\,.\label{proof.lem.linear.exterior.7}
\end{align}
Here we have used the boundary condition $u=0$ on $\partial \Omega$ and also the radial symmetry of $\phi_N$. By taking the cut-off function $\varphi$ above, and using the relation $\varphi \phi_N = \varphi$, we then compute the second term in the above as
\begin{align}
\int_{\Omega} \phi_N y^\bot \cdot f \dd y 
& \, = \, 
\int_{\Omega} \varphi y^\bot \cdot f \dd y  
+ \int_{\Omega} \phi_N (1-\varphi) y^\bot \cdot f \dd y  \nonumber \\
\begin{split}
&  
\, = \, \int_{\Omega} \varphi y^\bot \cdot f \dd y 
+ \int_\Omega (F_{12}-F_{21} ) \phi_N  \dd y 
- \int_\Omega (F_{12}-F_{21} ) \varphi \dd y 
\\
& \quad - \int_\Omega  y^\bot \cdot F \nabla \phi_N \dd y 
+ \int_\Omega y^\bot \cdot F \nabla \varphi \dd y\,.\label{proof.lem.linear.exterior.8}
\end{split}
\end{align}
Collecting \eqref{proof.lem.linear.exterior.5} - \eqref{proof.lem.linear.exterior.8} and using $\mathcal{F}=F$ for $|x|\geq 2R_0$, we obtain 
\begin{align}
\begin{split}
\tilde c[F] + c[g]  & \, = \, \int_{\partial\Omega} y^\bot \cdot T(u,p) \nu \dd \sigma_y \\
&  \quad  
+ \tilde c_\Omega [F] 
+  \int_\Omega \big\{ (y^\bot \cdot f - F_{12} + F_{21} ) \varphi + y^\bot \cdot  F \nabla \varphi \big\} \dd y\,,\label{proof.lem.linear.exterior.9}
\end{split}
\end{align}
as desired. When $F \in L^\infty_{2,0} (\Omega)^{2\times 2}$ the coefficient $b_\Omega [f]$ coincides with $c_\Omega [f]$ in virtue of \eqref{proof.lem.coefficient.2}. The proof is complete.
\end{proof}

Let us recall that $R_0\geq 1$ is taken so that $\R^2\setminus \Omega \subset B_{R_0}$. Let $\varphi\in C_0^\infty (\Omega)$ be a radial cut-off function such that $\varphi (x) =1$ for $|x|\leq R_0$ and $\varphi (x)=0$ for $|x|\geq 2 R_0$. Then we set 
\begin{align}
V(x) \, = \, \big(1-\varphi (x) \big) \frac{x^\bot}{4\pi |x|^2}\,.\label{def.V}
\end{align}
Note that $V$ is a radial circular flow satisfying ${\rm div}\, V=0$, which describes the asymptotic behavior of solutions to the Stokes system \eqref{a} as is shown in Theorem \ref{thm.linear.whole}. The main result of this section is stated as follows.
\begin{theorem}\label{thm.linear.exterior}
Let $\alpha \in \R \setminus \{0\}$ and $\gamma \in [0,1)$. Suppose that $f \in L^2 (\Omega)^2$ is of the form $f={\rm div}\, F$ with $F\in L^\infty_{2+\gamma} (\Omega)^{2\times 2}$. Assume in addition that $\tilde c_\Omega [F]$ converges when $\gamma=0$. Then there exists a unique  solution $(u, \nabla p)\in W^{2,2}_{loc} (\overline{\Omega})^2 \times L^2_{loc} (\overline{\Omega})$ to \eqref{Sa} satisfying \,$\displaystyle\lim_{|x|\rightarrow \infty} |u(x)|=0$ and
\begin{align}
\| \nabla u \|_{L^2 (\Omega)} & \leq \| F \|_{L^2 (\Omega)}\,, \label{est.thm.linear.exterior.0}\\
\| p \|_{L^2 (\Omega_{6 R_0})} & \leq C (1+|\alpha|) \| F\|_{L^2 (\Omega)}\,, \label{est.thm.linear.exterior.1}\\
\| \nabla^2 u \|_{L^2 (\Omega_{k R_0})}  + \| \nabla p \|_{L^2 (\Omega_{k R_0})} 
&\leq C (1+|\alpha|) \big ( \| F\|_{L^2 (\Omega)} + \| f\|_{L^2 (\Omega_{(k+1) R_0})}\big )\,,~~~2\leq k\leq 5\,.
\label{est.thm.linear.exterior.2}
\end{align}
\noindent
Moreover, the velocity $u$ is written as  
\begin{equation}\label{est.thm.linear.exterior.3}
u(x)  \, =\, \beta V(x)  + \mathcal{R}_\Omega [f](x)\,,  \qquad x\in \Omega\,,
\end{equation}
where $\beta\in \R$ is given by 
\begin{align}\label{est.thm.linear.exterior.5}
\begin{split}
\beta & \, = \,  \int_{\partial \Omega} y^\bot \cdot \big( T(u,p)\nu \big) \dd \sigma_{y} + b_\Omega [f]\,,\\
b_\Omega [f] & \, = \, \tilde c_\Omega [F] 
+ \int_\Omega 
\big\{ \big ( y^\bot \cdot f  - F_{12} + F_{21} \big ) \varphi  +  y^\bot\cdot F \nabla \varphi  \big\} \dd y\,,
\end{split}
\end{align}
while $\mathcal{R}_\Omega [f]$ satisfies 
\begin{align}
\begin{split}
\| \mathcal{R}_\Omega [f] \|_{L^{\infty}_{1+\gamma} (B^{{\rm c}}_{4R_{0}})}
& \le 
C \bigg ( 
\| F \|_{L^{\infty}_{2+\gamma} (B^{{\rm c}}_{2R_{0}})}
+ \sup_{|x|\geq 4 R_0} |x|^{-1 +\gamma} \| y F \|_{L^1 (\Omega_{ \frac{|x|}{2} })}  \\
&  \qquad 
+ \sup_{|x|\geq 4 R_0} \min \big \{ \frac{1}{|\alpha| |x|^{2-\gamma}}, |x|^\gamma \big \} 
\| F \|_{L^1 (\Omega_{ \frac{|x|}{2} })}  \\
& \quad \qquad  
+ \sup_{|x|\geq 4 R_0} |x|^\gamma \, \big | \lim_{\epsilon \rightarrow 0}  \int_{2 |y|\geq |x|} e^{-\epsilon |y|^2} (F_{12}-F_{21} ) \dd y \big | \bigg ) \\
&  
+ C \big ( |\alpha|^{-\frac{1+\gamma}{2}} + |\alpha|^{-\frac12} + 1 \big ) (1+|\alpha|) \| F\|_{L^2 (\Omega)} \,.
\label{est.thm.linear.exterior.4} 
\end{split}
\end{align}
Here the constant $C$ is independent of $\gamma$, $\alpha$, and $F$. If $F\in L^\infty_{2,0} (\Omega)^{2\times 2}$ then the coefficient $b_\Omega [f]$ coincides with $c_\Omega [f]$.
\end{theorem}
%
\begin{proof} We follow the argument of \cite[Theorem 2.2]{H3}. Since the argument is quite parallel to it, we only give the outline here. (Uniqueness) Let $(u,\nabla p),\, (u', \nabla p')\in W^{2,2}_{loc} (\overline{\Omega})^2\times L^2_{loc} (\overline{\Omega})^2$ be solutions to \eqref{Sa} with the same $f$ such that $\| \nabla u\|_{L^2 (\Omega)}$ and $\|\nabla u'\|_{L^2 (\Omega)}$ are finite and 
$|u(x)|+|u'(x)| \rightarrow 0$ as $|x|\rightarrow \infty$. Then the difference $(v, \nabla q) =(u-u', \nabla (p-p'))\in W^{2,2}_{loc}(\overline{\Omega})^2\times L^2_{loc}(\overline{\Omega})^2$ solves \eqref{Sa} with $f=0$ and satisfies $\|\nabla v\|_{L^2 (\Omega)}<\infty$ as well as $|v(x)|\rightarrow 0$ as $|x|\rightarrow \infty$. Moreover,  the standard elliptic regularity of the Stokes operator implies that $(v, \nabla q)$ is smooth in $\Omega$. Then we can apply \cite[Theorem 2.1, (2.8)]{H3}, which gives $\int_\Omega |D v|^2 \dd x=0$. Hence $v$ is the rigid motion, but the condition $v=0$ on the boundary leads to $v=0$ in $\Omega$. Then we obtain $\nabla q=0$ from the equation. The proof of the uniqueness is complete.
(Existence) Firstly we consider the regularized system 
\begin{equation}\tag{S$_{\alpha}^\lambda$}\label{Sae}
  \left\{\begin{aligned}
 \lambda u_\lambda -\Delta u_\lambda - \alpha ( x^\bot \cdot \nabla u_\lambda - u_\lambda^\bot ) + \nabla p_\lambda  &\,=\, f \,, 
  \quad    {\rm div}\, u_\lambda \,=\, 0\,,  ~~~~~~   x \in \Omega\,, \\
 u_\lambda  & \, =\,  0   \,,  ~~~~~~ x\in \partial\Omega\,, \\
 u_\lambda  & \, \rightarrow \,     0   \,,  ~~~~~~ |x| \rightarrow \infty\,.
\end{aligned}\right.
\end{equation}
Here $\lambda$ is a small positive number. For \eqref{Sae} one can show the existence of the solution $(u_\lambda, \nabla p_\lambda)$ satisfying $\int_{\Omega_{2R_0}} p_\lambda \dd x=0$ and the energy estimate
\begin{align}\label{proof.thm.linear.exterior.1}
\lambda \| u_\lambda \|_{L^2 (\Omega)}^2 + \frac12 \| \nabla u_\lambda \|_{L^2 (\Omega)}^2 \leq \frac12 \| F \|_{L^2 (\Omega)}^2\,.
\end{align}
Moreover, the assumption $f\in L^2 (\Omega)^2$ and the elliptic regularity for the Stokes operator imply the regularity $u_\lambda \in W^{2,2}_{loc} (\overline{\Omega})^2$, $\nabla p_\lambda \in L^2_{loc} (\overline{\Omega})^2$, 
where in virtue of \eqref{proof.thm.linear.exterior.1} each seminorm of $W^{2,2}_{loc} (\overline{\Omega})$ can be bounded uniformly in $\lambda \in (0,1)$. Indeed, since $(u_\lambda, p_\lambda)$ solves the Stokes system with the source term $f + \alpha ( x^\bot \cdot \nabla u_\lambda - u_\lambda^\bot)$, for any bounded subdomain $\omega \subset \Omega$, there exists $\rho>0$ with $\omega \subset \Omega_{\rho}$ such that
\begin{align*}
\|u_\lambda\|_{W^{2,2}(\omega)}
\le C(\|f\|_{L^{2}(\Omega)} 
+ \|\nabla u_\lambda\|_{L^{2}(\Omega)} 
+ \|u_\lambda\|_{L^{2}(\Omega_{\rho})})\,,
\end{align*}
where the constant $C$ depends on $\Omega$, $R_{0}$, $\omega$, and $\rho$; see  \cite[page 117, Theorem 1.5.1]{So} for the proof. From \eqref{proof.thm.linear.exterior.1} and the Poincar${\rm \acute{e}}$ inequality $\|u_\lambda\|_{L^{2}(\Omega_{\rho})} \le C_{\rho} \|\nabla u_\lambda\|_{L^{2}(\Omega)}$ with $C_{\rho}$ depending only on $\Omega$ and $\rho$, we obtain the bound of $u_\lambda$ in $W^{2,2}(\omega)$ which is independent of $\lambda$. 
Let us recall that $R_0\geq 1$ is taken so that 
$\R^2\setminus \Omega\subset B_{R_0}$ and $\varphi\in C_0^\infty (\R^2)$ is
a radial cut-off function such that $\varphi (x) =1$ for $|x|\leq R_0$ and $\varphi (x) =0$ for $|x|\geq 2R_0$. As in the proof of Lemma \ref{lem.linear.exterior}, we introduce the Bogovskii operator $\mathbb{B}$ in the closed annulus $A=\{x\in \R^2~|~ R_0\leq |x|\leq 2 R_0\}$, and set 
\begin{align*}
v_\lambda = (1-\varphi )u_\lambda + \mathbb{B} [\nabla \varphi \cdot u_\lambda]\,,\qquad
q_\lambda = (1-\varphi ) p_\lambda\,.
\end{align*}
Recall that $\mathbb{B}[\nabla \varphi \cdot u_\lambda]$ satisfies 
\begin{align}
& {\rm supp}\, \mathbb{B}[\nabla \varphi \cdot u_\lambda] \subset A\,, \qquad {\rm div}\, \mathbb{B} [\nabla \varphi \cdot u_\lambda] \, = \, \nabla \varphi \cdot u_\lambda\,,\label{proof.thm.linear.exterior.2}\\
& \|\mathbb{B} [\nabla \varphi \cdot u_\lambda]\|_{W^{m+1,2} (\Omega)} \leq C \| \nabla \varphi \cdot u_\lambda \|_{W^{m,2} (\Omega)}\,,~~~~~m=0,1\,.\label{proof.thm.linear.exterior.3}
\end{align}
Then $(v_\lambda, \nabla q_\lambda)$ satisfies
\begin{equation}\label{proof.thm.linear.exterior.4}
  \left\{\begin{aligned}
 \lambda v_\lambda -\Delta v_\lambda - \alpha ( x^\bot \cdot \nabla v_\lambda - v_\lambda^\bot ) + \nabla q_\lambda &\,=\,  {\rm div}\, F_\lambda + g_\lambda \,, ~~~ {\rm div}\, u_\lambda \,=\, 0\,,  ~~~~~~ x \in \R^2\,, \\
 v_{\lambda}  & \, \rightarrow \,     0   \,,  ~~~~~~ |x| \rightarrow \infty\,, 
\end{aligned}\right.
\end{equation}
where 
\begin{align*}
F_\lambda &\,  = \, (1-\varphi ) F - \nabla \mathbb{B}[\nabla \varphi \cdot u_\lambda]\,,\\
g_\lambda & \, =\, F \cdot \nabla \varphi + \lambda \mathbb{B}[\nabla \varphi \cdot u_\lambda] + 2 \nabla \varphi \cdot \nabla u_\lambda + (\Delta \varphi + \alpha x^\bot \cdot \nabla \varphi ) u_\lambda \\
& ~~~ - \alpha \big ( x^\bot \nabla  \mathbb{B}[\nabla \varphi \cdot u_\lambda] - \mathbb{B}[\nabla \varphi \cdot u_\lambda]^\bot  \big ) - ( \nabla \varphi ) p_\lambda \,.
\end{align*}
Note that ${\rm supp}\, g_\lambda \subset A$ due to \eqref{proof.thm.linear.exterior.2}. Let $\Gamma_\alpha^\lambda (x,y)$ be the function defined in \eqref{def.Gamma_epsilon}. Then, as is shown in \cite{H3} (see also Remark \ref{rem.thm.linear.whole}), the velocity $v_\lambda$ is written as 
\begin{align}
v_\lambda (x) 
& \, = \, 
\int_{\R^2}  \Gamma_{\alpha}^{\lambda}(x,y) {\rm div}\, F_\lambda (y) \dd y  
+ \int_{\R^2} \Gamma_{\alpha}^{\lambda}(x,y) g_\lambda (y) \dd y \nonumber \\
&  \, = \, 
w_\lambda (x) + r_\lambda (x) \,. \label{proof.thm.linear.exterior.5}
\end{align}
Since $g_\lambda =0$ for $|x|\geq 2R_0$, we have from \cite[Proposition 3.3]{H3},
\begin{align}
\| r_\lambda \|_{L^{\infty}_{1} (B^{{\rm c}}_{4R_{0}})} 
& \leq C_\alpha \int_\Omega (1+|y| ) |g_\lambda (y) | \dd y  \nonumber \\
& \leq C_\alpha \| g_\lambda \|_{L^2 (\Omega)} \nonumber \\
& \leq C_\alpha \big ( \| F \|_{L^2 (\Omega_{2 R_0})} + (1+|\alpha|) \| \nabla u_\lambda \|_{L^2 (\Omega_{2 R_0})}  + \| p_\lambda \|_{L^2 (\Omega_{2 R_0})}  \big )\,.\label{proof.thm.linear.exterior.6}
\end{align}
Since $\int_{\Omega_{2 R_0}} p_\lambda \dd x =0$ we have from \eqref{Sae},
\begin{align*}
\| p_\lambda \|_{L^2 (\Omega_{2R_0})} \leq C \| \nabla p_\lambda \|_{H^{-1} (\Omega_{2 R_0})} 
\leq C \big ( \| F \|_{L^2 (\Omega_{2 R_0})} + (1+|\alpha|) \| \nabla u_\lambda \|_{L^2 (\Omega_{2 R_0})}  \big )
\end{align*}  
Combining this estimate with \eqref{proof.thm.linear.exterior.1} and \eqref{proof.thm.linear.exterior.6}, we obtain 
\begin{align}
\| r_\lambda \|_{L^{\infty}_{1} (B^{{\rm c}}_{4R_{0}})} 
\leq C_\alpha \| F \|_{L^2 (\Omega)}\,.\label{proof.thm.linear.exterior.7}
\end{align}
Here $C_\alpha$ depends only on $\alpha$ and $R_0$, but is independent of $\lambda \in (0,1)$. As for $w_\lambda$, from Lemma \ref{lem.thm.linear.exterior.1}, there is $0<\theta<1$ such that 
\begin{align}
\| w_\lambda \|_{L^{\infty}_{\theta} (B^{{\rm c}}_{4R_{0}})} 
& \leq 
C \big ( \| F\|_{L^\infty_{2+\gamma}(B^{{\rm c}}_{2R_{0}})} 
+ \| F_\lambda \|_{L^1(B_{2R_{0}})} \big )\nonumber \\
& \leq 
C \big ( \| F\|_{L^\infty_{2+\gamma}(B^{{\rm c}}_{2R_{0}})} 
+ \| F \|_{L^2 (\Omega)}\big ) \,.\label{proof.thm.linear.exterior.8}
\end{align}
Collecting \eqref{proof.thm.linear.exterior.1}, \eqref{proof.thm.linear.exterior.7},  \eqref{proof.thm.linear.exterior.8}, 
and $u_\lambda \in W^{2,2}_{loc} (\overline{\Omega})^2$ with its uniform bound on $\lambda \in (0,1)$,
we have a uniform estimate in $\lambda \in (0,1)$:
\begin{align}
\|u_\lambda \|_{L^\infty_{\theta} (\Omega)} 
\leq C_\alpha \big ( \| F\|_{L^\infty_{2+\gamma}(\Omega)} + \| F \|_{L^2 (\Omega)} \big )\,,\label{proof.thm.linear.exterior.9}
\end{align}
where the Sobolev embedding $W^{2,2}(\Omega_{5 R_0})\hookrightarrow L^\infty (\Omega_{5 R_0})$ has been applied. Thus, there are a subsequence, denoted again by $(u_\lambda, \nabla p_\lambda)$, and $(u,\nabla p)\in W^{2,2}_{loc} (\overline{\Omega})^2 \times L^2_{loc}(\overline{\Omega})^2$, such that $u_\lambda \rightharpoonup^* u$ in $L^{\infty}_\theta (\Omega)^2$, $\nabla u_\lambda \rightharpoonup \nabla u$ in $L^2 (\Omega)^{2\times 2}$, and $p_\lambda \rightharpoonup p$ in $W^{1,2}_{loc}(\overline{\Omega})$. It is easy to see that $(u, \nabla p)$ satisfies \eqref{Sa} in the sense of distributions 
 (note that each term of \eqref{Sa} makes sense at least as a function in $L^2_{loc}(\overline{\Omega})$). The proof of the existence is complete. 

\noindent (Estimates) We note that the solution $(u,\nabla p)$ obtained in the existence proof above satisfies 
$\| \nabla u \|_{L^2 (\Omega)} \leq \| F \|_{L^2 (\Omega)}$ in virtue of \eqref{proof.thm.linear.exterior.1}. 
Thus \eqref{est.thm.linear.exterior.0} holds.
Since the pressure $p$ is uniquely determined up to a constant, we may assume $\int_{\Omega_{6R_0}} p \dd x=0$. Then we have from \eqref{Sa},
\begin{align*}
\|  p \|_{L^2 (\Omega_{6 R_0})}   \leq C \| \nabla p  \|_{H^{-1} (\Omega_{6 R_0})} 
& \leq C \big ( \| F \|_{L^2 (\Omega_{6 R_0})} + (1+|\alpha|) \| \nabla u \|_{L^2 (\Omega_{6 R_0})} \big )\\
& \leq C (1+|\alpha|) \| F \|_{L^2 (\Omega)}\,.
\end{align*} 
Here $C$ depends only on $R_0$. This proves \eqref{est.thm.linear.exterior.1}.
The local estimates \eqref{est.thm.linear.exterior.2} follow from the standard cut-off argument and the elliptic estimates for the Stokes system in bounded domains, together with the estimates \eqref{est.thm.linear.exterior.0} and \eqref{est.thm.linear.exterior.1}. Since the argument is rather standard, we omit the details here.
The expansion \eqref{est.thm.linear.exterior.3} with \eqref{est.thm.linear.exterior.5}  and the estimate \eqref{est.thm.linear.exterior.4} follow from Lemma \ref{lem.linear.exterior} and \eqref{est.thm.linear.exterior.0}.
Note that the constant vector $u_\infty$ in \eqref{est.lem.linear.exterior.1} must be zero, for the solution $u$ constructed here decays as $|x|\rightarrow \infty$. The proof of Theorem \ref{thm.linear.exterior} is complete.
\end{proof}

\begin{remark} 
Let $R_0\geq 1$ be as in Theorem \ref{thm.linear.exterior} and let $\gamma \in  [0, 1)$. Then we have for $|x|\geq 4 R_0$,
\begin{align*}
\| y F \|_{L^1(\Omega_{ \frac{|x|}{2}})} 
& \leq 
\frac{C}{1-\gamma} 
|x|^{1-\gamma} \| F \|_{L^\infty_{2+\gamma} (\Omega)}\,,\\
\| F \|_{ L^1(\Omega_{ \frac{|x|}{2}})} 
& \leq 
C \| F \|_{L^\infty_{2+\gamma} (\Omega)}  \log |x| \,.
\end{align*}
Here $C$ is independent of $\gamma$ and $F$. Since 
\begin{align*}
 \min \big \{  \frac{1}{|\alpha| |x|^{2-\gamma}}, |x|^\gamma \big \} \log |x|  \leq |\alpha|^{-\frac{\gamma}{2}} \big |\log |\alpha |\big |\,,~~~~~|\alpha|>0\,,
\end{align*}
we have for $\gamma \in [0, 1)$ and $0<|\alpha|<1$, by using \eqref{est.thm.linear.exterior.4}, 
\begin{align}
\begin{split}
\| \mathcal{R}_\Omega [f] \|_{L^{\infty}_{1+\gamma} (B^{{\rm c}}_{4R_{0}})} 
& \le \frac{C}{1-\gamma} \bigg (   |\alpha|^{-\frac{\gamma}{2}} \big |\log |\alpha| \big | \, \| F \|_{L^\infty_{2+\gamma} (\Omega)}   
 + |\alpha|^{-\frac{1+\gamma}{2}}  \| F\|_{L^2(\Omega)}  \\
& \qquad
+ \sup_{|x|\geq 4 R_0} |x|^{\gamma} \, \big | \lim_{\epsilon \rightarrow 0}  \int_{2 |y|\geq |x|} e^{-\epsilon |y|^2} (F_{12}-F_{21} ) \dd y \big | \bigg )  \,. \label{est.thm.linear.exterior.4'} 
\end{split}
\end{align}
Here $C$ is independent of $0<|\alpha|<1$, $\gamma\in [0,1)$, and $F$.
The estimate \eqref{est.thm.linear.exterior.4'} plays a central role to solve the Navier-Stokes equations for small $|\alpha|$ in the next section. We note that $\tilde c_\Omega [F]$ and the last term in the right-hand side of \eqref{est.thm.linear.exterior.4'} do not converge in general when $F\in L^\infty_{2} (\Omega)^{2\times 2}$ . In solving the Navier-Stokes equations, especially for the case $\gamma=0$, it is crucial that we only need the decay of the component $F_{12}-F_{21}$, which always vanishes when $F$ is symmetric.
\end{remark}


\section{Solvability of nonlinear problem}\label{sec.nonlinear}

Based on the linear analysis in the previous sections the following Navier-Stokes equations are studied in this section:
\begin{equation}\tag{NS$_{\alpha}$}\label{NSa}
  \left\{\begin{aligned}
  -\Delta u - \alpha ( x^\bot \cdot \nabla u - u^\bot ) + \nabla p  &\,=\, - u\cdot\nabla u  + f \,, 
  \quad {\rm div}\, u \,=\, 0\,,  ~~~~~~   x \in \Omega\,, \\
 u  & \, =\,  \alpha x^\bot   \,,  ~~~~~~ x\in \partial\Omega\,, \\
 u  & \, \rightarrow \,     0   \,,  ~~~~~~ |x| \rightarrow \infty\,,
\end{aligned}\right.
\end{equation}
Our aim is to prove, under some conditions on $f$, the unique existence of solutions $(u,\nabla p)$ to \eqref{NSa} satisfying the asymptotic behavior 
\begin{equation*}
u (x) \, = \, \beta V(x) + o(|x|^{-1})
\qquad {\rm as} ~\ |x| \rightarrow \infty
\end{equation*}
for some $\beta \in \R$, where $V$ is a radial circular flow defined by \eqref{def.V} and coincides with $\frac{x^\bot}{4\pi |x|^2}$ for $|x|\gg 1$. As in the previous sections we fix a positive number $R_0\geq 1$ large enough so that $\R^2\setminus \Omega \subset B_{R_0}$, and let $\varphi \in C^\infty_0 (\R^2)$ be a radial cut-off function satisfying 
$\varphi(x) = 1 $ for $|x| \le R_0$, $\varphi (x)=0$ for $|x|\geq 2 R_0$. Set 
\begin{align}
U(x) & \, = \, \varphi (x) x^\bot\,, \label{def.U}
\end{align}
which is a radial circular flow supported in the ball $B_{2R_0}$. We also introduce the function space $X_{\gamma}$, $\gamma \geq 0$, as 
\begin{align}\label{def.X_gamma}
X_{\gamma} \, = \, 
\R  \times \big (\dot{W}^{1,2}_{0,\sigma} (\Omega) \cap L^\infty_{1+\gamma} (\Omega)^2 \big )\,,
\end{align}
which is the Banach space under the norm for $(\beta, w) \in X_\gamma$:
\begin{align}
\| (\beta, w) \|_{X_\gamma} \, = \, |\beta| + \| \nabla w\|_{L^2(\Omega)} + \| w \|_{L^\infty_{1+\gamma}(\Omega)}\,.
\end{align} 
We sketch the proof that $X_\gamma$ is complete. It suffices to show the completeness of the space $\dot{W}^{1,2}_{0,\sigma} (\Omega) \cap L^\infty_{1+\gamma} (\Omega)^2$. Suppose that $\{w^{(n)}\} \subset \dot{W}^{1,2}_{0,\sigma} (\Omega) \cap L^\infty_{1+\gamma} (\Omega)^2$ is a Cauchy sequence. Then there exist $u \in \dot{W}^{1,2}_{0,\sigma} (\Omega)^2$ and $v \in L^\infty_{1+\gamma} (\Omega)^2$ such that $\|\nabla(w^{(n)} - u)\|_{L^2(\Omega)} \rightarrow 0$ and $\|w^{(n)} - v\|_{L^\infty_{1+\gamma}(\Omega)} \rightarrow 0$ as $n \rightarrow \infty$. What we need to show is $u=v$. To show this, set $f=u-v$.
Note that the fact $u, w^{(n)} \in \dot{W}_0^{1,2}(\Omega)^2$ implies $u=w^{(n)}=0$ on $\partial\Omega$.
Then, for any $\phi \in W^{1,2}(\Omega)$ such that ${\rm supp}\, \phi$ is compact, the integration by parts yields for $j,k=1,2$,
\begin{equation*}
\begin{split}
\int_{\Omega}  f_j  \partial_k \phi  \dd x  & \, = \, \int_\Omega (u_j - v_j ) \partial_k \phi \dd x  \\
& \, = \, - \int_\Omega \phi \partial_k u_j  \dd x - \int_\Omega v_j \partial_k \phi \dd x \\ 
& \,=\, - \lim_{n \rightarrow \infty} \big ( \int_{\Omega}  \phi \partial_k w_j^{(n)}  \dd x + \int_{\Omega} w_j^{(n)} \partial_k  \phi \dd x \big ) \,=\, 0\,.
\end{split}
\end{equation*}
Since we may take an arbitrary $\phi \in C_0^\infty (\Omega)$ we first conclude from the above computation that $f_j$ is a constant in $\Omega$, denoted by $c_j$. Next we have for $\varphi\in W^{1,2} (\Omega)^2$ such that ${\rm supp}\, \varphi$ is compact,
\begin{align*}
c_j \int_{\partial\Omega} \varphi \cdot \nu   \dd \sigma_x \, = \, \int_\Omega c_j\,  {\rm div}\,  \varphi \dd x \, = \,  \int_\Omega f_j \, {\rm div}\, \varphi \dd x  \, = \, 0\,,
\end{align*}
where the result of the above computation is used.
This implies $c_j=0$ since we can choose $\varphi$ so that $\int_{\partial\Omega} \varphi \cdot \nu \dd \sigma_x \ne 0$. Thus we obtain $u=v$, and hence, $X_\gamma$ is complete.

Let us recall that for $f\in L^2 (\Omega)^2$ of the form $f={\rm div}\, F=(\sum_{j=1,2}\partial_j F_{1j}, \sum_{j=1,2}\partial_j F_{2j})^\top$ with some $F\in L^2(\Omega)^{2\times 2}$ satisfying $F_{12}-F_{21}\in L^1 (\Omega)$ the coefficients  $\tilde c_\Omega[F]$ and $b_\Omega [f]$ in \eqref{def.c.tildec.exterior} and \eqref{est.thm.linear.exterior.5} are well-defined.
The main results of this section are Theorems \ref{thm.nonlinear}, \ref{thm.nonlinear'} below.
Let us start from the next theorem.
\begin{theorem}\label{thm.nonlinear} Let $\gamma\in [0,1)$.
There exists a positive constant $\epsilon=\epsilon (\Omega,\gamma)$ such that the following statement holds.
Suppose that $f\in L^2 (\Omega)^2$ is of the form $f={\rm div}\, F$ with some $F\in L^\infty_{2+\gamma} (\Omega)^{2\times 2}$, and in addition that $F_{12}-F_{21} \in L^1 (\Omega)$ when $\gamma=0$. If $\alpha\ne 0$ and 
\begin{align}
\begin{split}
|\alpha|^{\frac{1-\gamma}{2}}\, \big |\log |\alpha| \big | & + |\alpha|^{-\frac{\gamma}{2}} \big | \log |\alpha|\big | \, 
\bigg ( |\alpha|^{-\frac12}  \big ( |b_\Omega [f]|  +  \| F \|_{L^2 (\Omega)} + \| f \|_{L^2 (\Omega_{6 R_0})} \big )  \\
& \qquad \qquad 
+ \| F_{12}-F_{21}\|_{L^1 (\Omega)} + \big | \log |\alpha| \big | \,  \| F \|_{L^\infty_{2}(\Omega)}  \bigg )  
< \epsilon\,,
\end{split}
\end{align}
then there exists a unique solution $(u,\nabla p)\in  W^{2,2}_{loc} (\overline{\Omega})^2 \times L^2_{loc} (\overline{\Omega})^2$ to \eqref{NSa} satisfying 
\begin{align}\label{est.thm.nonlinear.1}
\begin{split}
\| \nabla u \|_{L^2 (\Omega)}  
\leq \frac{\| F\|_{L^2 (\Omega)} + C_2 |\alpha |}{\sqrt{1-C_1|\alpha|}} \,, 
\end{split}
\end{align}
and enjoying the expression $ u = \alpha U + \beta V + w$ with $U$ and $V$ defined by \eqref{def.U} and \eqref{def.V}, respectively, and
\begin{align}\label{est.thm.nonlinear.2}
\begin{split}
\beta & \, = \,  \int_{\partial \Omega} y^\bot \cdot \big( T(u,p)\nu \big) \dd \sigma_{y} \,+\, b_\Omega [f]\,,
\end{split}
\end{align}
while
\begin{align}
\begin{split}
\| w \|_{L^\infty_1 (\Omega)} 
& \leq C_3 
\bigg ( |\alpha|^{-\frac{1}{2}} 
\big ( |\alpha | + |b_\Omega[f]| +  \| F\|_{L^2 (\Omega)} + \| f \|_{L^2 (\Omega_{6 R_0})}  \big )  \\
& \qquad\qquad 
+ \big | \log |\alpha|\big |\,    \| F \|_{L^\infty_{2} (\Omega)}  + \| F_{12} - F_{21} \|_{L^1 (\Omega)} \bigg ) \,,\label{est.thm.nonlinear.3}
\end{split}
\end{align}
and if $\gamma \in (0,1)$,
\begin{align}
\begin{split}
\| w \|_{L^\infty_{1+\gamma} (\Omega)} 
& \leq C_3 \bigg ( |\alpha|^{-\frac{1+\gamma}{2}} \big ( |\alpha | |\log |\alpha|  \big | 
+ |b_\Omega[f]| +  \| F\|_{L^2 (\Omega)} + \| f \|_{L^2 (\Omega_{6 R_0})}  \big )  \\
& \quad  \quad\quad\quad  + \big (  |\alpha|^{-\frac{\gamma}{2}} \big | \log |\alpha|\big |\,   + \frac{1}{\gamma} \big ) \| F \|_{L^\infty_{2+\gamma} (\Omega)}  \bigg )\,.\label{est.thm.nonlinear.4}
\end{split}
\end{align}
Here $\epsilon$, $C_1$, $C_2$, and $C_3$ depend only on $\Omega$ and $\gamma$, and are taken uniformly with respect to $\gamma$ in each compact subset of $[0,1)$. 
\end{theorem}

\begin{remark} (i) A careful analysis implies that  $\beta$ in Theorem \ref{thm.nonlinear} is estimated as 
\begin{align}
|\beta| \leq C_4 \big ( |\alpha| + |b_\Omega[f]| + \| F \|_{L^2 (\Omega)} + \| f\|_{L^2 (\Omega_{6 R_0})} \big )\,,
\end{align}
where $C_4$ depends only on $\Omega$. But we do not go into details in this paper.

\noindent (ii) In Theorem \ref{thm.nonlinear} when $\gamma=0$ the term $w$ decays with the order $O(|x|^{-1})$ and there is no reason why $\beta V$ provides a leading term of the asymptotic behavior of $u$ at $|x|\rightarrow \infty$. 
To achieve this asymptotics  we need the additional decay of $F$ such as $F\in L^\infty_{2,0} (\Omega)^{2\times 2}$; see Theorem \ref{thm.nonlinear'} below.

\end{remark}

\begin{proofx}{Theorem \ref{thm.nonlinear}}
\noindent
In the following argument we will freely use the condition $0<|\alpha| <  e^{-1} $. We look for the solution to \eqref{NSa} of the form
\begin{equation}
u \, = \, \alpha U + v\,,~~~~~~~~~ v\, =\, \beta V + w\,, \qquad (\beta, w)  \in X_\gamma\,. \label{proof.thm.nonlinear.0}
\end{equation}
We need to determine $\beta$ and $w$. Inserting \eqref{proof.thm.nonlinear.0} into \eqref{NSa}, we see that $v$ is the solution to the system
\begin{equation}\tag{NS$_{\alpha}'$}\label{NSad}
 \left\{ \begin{aligned}
  -\Delta v - \alpha ( x^\bot \cdot \nabla v - v^\bot ) + \nabla q  
& \,=\, {\rm div}\, G_{\alpha}(\beta, w) + {\rm div}\, H_{\alpha} (F)\,, ~~~~~~ x\in \Omega\,,\\
 {\rm div}\, v & \,=\, 0\,, ~~~~~~ x \in \Omega\,, \\
  v  & \,=\, 0 \,,  ~~~~~~ x \in \partial \Omega\,. \\
  v  & \, \rightarrow \,     0   \,,  ~~~~~~ |x| \rightarrow \infty\,.
\end{aligned}\right.
\end{equation}
Here 
\begin{align*}
q & \,=\, p + P\,, \\
G_\alpha (\beta,w) 
& \,=\,  -\alpha ( U \otimes w + w \otimes U )  - \beta ( V \otimes w + w \otimes V ) - w \otimes w\,, \\
H_\alpha (F) & \,=\, \alpha \nabla U + F\,,
\end{align*}
and we may assume that $\int_{\Omega_{6 R_0}} q \dd x=0$. Note that we have used the relations 
$x^\bot \cdot \nabla U - U^\bot =0$, and the radial scalar function $P=P(|x|)$ is taken so that 
$ \nabla P = {\rm div}\, [ (\alpha U + \beta V) \otimes (\alpha U + \beta V)]$.
Both of these follow from the direct calculation. The proof of the unique existence below relies on the standard Banach fixed point argument in a suitable class of functions. To this end we introduce the closed convex set $\mathcal{B}_{\vec{\delta},\gamma}$ in $X_0$:
\begin{align}\label{proof.thm.nonlinear.1}
\begin{split}
\mathcal{B}_{\vec{\delta}, \gamma}   
\, = \, \mathcal{B}_{(\delta_1,\delta_2,\delta_3),\gamma}  
&\, = \, 
\big \{ (\beta, w) \in X_0~|~ 
\quad |\beta | + \| \nabla w\|_{L^2 (\Omega)} + \| w\|_{L^\infty (\Omega_{5 R_0})}  \leq \delta_1\,, \\
&  ~~\qquad\qquad\qquad\qquad~~
\| w\|_{L^\infty_1 (\Omega)} \leq \delta_2\,,~~~\quad \| w\|_{L^\infty_{1+\gamma} (\Omega)}\leq \delta_3 \big \}\,.
\end{split}
\end{align}
Here we have set $\vec{\delta}=(\delta_1,\delta_2,\delta_3)$, and the positive numbers $\delta_1, \delta_2,\delta_3$ with $\delta_2 \leq \delta_3$ will be suitably determined later. We note that the following inclusion always holds for $\delta_2\leq \delta_3$.
\begin{align}
\mathcal{B}_{(\delta_1,\delta_2,\delta_3), \gamma} \subset \mathcal{B}_{(\delta_1,\delta_2,\delta_2),0}\,.\label{inclusion}
\end{align}
For any $\omega = (\beta, w) \in \mathcal{B}_{\vec{\delta}, \gamma}$,
let $(u_{\omega},\nabla q_\omega)$ be the unique solution in Theorem \ref{thm.linear.exterior} to the linear system
\begin{equation*}
  \left\{
\begin{aligned}
  -\Delta u_{\omega} - \alpha ( x^\bot \cdot \nabla u_{\omega} - u_{\omega}^\bot ) + \nabla q_{\omega} 
 &  \,=\, {\rm div}\, G_{\alpha}(\beta, w) + {\rm div}\, H_{\alpha} (F)\,, ~\quad ~ x\in \Omega\,,\\
 {\rm div}\, u_{\omega} & \,=\, 0\,, ~~~~~~  x \in \Omega, \\
 u_{\omega} & \,=\, 0 \,,  ~~~~~~ x \in \partial \Omega\,, \\
 u_{\omega} & \, \rightarrow \,     0   \,,  ~~~~~~ |x| \rightarrow \infty\,,
\end{aligned}\right.
\end{equation*}
Our aim is to show the unique existence of $(\beta, w)\in \mathcal{B}_{\vec{\delta},\gamma}$ such that $u_{\omega} = u_{(\beta,w)} = \beta V + w$ for suitably chosen and sufficiently small $0<\delta_1  \leq \delta_2< e^{-2}$ and $\delta_2 \leq \delta_3$. We remark that the value $\delta_3$ need not to be small when $\gamma$ is positive. Let us start from the estimates for  $G_{\alpha}(\beta, w)$. Firstly we estimate its $L^2$ norm as 
\begin{align}
\begin{split}
&~~~ \| G_{\alpha}(\beta, w)\|_{L^2 (\Omega)}  \\
& \leq 
C\bigg ( |\alpha| \, \| \nabla w\|_{L^2 (\Omega)} + |\beta|  \, \|w\|_{L^\infty_{1} (\Omega)}  + \| w\|_{L^\infty_{1}(\Omega)} \| \nabla w\|_{L^2 (\Omega)} | \log \| \nabla w \|_{L^2 (\Omega)} | \bigg ) \,. \label{proof.thm.nonlinear.2}
\end{split}
\end{align}
Here, for the nonlinear term, we have used \eqref{est.lem.A.1.2} and the smallness of $\delta_1$ and $\delta_2$ to obtain
\begin{align*}
\| w\otimes w\|_{L^2(\Omega)} & \leq C \| w\|_{L^\infty_1 (\Omega)} \|(1+|x|)^{-1} w\|_{L^2(\Omega)} \\
&\leq C \| w\|_{L^\infty_1 (\Omega)} \|\nabla w\|_{L^2 (\Omega)} |\log \| \nabla w \|_{L^2 (\Omega)} |\,.
\end{align*}
On the other hand, it is not difficult to see that
\begin{align}
\| G_{\alpha}(\beta, w)\|_{L^\infty_{2+\gamma'} (\Omega)} 
& \leq C\big ( |\alpha| + |\beta| + \|w\|_{L^\infty_{1}(\Omega)} \big ) \| w\|_{L^\infty_{1+\gamma'}(\Omega)}\,, \quad 0\leq \gamma'\leq \gamma\,, \label{proof.thm.nonlinear.3}  \\
\|  {\rm div}\, G_{\alpha}(\beta, w) \|_{L^2 (\Omega_{6R_0})} 
& \leq C \big ( |\alpha| + |\beta| + \|w\|_{L^\infty(\Omega_{6R_0})} \big ) \| \nabla w\|_{L^2 (\Omega)}\,, \label{proof.thm.nonlinear.5}
\end{align}
and 
\begin{align}
\| H_\alpha (F)\|_{L^2 (\Omega)} & \leq C \big ( |\alpha| + \| F \|_{L^2 (\Omega)} \big )\,,\label{proof.thm.nonlinear.6}  \\
\| H_\alpha (F)\|_{L^\infty_{2+\gamma'} (\Omega)} & \leq C \big ( |\alpha| + \| F \|_{L^\infty_{2+\gamma'} (\Omega)} \big )\,, \quad \quad 0\leq \gamma'\leq \gamma\,, \label{proof.thm.nonlinear.7}\\
\| {\rm div}\, H_\alpha (F)\|_{L^2 (\Omega_{6 R_0})} & \leq C \big ( |\alpha| + \| f \|_{L^2 (\Omega_{6 R_0})} \big )\,.\label{proof.thm.nonlinear.8}
\end{align}
Then we can apply the result of Theorem \ref{thm.linear.exterior}. To simplify the notation we set 
\begin{align}\label{proof.thm.nonlinear.9}
\begin{split}
M (\alpha,\beta,F, w) & \, = \,  ( |\alpha| +|\beta| ) \, \| \nabla w\|_{L^2 (\Omega)} +  |\beta |\,  \| w\|_{L^\infty_{1} (\Omega)}  \\
& ~~~ + \| w\|_{L^\infty_{1}(\Omega)}  \| \nabla w\|_{L^2 (\Omega)} |\log \| \nabla w \|_{L^2 (\Omega)}|  \, + |\alpha| + \| F \|_{L^2 (\Omega)} \,.
\end{split}
\end{align}
From \eqref{est.thm.linear.exterior.0}, \eqref{proof.thm.nonlinear.2}, and  \eqref{proof.thm.nonlinear.6}, we have 
\begin{align}
\| \nabla u_{(\beta,w)} \|_{L^2 (\Omega)}
& \le CM (\alpha,\beta,F, w)\,.\label{proof.thm.nonlinear.10}
\end{align}
Moreover, by the Sobolev embedding $W^{2,2}(\Omega_{5 R_0})\hookrightarrow L^\infty (\Omega_{5 R_0})$ 
and \eqref{est.thm.linear.exterior.0} - \eqref{est.thm.linear.exterior.2} 
combined with \eqref{proof.thm.nonlinear.2},  \eqref{proof.thm.nonlinear.5},  \eqref{proof.thm.nonlinear.6}, 
\eqref{proof.thm.nonlinear.8}, and $\| w\|_{L^\infty (\Omega_{6 R_0})} \leq \| w\|_{L^\infty_{1} (\Omega)}$, we have 
\begin{align}\label{proof.thm.nonlinear.11}
\begin{split}
& ~~\| u_{(\beta,w)} \|_{L^\infty(\Omega_{5 R_0})} + \| u_{(\beta,w)} \|_{W^{2,2}(\Omega_{5 R_0})} + \| q_{(\beta,w)} \|_{W^{1,2}(\Omega_{5 R_0})}  \\
& \leq C \big ( M (\alpha,\beta,F, w) +  \| f \|_{L^2 (\Omega_{6 R_0})} \big )\,.
\end{split}
\end{align}
Set  $\tilde F =G_\alpha (\beta, w) + H_\alpha (F)$ and $\tilde f = {\rm div}\, \tilde F$. 
By Theorem \ref{thm.linear.exterior}, the velocity $u_\omega = u_{(\beta,w)}$ is written as
\begin{equation*}
u_{\omega} \,=\, \psi [\omega] V + R[\omega]\,,
\end{equation*}
where $R[\omega]$ belongs to $L^\infty_{1+\gamma} (\Omega)^2$ and $\psi [\omega]$ is given by
\begin{align}
\begin{split}
\psi [\omega] 
& \, = \, 
\int_{\partial \Omega} y^\bot \cdot T(u_{\omega}, q_{\omega}) \nu \, \dd \sigma_{y} + b_\Omega [\tilde f]\,,\\
b_\Omega [\tilde f]  
& \, = \, 
\tilde c_\Omega [ \tilde F]  
+ \int_\Omega \
\big\{ \big ( y^\bot \cdot \tilde f  - \tilde F_{12} + \tilde F_{21} \big ) \varphi  
+ y^\bot \cdot \tilde F \nabla \varphi  \big\} \dd y\,. 
\label{proof.thm.nonlinear.12}
\end{split}
\end{align}
We observe that $\tilde c_\Omega [G_\alpha (\beta,w)]=0$ and 
\begin{align*}
\int_\Omega 
\big\{
\big ( y^\bot \cdot {\rm div}\, G_\alpha (\beta,w)  - G_\alpha (\beta,w)_{12} + G_\alpha (\beta,w)_{21} \big ) \varphi  + y^\bot \cdot (G_\alpha (\beta,w) \nabla \varphi ) 
\big\} \dd y  \, = \, 0\,.
\end{align*}
Here we have used the facts that $G_{\alpha}(\beta, w)$ is symmetric and its trace on the boundary is zero.
This implies $b_\Omega [{\rm div}\, G_\alpha (\beta, w)]=0$. Moreover, we have 
\begin{align*}
b_\Omega [ \Delta U]  \, = \, c_\Omega [\Delta U] & \, =\, 0
\end{align*}
in virtue of the computation
\begin{align*}
\int_{\Omega} y^\bot \cdot \Delta U \dd y 
\, = \, \int_{\Omega} y\cdot \nabla {\rm rot}\, U \dd y & 
\, = \, \int_{\partial\Omega} y\cdot \nu  \, ({\rm rot}\, U ) \dd \sigma_y
- 2\int_\Omega {\rm rot}\, U \dd y \\
& \, = \,  \int_{\partial\Omega} y\cdot \nu  \, ({\rm rot}\, U ) \dd \sigma_y 
- 2 \int_{\partial\Omega} \nu^\bot \cdot U \dd \sigma_y \\
& \, = \, 2 \int_{\partial\Omega} y\cdot \nu  \, \dd \sigma_y 
- 2 \int_{\partial\Omega} \nu^\bot \cdot y^\bot \dd \sigma_y  
\, =\, 0\,.
\end{align*}
Here ${\rm rot}\, U = \partial_1 U_2- \partial_2 U_1$ and we have used the identity $U(x)=x^\bot$ near $\partial\Omega$. 
Hence, \eqref{proof.thm.nonlinear.12} is in fact written as
\begin{align}
\psi [\omega] &\, = \, \int_{\partial \Omega} y^\bot \cdot T(u_{\omega}, q_{\omega}) \nu \, \dd \sigma_{y}  + b_\Omega [f]\,. \label{proof.thm.nonlinear.12'}
\end{align}

Now let us define the mapping $\Phi : \mathcal{B}_{\vec{\delta},\gamma} \rightarrow X_0$ as 
\begin{align}
\Phi [\omega] \, = \, ( \psi [\omega], R[\omega] )\,,~~~~~~
\psi [\omega] {\rm ~is~given~by}~\eqref{proof.thm.nonlinear.12'}\,,
~~~R[\omega] = u_{\omega} - \psi [\omega] V\,.\label{def.Phi}
\end{align} 
Recalling the inclusion \eqref{inclusion}, our aim is to show

\noindent 
(i) $\Phi$ is a mapping from $\mathcal{B}_{\vec{\delta},\gamma}$ into $\mathcal{B}_{\vec{\delta},\gamma}$, and

\noindent 
(ii) $\Phi$ is a contraction on $\mathcal{B}_{\vec{\delta},0}$ in the topology of $X_0$. i.e., there is $\tau\in (0,1)$ such that $\| \Phi (\omega_1) - \Phi (\omega_2) \|_{X_0} \leq \tau \| \omega_1 - \omega_2 \|_{X_0}$ for any $\omega_1, \omega_2 \in \mathcal{B}_{\vec{\delta},0}$. 

\noindent 
The properties (i) and (ii) imply the existence of the fixed point of $\Phi$ in $\mathcal{B}_{\vec{\delta},\gamma}$ even for the case $\gamma >0$. Indeed, note that the sequence $\{\omega^{(n)}\}_{n=0}^\infty = \{(\beta^{(n)}, w^{(n)}\}_{n=0}^\infty$ defined by $\omega^{(0)}= \Phi (0)$ and $\omega^{(n)}=\Phi (\omega^{(n-1)})$ for $n=1,\ldots$ is a Cauchy sequence in $X_0$ and each $\omega^{(n)}$ belongs to $\mathcal{B}_{\vec{\delta},\gamma}$, which is not difficult to see from (i) and (ii). 
Then the limit $\omega =(\beta, w)$ of $\{\omega^{(n)}\}_{n=0}^\infty$ in $X_0$ also belongs to $\mathcal{B}_{\vec{\delta},\gamma}$ since $\mathcal{B}_{\vec{\delta},\gamma}$ is a closed subset in $X_0$ by the definition. 

To prove (i) let us estimate $\psi [\omega]$ based on the representation \eqref{proof.thm.nonlinear.12'}. By the trace theorem  we have 
\begin{align*}
|  \int_{\partial \Omega} y^\bot \cdot T(u_{\omega}, q_{\omega}) \nu \, \dd \sigma_{y} | 
\leq C \big (\| \nabla u_\omega \|_{W^{1,2}(\Omega_{5 R_0})} + \| q_\omega \|_{W^{1,2}(\Omega_{5 R_0})} \big )\,,
\end{align*}
Hence we have from \eqref{proof.thm.nonlinear.11},
\begin{align}
| \psi [\omega] | &\leq  C \big ( M (\alpha,\beta,F, w) +  | b_\Omega [f] | + \| f \|_{L^2 (\Omega_{6 R_0})} \big )\,.\label{proof.thm.nonlinear.13}
\end{align}
Next let us estimate $R[\omega]$. Firstly we observe from \eqref{proof.thm.nonlinear.11}, \eqref{proof.thm.nonlinear.10}, 
and \eqref{proof.thm.nonlinear.13} that 
\begin{align}
\| R[\omega]\|_{L^\infty (\Omega_{5 R_0})} + \| \nabla R[\omega] \|_{L^2 (\Omega)} 
& = 
\| u_\omega - \psi [\omega] V \|_{L^\infty (\Omega_{5 R_0})} + \| \nabla (u_\omega - \psi [\omega]V )  \|_{L^2 (\Omega)} \nonumber \\
& \leq  C \big ( \| u_\omega \|_{L^\infty (\Omega_{5 R_0})} + \| \nabla u_\omega  \|_{L^2 (\Omega)}  + | \psi [\omega]| \big )\nonumber \\
& \leq C \big ( M (\alpha,\beta,F, w) +  |b_\Omega [f]| + \| f \|_{L^2 (\Omega_{6 R_0})} \big )\,.\label{proof.thm.nonlinear.14}
\end{align}
On the other hand, we have from \eqref{est.thm.linear.exterior.4'} and by the condition $F_{12}-F_{21}\in L^1 (\Omega)$, for any $\gamma'\in [0,\gamma]$,
\begin{equation}\label{proof.thm.nonlinear.15}
\begin{split}
\| R[\omega] \|_{L^{\infty}_{1+\gamma'} (B^{{\rm c}}_{4R_{0}})} 
&\leq 
\frac{C}{1-\gamma'}  
\bigg ( 
|\alpha|^{-\frac{\gamma'}{2}} \big | \log |\alpha| \big | \, \| G_{\alpha}(\beta, w) + H_{\alpha} (F) \|_{L^\infty_{2+\gamma'}(\Omega)} \\
& \qquad
+ |\alpha|^{-\frac{1+\gamma'}{2}} \| G_{\alpha}(\beta, w) + H_{\alpha} (F) \|_{L^2 (\Omega)}
+ d_{\gamma'} [F] \bigg ) \,, \\
d_{\gamma'} [F] & \, = \,  \sup_{|x|\geq 4 R_0}  |x|^{\gamma'} \big | \int_{2|y|\geq |x|}  (F_{12} -F_{21} ) \dd y \big |\,,
\end{split}
\end{equation}
where  $C$ is independent of  $\gamma'$, $\gamma$, and $\alpha$. Here we have used that $G_\alpha (\beta, w)$ is symmetric and that $U=0$ for $|x|\geq 2 R_0$ by its definition. Note that $d_0[F]\leq \| F_{12}-F_{21} \|_{L^1 (\Omega)}$ holds, which will be used later. 
Combining \eqref{proof.thm.nonlinear.14} with \eqref{proof.thm.nonlinear.15}, 
\eqref{proof.thm.nonlinear.2}, \eqref{proof.thm.nonlinear.3}, 
\eqref{proof.thm.nonlinear.6}, and \eqref{proof.thm.nonlinear.7},
we obtain for $\gamma'\in [0,\gamma]$,
\begin{align}\label{proof.thm.nonlinear.16}
\begin{split}
\| R[\omega] \|_{L^\infty_{1+\gamma'} (\Omega)} 
&\leq 
\frac{C}{1-\gamma'}
\bigg \{ 
| b_\Omega [f] | + \| f \|_{L^2 (\Omega_{6 R_0})} +  |\alpha|^{-\frac{1+\gamma'}{2}}  M (\alpha,\beta, F, w)
+ d_{\gamma'} [F] \\
& ~~~~~ \qquad  +  |\alpha|^{-\frac{\gamma'}{2}} \big | \log |\alpha | \big | \,  \big ( |\alpha| + |\beta| + \| w\|_{L^\infty_{1}(\Omega)} \big ) \| w\|_{L^\infty_{1+\gamma'}(\Omega)} \\
& ~~~~~ \qquad\quad  + |\alpha|^{-\frac{\gamma'}{2}} \big | \log |\alpha| \big | \, \big (  |\alpha| +  \|F\|_{L^\infty_{2+\gamma'}(\Omega)} \big )  \bigg \}\,.
\end{split}
\end{align}
Now we observe that for sufficiently small $\delta_1$ and $\delta_2$ (depending only on $\Omega$ so far) 
the function $M(\alpha,\beta,F,w)$ is bounded from above as
\begin{align}
M(\alpha,\beta,F,w) \leq \big ( |\alpha| + \delta_1 + \delta_2 |\log \delta_1 | \big ) \delta_1  + |\alpha| + \| F\|_{L^2 (\Omega)}\,.\label{proof.thm.nonlinear.16'}
\end{align}
Here we have used the fact that $\rho(r) = r |\log r|$ is monotone increasing on $(0, e^{-1} ]$, which implies $ \| \nabla w\|_{L^2 (\Omega)} |\log \| \nabla w \|_{L^2 (\Omega)}| \leq \delta_1 |\log \delta_1 | $.
By taking \eqref{proof.thm.nonlinear.13}, \eqref{proof.thm.nonlinear.14}, and \eqref{proof.thm.nonlinear.16'} into account, we  assume that  $|\alpha|$, $\|F\|_{L^2(\Omega)}$, $|b_\Omega [f]|$, and $\| f\|_{L^2(\Omega_{6 R_0})}$ are small enough so  that
\begin{align}
\delta_1 = 16 (C_0 + 1) 
\big ( |\alpha| + \|F\|_{L^2(\Omega)} + |b_\Omega [f]| + \| f\|_{L^2(\Omega_{6 R_0})}\big ) 
< \frac{1}{16 (C_0 + 1)}\,.\label{def.delta_1}
\end{align}
Here $C_0$ is the largest constant of $C$ appearing in 
\eqref{proof.thm.nonlinear.13}, \eqref{proof.thm.nonlinear.14}, 
and \eqref{proof.thm.nonlinear.16} (larger than $1$ without loss of generality),
and then,  $C_0$ is independent of $\gamma$ and $\alpha$.
Then for $\delta_2\in (0,\frac{1}{16 (C_0 + 1) |\log \delta_1|}]$ we see from \eqref{proof.thm.nonlinear.16'},
\begin{align}
M(\alpha,\beta,F,w) \leq \frac{1}{4  (C_0 + 1) } \delta_1\,.\label{proof.thm.nonlinear.16''}
\end{align} 
Thus, \eqref{proof.thm.nonlinear.13} and \eqref{proof.thm.nonlinear.14} imply that for $\delta_2\in (0,\frac{1}{16  (C_0 + 1) |\log\delta_1|}]$,
\begin{align*}
|\psi [\omega]| + \| \nabla R[\omega] \|_{L^2 (\Omega)} + \| R[\omega] \|_{L^\infty (\Omega_{5 R_0})}  \leq \frac{\delta_1}{2}~~~~~~{\rm for~all}~~\omega\in \mathcal{B}_{\vec{\delta},\gamma}\,.
\end{align*}
\noindent
Next we focus on $\| R[\omega] \|_{L^\infty_{1}(\Omega)}$. Taking \eqref{proof.thm.nonlinear.16} with $\gamma'=0$  and \eqref{def.delta_1} (with $|\alpha|<e^{-1}$) into account, we set $\delta_2$ as
\begin{align}
\delta_2 = \frac{ 16  (C_0+1) }{|\log \delta_1|} 
\bigg ( |\alpha|^{-\frac{1}{2}}  \delta_1 + \big |\log |\alpha|  \big | \, \big ( |\alpha|  +  \| F \|_{L^\infty_{2} (\Omega)}  \big ) \, + \| F_{12} - F_{21} \|_{L^1 (\Omega)} \bigg )\,,\label{def.delta_2}
\end{align}
which is smaller than $\frac{1}{16 (C_0 + 1) |\log \delta_1|}$ if $|\alpha|$ and the data related to $F$ appearing in \eqref{def.delta_1} and \eqref{def.delta_2} are small enough, while $\delta_2$ is larger than $\delta_1$ 
since $ \delta_{1} \geq |\alpha|$ and $ |\alpha|^{\frac{1}{2}} \big |\log |\alpha| \big | \leq 1$ for $|\alpha| <e^{-1}$.
Note that $d_0[F]\leq \| F_{12}-F_{21} \|_{L^1 (\Omega)}$ is also taken into account in the choice of \eqref{def.delta_2}.
The key observation here is that, when $f=F=0$, the numbers $\delta_1$ and $\delta_2$ are of the order $O(|\alpha|)$ and $O(|\alpha|^{\frac{1}{2}})$ for $|\alpha|\ll 1$, respectively. 
Then the term $C\big |\log |\alpha | \big | \big ( |\alpha| + |\beta| + \| w\|_{L^\infty_{1} (\Omega)} \big )$ in the right-hand side of \eqref{proof.thm.nonlinear.16} with $\gamma'=0$ is bounded from above by 
\begin{align}
C_0 \big |\log |\alpha | \big |  \, \big ( |\alpha| + \delta_1  + \delta_2\big )  
\leq \frac{1}{32}
\,, \label{proof.thm.nonlinear.16''}
\end{align}
if $\gamma\in [0,1)$ and if $|\alpha|$ and the data related to $F$ (and $f={\rm div}\, F$) appearing in \eqref{def.delta_1} and \eqref{def.delta_2} are sufficiently small. 
Note that, since $\delta_2$ is at best of the order $O(|\alpha|^\frac12)$,
the condition $\gamma\in [0,1)$ is crucial to ensure \eqref{proof.thm.nonlinear.16''}.
Precisely, we need the smallness such as
\begin{align}
|\alpha|^{\frac{1}{2}}\, \big |\log |\alpha| \big | + \kappa_\alpha (F) < \epsilon (\Omega)\ll 1\,,
\label{proof.thm.nonlinear.17'}
\end{align}
where
\begin{align}\label{proof.thm.nonlinear.17}
\begin{split}
\kappa_\alpha (F) & =|\alpha|^{-\frac{1}{2}} \big |\log |\alpha| \big |\,   \big ( |b_\Omega [f]|  + \| F \|_{L^2 (\Omega)} + \| f \|_{L^2 (\Omega_{6 R_0})} \big )\\
& \qquad  
+  \big |\log |\alpha| \big | \, \| F_{12}-F_{21} \|_{L^1 (\Omega)}  +  (\log |\alpha|)^2 \,  \| F \|_{L^\infty_{2}(\Omega)}\,.
\end{split}
\end{align}
Here the number $\epsilon (\Omega)$ depends only on $\Omega$ and is independent of $\alpha$ and $\gamma$, 
and we also note that $\kappa_\alpha [F]$ does not contain the number $\gamma$ in its definition.
Under the above smallness condition we have from \eqref{proof.thm.nonlinear.16} with $\gamma'=0$ and the choice of $\delta_2$,
\begin{align*}
\| R[\omega] \|_{L^\infty_{1} (\Omega)}\leq \frac{\delta_2}{2}~~~~~~{\rm for~all}~~\omega \in \mathcal{B}_{\vec{\delta},\gamma}\,,
\end{align*}
as desired. In the above argument the number $\delta_3$ can be arbitrary.

Next we estimate the norm $\| R[\omega] \|_{L^\infty_{1+\gamma}(\Omega)}$ (in the case $\gamma$ is positive). To bound the term 
\begin{align*}
\frac{C}{1-\gamma} |\alpha|^{-\frac{\gamma}{2}} \big |\log |\alpha| \big | \big ( |\alpha| + |\beta| + \| w\|_{L^\infty_{1} (\Omega)} \big )
\end{align*}
in the right-hand side of \eqref{proof.thm.nonlinear.16} with $\gamma'=\gamma$, we need the additional smallness for $\delta_1$ and $\delta_2$ depending on $\gamma$:
\begin{align}
\frac{C_0}{1-\gamma} |\alpha|^{-\frac{\gamma}{2}} \big |\log |\alpha | \big |  \, \big ( |\alpha| + \delta_1  + \delta_2\big )  
\leq \frac{1}{32}\,.
\end{align}
Precisely, in the case $\gamma$ is positive, $\delta_1$ and $\delta_2$ are required to have the smallness as
\begin{align}
|\alpha|^{\frac{1-\gamma}{2}}\, \big |\log |\alpha| \big | + |\alpha|^{-\frac{\gamma}{2}} \kappa_\alpha (F) < \epsilon_{\gamma} (\Omega)\ll 1\,,\label{proof.thm.nonlinear.17''}
\end{align}
where the number $\epsilon_{\gamma} (\Omega)$ depends $\Omega$ on $\gamma$, contrary to the case of $\epsilon (\Omega)$ in \eqref{proof.thm.nonlinear.17'}. We note that $\epsilon _0 (\Omega)=\epsilon (\Omega)$ and $\epsilon_\gamma (\Omega)$ is taken so that it is monotone decreasing and continuous on $\gamma\in [0,1)$ in virtue of \eqref{proof.thm.nonlinear.16}. Then we set $\delta_3$ as 
\begin{align}
\delta_3 \, = \, 
2 \bigg ( 
 |\alpha|^{-\frac{1+\gamma}{2}}  \delta_1 + |\alpha|^{-\frac{\gamma}{2}} \big |\log |\alpha|  \big | \, \| F \|_{L^\infty_{2+\gamma} (\Omega)}  \, + d_\gamma [F]  \bigg )\,,\label{def.delta_3}
\end{align}
Then we can conclude from \eqref{proof.thm.nonlinear.16}  with $\gamma'=\gamma$ and \eqref{proof.thm.nonlinear.16''} that 
\begin{align*}
\| R[\omega] \|_{L^\infty_{1+\gamma} (\Omega)} \leq \frac{\delta_3}{2} ~~~~~~ {\rm for~all} ~~\omega \in \mathcal{B}_{\vec{\delta},\gamma}\,.
\end{align*}
It should be emphasized here that the argument works even if $\delta_3$ itself  is large.
We have now shown that $\Phi$ is a mapping from $\mathcal{B}_{\vec{\delta},\gamma}$ into $\mathcal{B}_{\vec{\delta},\gamma}$ with the choice of $\delta_j$ in \eqref{def.delta_1}, \eqref{def.delta_2}, and \eqref{def.delta_3} for $j=1,2,3,$ respectively.

Next let us show that $\Phi$ is a contraction mapping on $\mathcal{B}_{(\delta_1,\delta_2,\delta_2),0}$. For convenience we set $\vec{\beta} = (\beta_1, \beta_2)$, and ${\bf w} = (w_1, w_2)$ for $\omega_j = (\beta_j, w_j) \in \mathcal{B}_{(\delta_1,\delta_2,\delta_2),0}$, $j=1,2$. We also set 
\begin{align}
h \, = \, \big ( \psi [\omega_1] - \psi [\omega_2] \big ) V + R [\omega_1] - R[\omega_2]\,,\label{def.h}
\end{align}
which is equal to $u_{\omega_1}-u_{\omega_2}$, and hence, the velocity $h$ satisfies
\begin{equation*}
  \left\{\begin{aligned}
  -\Delta h - \alpha ( x^\bot \cdot \nabla h - h^\bot ) + \nabla q  & \, = \,  {\rm div}\, G'_{\alpha}(\vec{\beta}, {\bf w})\,,  ~~~  {\rm div}\, h \,=\, 0\,, ~~~~~  x \in \Omega\,, \\
  h  & \,=\, 0 \,,  ~~~~~~ x \in \partial \Omega\,, \\
  h & \, \rightarrow \,     0   \,,  ~~~~~~ |x| \rightarrow \infty\,,
\end{aligned}\right.
\end{equation*}
where $q=q_{\omega_1}-q_{\omega_2}\in W^{1,2}_{loc} (\overline{\Omega})$. Here $G'_{\alpha} (\vec{\beta}, {\bf w} )$ is given by
\begin{align*}
G'_{\alpha} (\vec{\beta}, {\bf w} ) & \, = \,  -\alpha ( U \otimes (w_1 - w_2) + (w_1 - w_2) \otimes U )  - (\beta_1 - \beta_2) ( V \otimes w_1 + w_1 \otimes V ) \\
 & ~~~
- \beta_2 ( V \otimes (w_1 - w_2) + (w_1 - w_2) \otimes V )  - w_1 \otimes (w_1 - w_2) -  (w_1 - w_2) \otimes w_2 \,.
\end{align*}
Below we give the estimates of $G'_\alpha (\vec{\beta}, {\bf w})$, where the estimate for the $L^2$ norm of the term $V\otimes w_1 + w_1 \otimes V$ has to be carefully computed: in principle, we need to estimate it by $\delta_1$ rather than $\delta_2$, for their dependence on $|\alpha|$ is essentially different. Due to the negative power on $|\alpha|$ in the linear estimate \eqref{est.thm.linear.exterior.4'}  this is crucial to show that $\Phi$ is a contraction mapping. Because of this reasoning we apply \eqref{est.lem.A.1.2} in Lemma A.1 by recalling the bound $|V(x)|\leq C (1+|x|)^{-1}$, which yields
\begin{align}
\| V \otimes w_1 + w_1 \otimes V \|_{L^2 (\Omega)} \leq C \| \nabla w_1 \|_{L^2 (\Omega)} | \log \| \nabla w_1 \|_{L^2 (\Omega)} |\,.
\end{align} 
Here we have used the smallness of $\| \nabla w_1 \|_{L^2 (\Omega)} + \| w_1\|_{L^\infty_{1} (\Omega)}$. Similarly, also for the nonlinear term in $G'_\alpha (\vec{\beta},{\bf w})$ we will apply \eqref{est.lem.A.1.2}. Then it follows that
\begin{align}
& ~~~ \| G'_{\alpha} (\vec{\beta}, {\bf w} ) \|_{L^2 (\Omega)} \nonumber \\
& \leq C \big ( |\alpha| \, \|\nabla (w_1-w_2) \|_{L^2 (\Omega)} + |\beta_1-\beta_2| \, \| \nabla w_1\|_{L^2 (\Omega)} |\log \| \nabla w_1 \|_{L^2 (\Omega)} |  \nonumber \\
& ~~~ + |\beta_2| \, \| w_1-w_2 \|_{L^\infty_{1} (\Omega)} +  \| w_1-w_2 \|_{L^\infty_{1} (\Omega)} \| \nabla {\bf w} \|_{L^2 (\Omega)} \big | \log\| \nabla {\bf w} \|_{L^2 (\Omega)} \big \|  \big )  \nonumber\\
& \leq C \big ( |\alpha| \, \|\nabla (w_1-w_2) \|_{L^2 (\Omega)} + \delta_1 |\log \delta_1| \,  |\beta_1-\beta_2|  +  3 \delta_1 |\log \delta_1 | \| w_1-w_2 \|_{L^\infty_{1} (\Omega)} \big ) \nonumber \\
& \leq C (|\alpha| +  \delta_1 |\log \delta_1 | ) \| \omega_1-\omega_2 \|_{X_0}\,,\label{proof.thm.nonlinear.18}
\end{align}
and on the other hand, it is not difficult to see
\begin{align}
 \| G'_{\alpha}(\vec{\beta}, {\bf w}) \|_{L^\infty_{2} (\Omega)} & \le 
C \big ( |\alpha| \, \| w_1-w_2 \|_{L^\infty_{1} (\Omega)}  + | \beta_1 - \beta_2 | \, \| w_1 \|_{L^\infty_{1} (\Omega)} \nonumber \\
&~~~ + | \beta_2 | \, \| w_1-w_2 \|_{L^\infty_{1} (\Omega)}  +  \| {\bf w}\|_{L^\infty_{1} (\Omega)}  \| w_1 - w_2\|_{L^\infty_{1} (\Omega)} \big ) \nonumber \\
& \le
C \big ( \delta_2  | \beta_1 - \beta_2 |    + ( |\alpha | + \delta_1 + 2\delta_2 )  \| w_1-w_2 \|_{L^\infty_{1} (\Omega)} \big ) \nonumber \\
& \leq C (|\alpha| + \delta_1 +  \delta_2  ) \| \omega_1-\omega_2 \|_{X_0}\,.\label{proof.thm.nonlinear.19}
\end{align}
Similarly, we observe that
\begin{align}
& ~~~ \| {\rm div}\,  G'_{\alpha} (\vec{\beta}, {\bf w} ) \|_{L^2 (\Omega_{5 R_0})} \nonumber \\
& \leq C \big ( |\alpha | \| \nabla (w_1-w_2) \|_{L^2 (\Omega)} + |\beta_1-\beta_2 | \| \nabla  w_1 \|_{L^2 (\Omega_{5 R_0}) }  + |\beta_2 | \| \nabla (w_1-w_2) \|_{L^2 (\Omega_{5 R_0})}  \nonumber \\
& ~~~ + \| w_1\|_{L^\infty (\Omega_{5 R_0})}  \| \nabla ( w_1-w_2 ) \|_{L^2 (\Omega)} + \|\nabla w_2 \|_{L^2 (\Omega)} \| w_1-w_2 \|_{L^\infty (\Omega)}\big ) \nonumber \\
& \leq  C \big ( |\alpha | \| \nabla (w_1-w_2) \|_{L^2 (\Omega)} + 
 \delta_1  |\beta_1-\beta_2 | + \delta_1 \| \nabla (w_1-w_2) \|_{L^2 (\Omega)} \nonumber \\
&~~~ + \delta_1   \| \nabla (w_1-w_2 ) \|_{L^2 (\Omega)} +  \delta_1 \| w_1-w_2 \|_{L^\infty_1 (\Omega)}\big ) \nonumber \\
& \leq C ( |\alpha| +  \delta_1) \| \omega_1-\omega_2 \|_{X_0}\,.\label{proof.thm.nonlinear.20}
\end{align}
By applying Theorem \ref{thm.linear.exterior}, we have the representation of the velocity $h$ as 
\begin{equation}
h \, = \, \bigg ( \int_{\partial \Omega} y^\bot \cdot T(h, q) \nu \, \dd \sigma_y \bigg )V 
+ \mathcal{R}_\Omega [{\rm div}\, G'_\alpha (\vec{\beta}, {\bf w})]\,.\label{def.h'}
\end{equation}
Here we have used $b_\Omega [{\rm div}\, G'_\alpha (\vec{\beta}, {\bf w})]=0$ again, which follows from the symmetry of  $G'_\alpha (\vec{\beta},{\bf w})$ and from the fact that the trace of $G'_\alpha (\vec{\beta},{\bf w})$ on $\partial\Omega$ is zero. Since $h=u_{\omega_1}-u_{\omega_2}$ and $q=q_{\omega_1}-q_{\omega_2}$, we see from the definitions of $T(h,q)$ and $\psi [\omega_j]$ in \eqref{proof.thm.nonlinear.12'},
\begin{align*}
\int_{\partial \Omega} y^\bot \cdot T(h, q) \nu \, \dd \sigma_y & \, = \, \psi [\omega_1] - \psi [\omega_2]\,,
\end{align*}
and thus, we also have from \eqref{def.h} and \eqref{def.h'},
\begin{align*}
\mathcal{R}_\Omega [{\rm div}\, G'_\alpha (\vec{\beta}, {\bf w})] \, = \, R [\omega_1] - R[\omega_2]\,.
\end{align*}
In virtue of \eqref{est.thm.linear.exterior.0} - \eqref{est.thm.linear.exterior.2} we see 
\begin{align}
\big | \int_{\partial \Omega} y^\bot \cdot T(h, q) \nu \, \dd \sigma_y \big | & \leq C \big (\| \nabla h \|_{W^{1,2}(\Omega_{4 R_0})} + \| q \|_{W^{1,2} (\Omega_{4R_0})} \big ) \nonumber \\
& \leq  C \big ( \|G'_\alpha (\vec{\beta}, {\bf w}) \|_{L^2 (\Omega)} + \| {\rm div}\, G'_\alpha (\vec{\beta},{\bf w})\|_{L^2 (\Omega_{5 R_0})}  \big )\,.\label{proof.thm.nonlinear.23}
\end{align}
A similar argument as in the derivation of  \eqref{proof.thm.nonlinear.14} yields
\begin{align}
\begin{split}
& \| \mathcal{R}_\Omega [{\rm div}\, G'_\alpha (\vec{\beta}, {\bf w})] \|_{L^\infty (\Omega_{4 R_0})} + \| \nabla \mathcal{R}_\Omega [{\rm div}\, G'_\alpha (\vec{\beta}, {\bf w})]  \|_{L^2 (\Omega)} \\
&  \leq  C \big (\| G'_{\alpha} (\vec{\beta}, {\bf w} ) \|_{L^2 (\Omega)}  +  \| {\rm div}\,  G'_{\alpha} (\vec{\beta}, {\bf w} ) \|_{L^2 (\Omega_{5 R_0})} \big )\,.\label{proof.thm.nonlinear.21}
\end{split}
\end{align}
Moreover, by applying \eqref{est.thm.linear.exterior.4'} we see that the term $\mathcal{R}_\Omega [{\rm div}\, G'_\alpha (\vec{\beta}, {\bf w})] $ satisfies 
\begin{equation}\label{proof.thm.nonlinear.22}
\begin{split}
& 
\| \mathcal{R}_\Omega [{\rm div}\, G'_\alpha (\vec{\beta}, {\bf w})]  \|_{L^{\infty}_{1} (B^{{\rm c}}_{4R_{0}})} \\
& \le 
C \bigg ( |\alpha|^{-\frac{1}{2}} \| G'_\alpha (\vec{\beta}, {\bf w} ) \|_{L^2 (\Omega)}  + \big | \log |\alpha | \big |   \| G'_\alpha (\vec{\beta}, {\bf w} ) \|_{L^\infty_{2} (\Omega)} \bigg ) \,. 
\end{split}
\end{equation}
Here we have used again the symmetry of $G'_\alpha (\vec{\beta}, {\bf w})$. 
Combining  \eqref{proof.thm.nonlinear.23}, \eqref{proof.thm.nonlinear.21}, and \eqref{proof.thm.nonlinear.22} with  \eqref{proof.thm.nonlinear.18}, \eqref{proof.thm.nonlinear.19}, and \eqref{proof.thm.nonlinear.20}, we obtain for sufficiently small  $|\alpha|\ne 0$ and $\kappa_\alpha [F]$ in \eqref{proof.thm.nonlinear.17}, 
\begin{align}
& \| \Phi [\omega_1] - \Phi [\omega_2] \|_{X_0}  \nonumber \\
& \, =\,  | \psi [\omega_1] - \psi [\omega_2] | + \| \nabla \big (R[\omega_1] - R [\omega_2] \big ) \|_{L^2 (\Omega)} + \| R[\omega_1] - R[\omega_2] \|_{L^\infty_{1} (\Omega)} \nonumber \\
& \leq C \bigg ( |\alpha|^{-\frac{1}{2}} \big ( |\alpha| + \delta_1|\log \delta_1|   \big ) + \big | \log |\alpha| \big | \, \big ( |\alpha| + \delta_1 + \delta_2 ) \bigg ) \| \omega_1-\omega_2 \|_{X_0}\nonumber 
\\
& \leq \frac34 \| \omega_1 - \omega_2 \|_{X_0}\,,\label{proof.thm.nonlinear.23'}
\end{align}
that is, the map $\Phi$ is a contraction on $ \mathcal{B}_{(\delta_1,\delta_2,\delta_2),0} $.
Here we have used the estimates $|\log \delta_1| \leq \big |\log |\alpha| \big | $ 
and $ \delta_1 \leq  2^{-1} |\alpha|^{\frac{1}{2}} \big |\log |\alpha| \big |^{-1} $ 
if $\delta_1 \geq |\alpha|$ and the data related to $F$ appearing \eqref{def.delta_1} are small enough.
Therefore, there exists a fixed point $\omega=(\beta, w)$ of $\Phi$ in $\mathcal{B}_{\vec{\delta},\gamma}$, which is unique in $\mathcal{B}_{(\delta_1,\delta_2,\delta_2),0}$. 
By the definition of $\Phi$ in \eqref{def.Phi}, we see that the fixed point $\omega=(\beta, w)$ satisfies
\begin{align*}
u_\omega \, =\,  u_{(\beta,w)} \, =\,  \psi [\omega] V + R[\omega] \, = \, \beta V + w\,, 
\end{align*}
which is the solution to \eqref{NSad}, as desired. 
Let us set $v = \beta V + w$ for the fixed point $(\beta, w) \in \mathcal{B}_{\vec{\delta},\gamma}$.
The local regularity of $v\in W^{2,2}_{loc} (\overline{\Omega})^2$ as well as $\nabla q \in L^2_{loc}(\overline{\Omega})^2$ follows from the standard elliptic regularity of the Stokes operator by regarding the nonlinear term, which belongs to $L^2 (\Omega)^2$ by the above construction,  as a given external force. This leads to the regularity $u\in W^{2,2}_{loc}(\overline{\Omega})^2$ and $\nabla p\in L^2_{loc} (\overline{\Omega})^2$ for the solution $(u,\nabla p)$ to \eqref{NSa} by \eqref{proof.thm.nonlinear.0}.  
Next we observe that $v=\beta V + w$ solves 
\begin{equation}\tag{NS$_{\alpha}''$}\label{NSadd}
 \left\{ \begin{aligned}
  -\Delta v - \alpha ( x^\bot \cdot \nabla w - w^\bot ) + \nabla \tilde q   & \,=\,  - {\rm div}\, (\alpha U \otimes v + v \otimes \alpha U + v \otimes v) \\
&  \quad\quad + {\rm div}\, H_{\alpha} (F)\,, ~~~~~~ x\in \Omega\,,\\
 {\rm div}\, v & \,=\, 0\,, ~~~~~~ x \in \Omega\,, \\
  v  & \,=\, 0 \,,  ~~~~~~ x \in \partial \Omega\,, \\ 
  v  & \, \rightarrow \,     0   \,,  ~~~~~~ |x| \rightarrow \infty\,.
\end{aligned}\right.
\end{equation}
Here we have used the identity $x^\bot \cdot \nabla V - V^\bot =0$ by the definition of $V$.
Let us take the approximation of $v$ of the form
\begin{align}
v^{(N)} = \chi_N \beta V + w^{(N)}\,,~~~~~w^{(N)} = \chi_N w - \mathbb{B}_N [\nabla \chi_N \cdot w]\,,~~~~~ N\gg 1\,,
\end{align}
where $\chi_N(|x|)$ is the radial cut-off function satisfying $\chi_N=1$ for $|x|\leq N$, $\chi_N=0$ for $|x|\geq 2 N$, and $|\nabla \chi_N|\leq C N^{-1}$, while $\mathbb{B}_N$ is the Bogovskii operator in the  closed annulus 
$A_N = \{ N \leq |x|\leq 2 N\}$ which satisfies 
\begin{align*}
{\rm supp}\, \mathbb{B}_N [\nabla \chi_N \cdot w] \subset A_N\,,  ~~~~~~ {\rm div}\,  \mathbb{B}_N [\nabla \chi_N \cdot w] = \nabla \chi_N \cdot w
\end{align*}
and 
\begin{align}
N^{-1} \| \mathbb{B}_N [\nabla \chi_N \cdot w] \|_{L^2 (\Omega)} + \| \nabla \mathbb{B}_N [\nabla \chi_N \cdot w] \|_{L^2 (\Omega)} &\leq C\| \nabla \mathbb{B}_N [\nabla \chi_N \cdot w] \|_{L^2 (\Omega)}  \nonumber \\
& \leq C \| \nabla \chi_N \cdot w \|_{L^2 (\Omega)}\,.\label{est.Bogovskii}
\end{align}
Here $C$ is independent of $N$; see, e.g. Borchers and Sohr \cite[Theorem 2.10]{BS}. Then, by multiplying $v^{(N)}$ both sides of the first equation in \eqref{NSadd} and integrating over $\Omega$, we obtain
\begin{align}
\begin{split}
& \langle \nabla v, \nabla v^{(N)}\rangle_{L^2 (\Omega)} + \alpha \langle w, x^\bot \cdot \nabla w^{(N)} - (w^{(N)})^\bot \rangle _{L^2 (\Omega)} \\
& \, = \, 
\langle v \otimes  v + \alpha U \otimes \bar{v} + v \otimes \alpha U , \nabla v^{(N)}\rangle _{L^2 (\Omega)}
- \langle H_\alpha (F), \nabla v^{(N)}\rangle_{L^2 (\Omega)}
\end{split}
\end{align} 
from the integration by parts. Here we have used again the identity for the radial circular flow: $x^\bot \cdot \nabla (\chi_N V) - \chi_N V^\bot =0$. It is easy to see from \eqref{est.Bogovskii} and 
$ w \in \dot{W}^{1,2}_{0,\sigma} (\Omega) \cap L^\infty_{1+\gamma}(\Omega)^2$ that 
\begin{align*}
\langle \nabla v, \nabla v^{(N)}\rangle_{L^2 (\Omega)}  & \,  \rightarrow \, \langle \nabla v, \nabla v \rangle_{L^2 (\Omega)}\,,\\
\langle v \otimes v , \nabla v^{(N)}\rangle _{L^2 (\Omega)} &\,  \rightarrow \, \langle v \otimes v, \nabla v \rangle _{L^2 (\Omega)} \, = \, 0\,,\\
\langle \alpha U \otimes v + v \otimes \alpha U , \nabla v^{(N)}\rangle _{L^2 (\Omega)} &\,  
\rightarrow \, \langle \alpha U \otimes  v + v \otimes \alpha U , \nabla v \rangle _{L^2 (\Omega)} 
\, = \, \alpha \langle 
U \otimes v, 
\nabla v \rangle _{L^2 (\Omega)}\,,\\
\langle H_\alpha (F), \nabla v^{(N)}\rangle _{L^2 (\Omega)} & \, \rightarrow \, \langle H_\alpha (F), \nabla v \rangle_{L^2 (\Omega)}\,,
\end{align*}
as $N\rightarrow \infty$. As for the term $\langle w,  (w^{(N)})^\bot \rangle _{L^2 (\Omega)}$ we see 
\begin{align*} 
| \langle w,  (w^{(N)})^\bot \rangle _{L^2 (\Omega)} |  \, 
& = \, | \langle w, \mathbb{B}_N [\nabla \chi_N \cdot w]^\bot \rangle _{L^2 (\Omega)}| \\
& \leq \| w\|_{L^2 (\{ N\leq |x|\leq 2 N\})} \| \mathbb{B}_N [\nabla \chi_N \cdot w] \|_{L^2 (\Omega)} \\
& \leq C N \| w\|_{L^2 (\{ N\leq |x|\leq 2 N\})} \| \nabla \chi_N \cdot w \|_{L^2 (\Omega)}\\
& \leq C N^{-2\gamma} \| w\|_{L^\infty_{1+\gamma} (\Omega)}^2\\
&\begin{cases}
& \rightarrow 0~~~~~(N\rightarrow \infty) \quad\quad {\rm if} \quad \gamma>0\,,\\
& \leq C \| w \|_{L^\infty_1 (\Omega)}^2 \quad\quad \quad ~ {\rm if} \quad \gamma=0\,.
\end{cases}
\end{align*}
It remains to consider the term $\langle w, x^\bot \cdot \nabla w^{(N)} \rangle _{L^2 (\Omega)}$. From the integration by parts and from $x^\bot \cdot \nabla \chi_N=0$,  $ {\rm div} (x^\bot \chi_N) = 0$, and ${\rm supp}\,  \mathbb{B}_N [\nabla \chi_N \cdot w] \subset A_N$ we have 
\begin{align*}
| \langle w, x^\bot \cdot \nabla w^{(N)} \rangle _{L^2 (\Omega)} | & \, =\, |  \langle w, x^\bot \cdot \nabla \mathbb{B}_N [\nabla \chi_N \cdot w] \rangle _{L^2 (\Omega)} | \\
& \leq N \| w \|_{L^2 ( \{N\leq |x|\leq 2N\})} \| \nabla \mathbb{B}_N [\nabla \chi_N\cdot w] \|_{L^2 (\Omega)} \\
& \leq C N^{-2\gamma} \| w\|_{L^\infty_{1+\gamma} (\Omega)}^2\\
&\begin{cases}
& \rightarrow 0~~~~~(N\rightarrow \infty) \quad\quad {\rm if} \quad \gamma>0\,,\\
& \leq C \| w \|_{L^\infty_1 (\Omega)}^2 \quad\quad \quad ~ {\rm if} \quad \gamma=0\,.
\end{cases}
\end{align*}
Here we have also used \eqref{est.Bogovskii}. Collecting these above, we have arrived at the identity 
\begin{align}\label{energy.identity}
\langle \nabla v, \nabla v \rangle_{L^2 (\Omega)} \, = \,  
\alpha \langle 
U \otimes v
, \nabla v \rangle _{L^2 (\Omega)}
- \langle H_\alpha (F), \nabla v \rangle_{L^2 (\Omega)} ~~~~~~~ {\rm when} \quad \gamma>0\,.
\end{align}
In particular, from the Poincar${\rm \acute{e}}$ inequality $|\langle U \otimes v, \nabla v \rangle _{L^2 (\Omega)}|\leq C \| \nabla v \|_{L^2 (\Omega)}^2$ we obtain the estimate 
\begin{align}\label{proof.thm.nonlinear.24'}
(1-C |\alpha|) \| \nabla v \|_{L^2 (\Omega)}^2 \leq \| F + \alpha \nabla U \|_{L^2 (\Omega)}^2 \qquad ~~ {\rm when} \quad \gamma>0\,,
\end{align}
which shows \eqref{est.thm.nonlinear.1} for the case $\gamma>0$ by the relation $u=\alpha U + v$. 
Note that the constant $C$ in \eqref{proof.thm.nonlinear.24'} depends only on $R_0$ and is independent of $\alpha$ and $\gamma$. To obtain the energy inequality for the case $\gamma=0$ we first consider the approximation of  $F$ and $f$ such that  
\begin{align}
F_{n} (x) \, = \, e^{-\frac{1}{n} |x|^2} F(x)\,,  
\qquad  f_n \,=\, {\rm div} \, F_n \,.\label{proof.thm.nonlinear.25}
\end{align}
Then $F_n \in L^\infty_{2+\gamma} (\Omega)^{2\times 2}$ for $\gamma>0$ and 
\begin{align}\label{proof.thm.nonlinear.24}
\begin{split}
&\lim_{n\rightarrow \infty} b_\Omega [f_n-f]  \,=\,
\lim_{n\rightarrow \infty} \|  F- F_n \|_{L^2 (\Omega)}  \,=\,
\lim_{n \rightarrow \infty} \| f_n -f \|_{L^2 (\Omega_{6 R_0})}  \,=\, 0 \,,\\
& \lim_{n \rightarrow \infty} \|(F - F_n )_{12} - (F - F_n)_{21} \|_{L^1 (\Omega)} \,=\, 0 \,, \qquad
\| F_n \|_{L^\infty_2 (\Omega)}  \leq \| F\|_{L^\infty_2 (\Omega)}\,.
\end{split}
\end{align}
Here we have used the condition $F_{12}-F_{21}\in L^1 (\Omega)$ for the convergence of $b_\Omega [f_n]$. Assume that 
\begin{align*}
|\alpha|^\frac12 \big |\log |\alpha| \big | + \kappa_\alpha [F] <\epsilon (\Omega)\,,
\end{align*}
and we fix $\alpha$. Then there is a unique fixed point $(\beta,w)$ of $\Phi$ in $\mathcal{B}_{(\delta_1,\delta_2,\delta_2),0}$. On the other hand, since $\alpha$ is fixed, there is $\gamma_0>0$ such that 
\begin{align*}
\sup_{0\leq \gamma\leq \gamma_0} \big (|\alpha|^\frac{1-\gamma}{2} \big |\log |\alpha| \big | + |\alpha|^{-\frac{\gamma}{2}} \kappa_\alpha [F]  \big ) <\epsilon_{\gamma_0} (\Omega)\,.
\end{align*}
Here we have used the fact that $\epsilon_0 (\Omega)=\epsilon (\Omega)$ and $\epsilon_\gamma (\Omega)$ is continuous on $\gamma\in [0,1)$. Hence, in view of \eqref{proof.thm.nonlinear.24} and \eqref{proof.thm.nonlinear.17}, there is $N\gg 1$ such that 
\begin{align*} 
\sup_{n\geq N} \sup_{0\leq \gamma\leq \gamma_0} \big (|\alpha|^\frac{1-\gamma}{2}  \big |\log |\alpha| \big |  +  |\alpha|^{-\frac{\gamma}{2}} \kappa_\alpha [F_n]  \big ) <\epsilon_{\gamma_0} (\Omega)\,.
\end{align*}
Let $(v_n, \nabla \tilde q_n)$ with $v_n =  \beta_n V  + w_n$, $n\geq N$, be the unique solution to \eqref{NSadd} with $F$ replaced by $F_n$ such that $(\beta_n, w_n) \in \mathcal{B}_{(\delta_1,\delta_2, \delta_{3}^{(n)}),\gamma}\subset \mathcal{B}_{(\delta_1,\delta_2,\delta_2),0}$ with some $\gamma\in (0,\gamma_0]$. Note that for sufficiently large $n$, we can take the same $\delta_1$ and $\delta_2$. Then \eqref{energy.identity} implies 
\begin{align}\label{energy.identity'}
\| \nabla v_n \|_{L^2 (\Omega)}^2 \, = \,  
\alpha \langle 
U \otimes v_n
, \nabla v_n \rangle _{L^2 (\Omega)} 
- \langle H_\alpha (F), \nabla v_n \rangle_{L^2 (\Omega)}\,.
\end{align}
Since $(\beta_n,w_n)\in \mathcal{B}_{(\delta_1,\delta_2,\delta_2),0}$ we have uniform estimates of $(v_n, \nabla \tilde q_n)$, and thus, we find a subsequence, denoted again by  $(v_n, \nabla \tilde q_n)$, such that $\beta_n\rightarrow \beta_\infty$,
\begin{align*}
& w_n \rightharpoonup w_\infty \quad {\rm in} 
\quad  W^{2,2}_{loc} (\overline{\Omega})^2\,,~~~~~~~~~\tilde q_n \rightharpoonup \tilde q_\infty \quad {\rm in} \quad W_{loc}^{1,2} (\overline{\Omega})\,,\\
& \nabla w_n \rightharpoonup \nabla w_\infty \quad {\rm in}\quad  L^2 (\Omega)^{2\times 2}\,,
~~~~~~~ 
w_n \rightharpoonup^* w_\infty \quad {\rm in}\quad L^\infty_1 (\Omega)^2\,,
\end{align*}
and $w_n \rightarrow w_\infty$ strongly in $W^{1,2}_{loc} (\overline{\Omega})^2$. Moreover, we observe from \eqref{energy.identity}
that $v_\infty =\beta_\infty V + w_\infty$ satisfies the energy inequality
\begin{align}\label{energy.identity''}
\| \nabla v_\infty \|_{L^2 (\Omega)}^2 
\leq  
\alpha \ \langle 
U \otimes v_\infty
, \nabla v_\infty \rangle _{L^2 (\Omega)} 
- \langle H_\alpha (F), \nabla v_\infty \rangle_{L^2 (\Omega)}\,.
\end{align}
It is also easy to see that $(v_\infty, \nabla \tilde q_\infty)$ is a solution to \eqref{NSadd} and $(\beta_\infty, w_\infty)\in \mathcal{B}_{(\delta_1,\delta_2,\delta_2),0}$. By the uniqueness of the fixed point of $\Phi$ in $\mathcal{B}_{(\delta_1,\delta_2,\delta_2),0}$ , we have $(\beta_\infty, w_\infty)=(\beta,w)$.
Therefore, \eqref{energy.identity''} holds with $v_\infty$ replaced by $v=\beta V + w$, as desired.
Thus we have \eqref{est.thm.nonlinear.1} also when $F\in L^\infty_2 (\Omega)^{2\times 2}$ and $F_{12}-F_{21}\in L^1 (\Omega)$.

The estimates \eqref{est.thm.nonlinear.3} and \eqref{est.thm.nonlinear.4}  follow from the fact  $\| w\|_{L^\infty_{1} (\Omega)} \leq \delta_2$ and $\| w\|_{L^\infty_{1+\gamma} (\Omega)}\leq \delta_3$ together with the definitions of $\delta_j$ in \eqref{def.delta_2}, \eqref{def.delta_3}, and $d_\gamma [F] \leq C\gamma^{-1}\| F\|_{L^\infty_{2+\gamma} (\Omega)}$ when $\gamma>0$. As for the identity \eqref{est.thm.nonlinear.2} on the coefficient $\beta$, we observe from \eqref{proof.thm.nonlinear.12'},
\begin{align*}
\beta \, = \, \int_{\partial\Omega} y^\bot \cdot \big (T(v, q) \nu \big ) \, \dd \sigma_y + b_\Omega [f]\,.
\end{align*} 
Since $v=u-\alpha x^\bot$ and $q=p + P$ near $\partial\Omega$, where $P=P(|x|)$ is a radial function and has been taken so that $\nabla P ={\rm div}\, [(\alpha U + \beta V) \otimes (\alpha U + \beta V)]$, the straightforward calculations yield 
\begin{align*}
\int_{\partial \Omega} y^\bot \cdot \big (T(v, q) \nu \big ) \, \dd \sigma_{y} \, =\, \int_{\partial \Omega} y^\bot \cdot  \big ( T(u, p) \nu \big ) \, \dd \sigma_{y} \,.
\end{align*}
Thus \eqref{est.thm.nonlinear.2} holds. The proof of Theorem \ref{thm.nonlinear} is complete.
\end{proofx}

Finally we consider the case $F\in L^\infty_{2,0}(\Omega)^{2\times 2}$. Combining Theorem \ref{thm.nonlinear} 
with Theorem \ref{thm.nonlinear'} below, we obtain Theorem \ref{thm.main}.

\begin{theorem}\label{thm.nonlinear'} 
Assume that $f={\rm div}\, F$ satisfies the conditions in Theorem \ref{thm.nonlinear} for $\gamma=0$. 
Assume in addition that $F\in L^\infty_{2,0}(\Omega)^{2\times 2}$.
Then the remainder $w$ in Theorem \ref{thm.nonlinear} belongs to $L^\infty_{1,0} (\Omega)^2$.
\end{theorem}

\begin{proof} 
The proof is very similar to the derivation of the energy inequality for the case $\gamma=0$ in the proof of Theorem \ref{thm.nonlinear}. We set $F_n$ and $f_n$ as in \eqref{proof.thm.nonlinear.25}. Then $F_n$ and $f_n$ satisfy \eqref{proof.thm.nonlinear.24}, and moreover, the additional condition $F\in L^\infty_{2,0} (\Omega)^{2\times 2}$ implies
\begin{align}\label{proof.thm.nonlinear'.1}
\| F_n - F \|_{L^\infty_2 (\Omega)} \rightarrow 0\,, \qquad n\rightarrow \infty\,.
\end{align}
The proof of \eqref{proof.thm.nonlinear'.1} is as follows: for any small number $\epsilon>0$, there exists $R>0$ such that
$\| F_n - F \|_{L^{\infty}_{2} (B^{{\rm c}}_{R})} \le 2\epsilon \,\|F\|_{L^\infty_2 (\Omega)}$ by the decay condition $F\in L^\infty_{2,0} (\Omega)^{2\times 2}$. Then we have
\begin{equation*}
\begin{split}
\limsup_{n \rightarrow \infty} \| F_n - F \|_{L^\infty_2 (\Omega)}
& \le \limsup_{n \rightarrow \infty} 
\big( \| F_n - F \|_{L^{\infty}_{2} (B_{R})} + \| F_n - F \|_{L^{\infty}_{2} (B^{{\rm c}}_{R})} \big) \\
& \le \limsup_{n \rightarrow \infty} 
\big( (1- e^{ -\frac{R^2}{n}}) + 2\epsilon \big)\|F\|_{L^\infty_2 (\Omega)}
\,=\, 2\epsilon\,\|F\|_{L^\infty_2 (\Omega)}\,,
\end{split}
\end{equation*}
which implies \eqref{proof.thm.nonlinear'.1}. As in the proof of Theorem \ref{thm.nonlinear}, let $(v_n,\nabla q_n)$, $v_n = \beta_n V + w_n$, $n \gg 1$, be the solution to \eqref{NSadd} with $F$ replaced by $F_n$ such that $(\beta_n, w_n) \in \mathcal{B}_{(\delta_1,\delta_2, \delta_{3,}^{(n)}),\gamma}\subset \mathcal{B}_{(\delta_1,\delta_2,\delta_2),0}$  with some $\gamma\in (0, 1)$. 
Since $w_n\in L^\infty_{1+\gamma} (\Omega)^2$ and $\gamma>0$, 
it suffices to show that $(\beta_n,w_n)$ converges to $(\beta,w)$ in $\R\times L^\infty_1 (\Omega)^2$,
where $v=\beta V + w$ is the solution to \eqref{NSadd}. To prove this we observe that the difference $h=v-v_n$ solves 
\begin{equation*}
 \left\{\begin{aligned}
-\Delta h - \alpha ( x^\bot \cdot \nabla h - h^\bot ) + \nabla q  & \, = \, 
{\rm div}\, G'_{\alpha}(\vec{\beta}, {\bf w}) + {\rm div}\, (F-F_n)\,, ~~~~~  x \in \Omega\,, \\
{\rm div}\, h & \,=\, 0\,, ~~~~~  x \in \Omega\,, \\
  h  & \,=\, 0 \,,  ~~~~~ x \in \partial \Omega\,. \\
  h  & \, \rightarrow \,     0   \,,  ~~~~~~ |x| \rightarrow \infty\,.
\end{aligned}\right.
\end{equation*}
Here we have set $\vec{\beta}=(\beta,\beta_n)$, ${\bf w} = (w,w_n)$, and 
\begin{align*}
G'_{\alpha} (\vec{\beta}, {\bf w} ) & \, = \,  -\alpha ( U \otimes (w - w_n) + (w - w_n) \otimes U )  - (\beta - \beta_n) ( V \otimes w + w \otimes V ) \\
 & ~~~ - \beta_n ( V \otimes (w - w_n) + (w - w_n) \otimes V )  - w \otimes (w - w_n) -  (w - w_n) \otimes w_n \,.
\end{align*}
Then the same argument as in the derivation of \eqref{proof.thm.nonlinear.23'} shows 
\begin{align*}
\| (\beta,w) - (\beta_n, w_n) \|_{X_0}  
& \leq \frac34 \| (\beta,w)-(\beta_n,w_n) \|_{X_0} \\
& \qquad + C \bigg ( 
|b_\Omega [f-f_n]| + \| F- F_n \|_{L^2 (\Omega)} + \| f-f_n \|_{L^2 (\Omega_{6 R_0})}\\
& \qquad\quad + \| (F-F_n)_{12} - (F-F_n)_{21} \|_{L^1 (\Omega)} + \| F-F_n \|_{L^\infty_2 (\Omega)}  \bigg )\,,
\end{align*} 
where $C$ is independent of $n$. Thus, $(\beta_n,w_n)$ converges to $(\beta,w)$ in $\R\times L^\infty_1 (\Omega)^2$, 
which shows $w\in L^\infty_{1,0} (\Omega)^2$. The proof is complete.
\end{proof}

\vspace{0.3cm}

\noindent 
{\Large {\bf Appendix}}

\vspace{0.3cm}

We will prove the Hardy type inequality in two-dimensional exterior domains,
which is used in the proof of Theorem \ref{thm.nonlinear}.

\vspace{0.3cm}

\noindent {\bf \large Lemma A.1}{\it ~Let $\Omega$ be an exterior domain in $\R^2$. Then it follows that 
\begin{align}
\| \frac{f}{1+|x|} \|_{L^2 (\Omega)} 
\leq C \|\nabla f \|_{L^2 (\Omega)} \, \log \bigg ( e + \frac{\| f \|_{L^\infty_{1} (\Omega)}}{\| \nabla f \|_{L^2 (\Omega)}}  \bigg )\label{est.lem.A.1.1}
\end{align}
for any $f\in \dot{W}^{1,2}_0(\Omega) \cap L^\infty_{1}(\Omega)$. Here $C$ depends only on $\Omega$. In particular, if 
\begin{align*}
e \| \nabla f \|_{L^2 (\Omega)} + \| f \|_{L^\infty_{1} (\Omega)} \leq 1\,,
\end{align*}
then 
\begin{align}
\| \frac{f}{1+|x|} \|_{L^2 (\Omega)} \leq C \|\nabla f \|_{L^2 (\Omega)} \, \big | \log \| \nabla f \|_{L^2 (\Omega)} \big |\,.\label{est.lem.A.1.2}
\end{align} 
}

\begin{proof} 
Take $x_0 \in \R^2\setminus \overline{\Omega}$ and $0 < r_0 <e^{-1} $ so that $B_{r_0} (x_0)\subset \R^2\setminus \overline{\Omega}$. By considering the zero extension of $f$ to $\R^2$, it suffices to show \eqref{est.lem.A.1.1} for $\Omega=\R^2$ and $f \in \dot{W}^{1,2} (\R^2) \cap L^\infty_{1}(\R^2)$ such that $f=0$ in $B_{r_0}(x_0)$.  Fix $R > 2 |x_0|$. By the condition $f(x_0)=0$ and  the mean value theorem in the integral form we have 
\begin{align*}
\frac{|f(x)|}{1+|x|}  
& \leq \frac{|x-x_0|}{1+|x|}\int_0^1  | (\nabla f) (\tau (x-x_0)+x_0) |  \dd \tau\\
& \leq (1+|x_0|) \int_{ \frac{r_0}{ |x-x_0| } }^1 | (\nabla f) (\tau (x-x_0)+x_0) | \dd \tau\,,  
\quad x \in \R^2 \setminus B_{r_0} (x_0)\,,
\end{align*}
which gives
\begin{align}
\| \frac{ f }{1+|x|} \|_{L^2 (\{|x-x_0|\leq R\})} 
& 
\leq (1+|x_0|) \int_{ \frac{r_0}{R} }^1  \tau^{-1} 
\| \nabla f \|_{L^2 (\R^2)} \dd \tau
\nonumber \\
& \leq (1+|x_0|) \big ( | \log R | + | \log r_0 | \big ) \| \nabla f \|_{L^2 (\R^2)} \,. \label{proof.lem.A.1.1}
\end{align}
\noindent 
On the other hand, we have 
\begin{align}
\| \frac{ f }{1+|x|}  \|_{L^2 (\{|x-x_0|\geq R\})} 
& \leq \|\frac{1}{(1+|x|)^2} \|_{L^2 (\{|x|\geq \frac{R}{2} \})} \| f \|_{L^\infty_1 (\R^2)}\nonumber \\
& \leq \frac{C}{ R} \| f \|_{L^\infty_1 (\R^2)}\,.
\label{proof.lem.A.1.2}
\end{align}
\noindent 
If $\| f\|_{L^\infty_1 (\R^2)} \leq 2|x_0| \| \nabla f \|_{L^2 (\R^2)}$ 
then we obtain \eqref{est.lem.A.1.1} from \eqref{proof.lem.A.1.1} and \eqref{proof.lem.A.1.2} with $R=2|x_0| + 1$.
If $\| f\|_{L^\infty_1 (\R^2)} \geq 2|x_0| \| \nabla f \|_{L^2 (\R^2)}$ then we take $R=e+\frac{\| f\|_{L^\infty_1 (\R^2)}}{\| \nabla f \|_{L^2 (\R^2)}}$, which yields again from \eqref{proof.lem.A.1.1} and \eqref{proof.lem.A.1.2} that
\begin{align}
\| \frac{f}{1+|x|}  \|_{L^2 (\R^2)} 
&\leq C|\log r_0| (1+|x_0|)  \| \nabla f \|_{L^2 (\R^2)}  
\log \bigg (e+ \frac{\| f\|_{L^\infty_1 (\R^2)}}{\| \nabla f \|_{L^2 (\R^2)}} \bigg ) \,.
\end{align}

\noindent 
Here we have used $ |\log r_0| \geq 1$ and $ |\log R| \geq 1$, and $C$ is a numerical constant. Thus \eqref{est.lem.A.1.1} holds. The proof is complete. 
\end{proof}

\noindent 
{\bf Acknowledgements} \, The authors would like to express sincere thanks to Professor Toshiaki Hishida for helpful comments on our work. The first author is partially supported by the Grant-in-Aid for JSPS Fellows 17J00636. The second author is partially supported by the Grant-in-Aid for Young Scientists (B) 25800079.

\end{document}